\documentclass{templateArticle}
\definecolor{lime}{HTML}{A6CE39}
\DeclareRobustCommand{\orcidicon}{
	\begin{tikzpicture}
	\draw[lime, fill=lime] (0,0) 
	circle [radius=0.16] 
	node[white] {{\fontfamily{qag}\selectfont \tiny ID}};
	\draw[white, fill=white] (-0.0625,0.095) 
	circle [radius=0.007];
	\end{tikzpicture}
	\hspace{-2mm}
}

\title{Generalised analytical results on $n$-ejection-collision orbits in the RTBP. Analysis of bifurcations}
\author{T. M-Seara \hspace{-2mm}\href{https://orcid.org/0000-0001-8421-8717}{\orcidicon}\hspace{-1.5mm}, 
M. Ollé \hspace{-2mm}\href{https://orcid.org/0000-0002-8050-9055}{\orcidicon}\hspace{-1.5mm}, 
\'O. Rodr\'iguez\footnote{\href{mailto:oscar.rodriguez@upc.edu}{oscar.rodriguez@upc.edu}} \hspace{-1.5mm}\href{https://orcid.org/0000-0002-4545-5135}{\orcidicon}\hspace{-1.5mm}, 
J. Soler \hspace{-2mm}\href{https://orcid.org/0000-0002-6220-5170}{\orcidicon}}

\date{\vspace{-5ex}}

\begin{document}

\maketitle	

{\bf Abstract}
In the planar circular restricted three-body problem and for any value of the mass parameter $\mu \in (0,1)$ and $n\ge 1$, we prove the existence of four families of $n$-ejection-collision ($n$-EC) orbits, that is, orbits where the particle ejects from a primary, reaches $n$ maxima in the distance with respect to it and finally collides with the primary. Such EC orbits have a  value of the Jacobi constant of the form $C=3\mu +Ln^{2/3}(1-\mu)^{2/3}$, where $L>0$ is big enough but independent of $\mu$ and $n$. 
In order to prove this optimal result, we consider Levi-Civita's transformation to regularize
the collision with one primary and a perturbative approach using an ad hoc small parameter once a suitable scale in the configuration plane and time has previously been applied. This result improves a previous work where the existence of the $n$-EC orbits was stated when the mass parameter $\mu>0$ was small enough. In this paper, any possible value of    $\mu\in (0,1)$ and $n\ge 1$ is considered.
Moreover, for decreasing values of $C$, there appear some bifurcations which are first numerically investigated and afterwards explicit expressions for the approximation of the bifurcation values of $C$ are discussed.  
Finally, a detailed analysis of the existence of $n$-EC orbits when $\mu \to 1$ is also described. In a natural way Hill's problem shows up. 
For this problem, we prove an 
analytical result on the existence of four families of $n$-EC orbits and numerically we describe them as well as the appearing bifurcations.

\section{Introduction}

This paper studies the existence of ejection-collision orbits in  the planar circular Restricted three-body problem (RTBP), which describes the motion of a particle (of neglectible mass) under the attraction of two point massive bodies $P_1$ and $P_2$, called primaries, that describe circular orbits around their common center of mass. 
Introducing a rotating system of coordinates that rotates with the primaries, and using suitable units of length, time and mass, an autonomous system of four ODE of first order is derived, depending on a unique parameter $\mu \in (0,1)$, in such a way that the primaries have masses $1-\mu$ and $\mu$ respectively.
Such system of ODE has the well known Jacobi first integral (equal to $C$ along each solution) and is a regular system everywhere except when the particle collides with each of the primaries.

$n$-ejection-collision orbits ($n$-EC orbits from now on) are orbits which eject from a primary and its distance to it reaches $n$  relative maxima before colliding with it (see Definition \ref{defECO}).
Since the $n$-EC orbits are the main target of this paper and  collisions between the particle and one primary lead to singularities in the system of ODE, some kind of regularization, that transforms the original system to a new one which is regular at collisions, is necessary. Among the different possible choices, ranging from local to global regularizations (see \cite{Devaney,erdi,SS,Szebehely}) we will use along the paper the (local) Levi-Civita regularization \cite{LeviCivita1906}, because it is conceptually simple, suitable for our theoretical purposes and efficient for numerical simulations.


The main analytical result of this paper is Theorem \ref{maintheoremN}, where we prove that there exists an $\hat{L}$ such that for $L\geq\hat{L}$ and for any value of $\mu \in (0,1)$, $n\in\mathbb{N}$ and the Jacobi constant $C=3\mu +Ln^{2/3}(1-\mu)^{2/3}$, there exist four $n$-EC orbits, and we characterize them.

This improves a recent result (see \cite{orspaper2}) where the existence of four $n$-EC orbits ejecting (and colliding) from the {\it big} primary (of mass $1-\mu$) is proved but only for small enough $\mu >0$.

To prove this main result we first consider a weaker version in Theorem \ref{th:maintheorem2} where we show that for all $n\in\mathbb{N}$, there exists a  $\hat K(n)$ such that for  $K\ge \hat K (n)$ and for any value of $\mu \in (0,1)$ and  $C=3\mu +K(1-\mu)^{2/3}$, there exist four $n$-EC orbits.
This weaker version also improves the result of \cite{orspaper2} since we cover {\sl any} value of $\mu$ so we can eject from (and collide with) any of the primaries, irrespective of its mass. 
Another improvement is the proof's approach. 
In the previous paper a perturbative approach for small enough $\mu >0$ and big enough $C$ was considered. 
There, the authors computed the series expansion, with respect to the mass parameter $\mu$, of the ejection (collision) manifold. So the explicit analytical expansion, up to certain order,  of this manifold integrated up to a suitable Poincar\'e section $\Sigma$ (maximum distance to the ejecting primary)  was obtained. For suitable number of crossings with $\Sigma$, $i$ for the ejection manifold and $j$ for the collision one (with $i+j=n+1$),  the resulting two curves $C^+_i$ and $C^-_j$ were computed.
Achieving such curves required some technicalities, in particular,  the computation of terms up to order 9 (at least) in such expansions and the expressions of them in the usual polar coordinates (instead of the initial angle $\theta _0$).  
The application of the Implicit Function Theorem (IFT) to analyze the intersection of both curves gave rise to the existence of four $n$-EC orbits for any $n$, $C$ big enough and  $\mu >0$ small enough. 

In this paper, the perturbative approach considers a suitable small parameter, related with the inverse of the Jacobi constant, regardless of the value of $\mu$. 
Moreover, instead of computing the {\sl two} curves $C^+_i$ and $C^-_j$, we consider the angular momentum at the $n$-th passage with the minimum distance to the primary (the particle ejected from). We characterize  an  $n$-EC orbit by the zero value of that angular momentum.
This strategy to use the angular momentum simplifies the computations in three directions: first only expansions up to order 6 are required, second obtaining just one function instead of two different curves, and third the parametrization of the angular momentum directly in terms of $\theta _0$ (thus avoiding the technical issue of the transformation to usual polar coordinates).

The second part of the paper focuses on the bifurcations that may appear when doing the continuation of families of $n$-EC orbits.
It is clear that, given any value of $\mu>0$, and fixed $n$, we can continue the four families of $n$-EC orbits for $C$ big enough, from the IFT. 
According to previous papers (\cite{ors,orspaper2}) we will name such families as $\alpha _n$, $\beta _n$, $\delta _n$ and $\gamma _n$. 
However, as long as $C$ decreases, the IFT does not apply and bifurcations may appear for suitable values of $C$. 
We analyze such bifurcations from the analytical expressions obtained in the series expansions for order higher than 6.
A rich variety of bifurcations  show up. They are discussed and numerically described.

Precisely the results derived from this numerical exploration provides inspiration to obtain the main result: an explicit expression of the bifurcating value of $C$ as $\hat{C}=3\mu+\hat{L}n^{2/3}(1-\mu)^{2/3}$, i.e. we prove that $\hat{K}(n) = \hat{L}n^{2/3}$.

Finally, taking $\mu \to 1$ gives rise to the Hill problem. Quite naturally the same kind of proof developed previously applies to the Hill problem. So as a corollary we
obtain an analytical result that establishes the existence of four families of $n$-EC in this problem. Moreover the existence of the successive bifurcations when decreasing $C$
for all $n\in[1,100]$ are also numerically discussed.

Concerning previous published results on this subject for the circular planar RTBP, we distinguish between analytical and numerical results. Focusing on the theoretical analysis of $n$-EC orbits, only the case for $n=1$ is considered in Llibre \cite{Llibre}, Chenciner and Llibre \cite{CHLL} and Lacomba and Llibre \cite{Lacomba}. 
The general case $n\ge 1$ is studied in \cite{orspaper2}, but for small enough values of $\mu >0$. Regarding a numerical approach, and for $n=1$, we mention the papers by Bozis \cite{bo}, H\'enon \cite{henon1,henon2}, where the authors compute some particular EC orbits that naturally appear when doing the continuation of families of periodic orbits.
For the general case $n\ge 1$, we mention Oll\'e, Rodriguez and Soler papers \cite{ors,cadis}, where the authors compute and analyze the continuation of families of $n$-EC orbits for $n=1,....,25$ and discuss the advantages and disadvantages of Levi-Civita's versus McGehee's \cite{McGehee} regularization. 
Very recently, Oll\'e, Rodriguez and Soler \cite{orspaper3} analyze the global behavior of the whole set of ejection orbits and the dynamical consequences resulting from the interaction between ejection orbits and the Lyapunov periodic orbit around the collinear equilibrium point $L_1$. In particular infinitely many (in a chaotic way) EC orbits show up.  
 

We finally remark that the EC orbits appear quite naturally in astronomical applications. Let us mention that EC orbits allow to explain a mechanism of transfer of mass in binary star systems (see  \cite{Hurley,Mod,Pringle,Witjers}), to describe regions of capture of irregular moons by giant planets (\cite{asetal}) or to discuss temporary capture (\cite{PaezGuzzo}). Other applications include  the probability of crash motion (see \cite{Nagler1,Nagler2}) or the role of ejection orbits to explain a mechanism for ionization in atomic problems (see \cite{bru,olle}).

The paper is organized as  follows:  In Section 2 we recall some basics of the RTBP, we introduce the Levi-Civita coordinates and the new normalized variables that will become useful to  prove the existence of the $n$-EC orbits for any value of $\mu>0$. 
Section 3 recalls the topics described in Section 2 but for the Hill problem.
In Section 4 we state the two main theorems, Theorem 1 and Theorem 2, concerning the existence of $n$-EC orbits in the RTBP.
We provide the analytical proof of Theorem 2 in Section 5.
Section 6 is devoted to numerically analyse the bifurcations of families of $n$-EC orbits in the RTBP. 
In Section 7  we provide the analytical proof of Theorem \ref{maintheoremN}. 
 Section 8 is devoted to the Hill problem.

Finally, we observe that all the numerical computations have been done using double precision and the numerical integration of the systems of ODE rely on an own implemented Runge-Kutta (7)8 integrator with an adaptive step size control described in \cite{DPrince} and a  Taylor method implemented on a robust, fast and accurate software package in \cite{JZ}.

\section{The planar RTBP and the Levi Civita regularization}




As mentioned in the Introduction, we consider the RTPB. In the rotating (synodical) system, 
 the primaries with mass $1-\mu$ and $\mu$, $\mu \in (0,1)$, have  positions  $P_1=(\mu ,0)$ and $P_2=(\mu -1,0)$ respectively, and the period of their motion will be $2\pi$. In such context, the equations of motion for the particle in the rotating system are given by

\begin{equation} \label{eq:equacions3cos}
    \left\{
	\begin{aligned}
	    &\ddot x -2\dot y = \Omega_x(x,y),\\
		&\ddot y +2\dot x = \Omega_y(x,y),
	\end{aligned}
	\right.
\end{equation}
where $\dot{} =d/dt$ and
\begin{equation} \label{eq:oumega}
	\begin{split}    
	    \Omega(x,y)&=\frac{1}{2}(x^2+y^2)+\frac{1-\mu}{\sqrt{(x-\mu)^2+y^2}}+\frac{\mu}{\sqrt{(x-\mu+1)^2+y^2}}+\frac{1}{2}\mu(1-\mu)\\
	    & = \frac{1}{2}\left[(1-\mu)r_1^2+\mu r_2^2\right] + \frac{1-\mu}{r_1} + \frac{\mu}{r_2},
	\end{split}
\end{equation}
with $r_1 = \sqrt{(x-\mu)^2+y^2} $ and $r_2 = \sqrt{(x-\mu+1)^2+y^2} $. So, the equations become singular when $r_1$ or $r_2\rightarrow 0$.

The main properties of this system used later on are the following (see 
\cite{Szebehely} for details):
\begin{enumerate}
	\item There exists a first integral, defined by
	\begin{equation}  \label{constantJacobi}
	    C = 2\Omega(x,y) - \dot{x}^2-\dot{y}^2,
	\end{equation}
    and known as  Jacobi integral.

	\item System \eqref{eq:equacions3cos} has the symmetry
	\begin{equation} \label{simetria}
	    (t, x, y, \dot{x},\dot{y}) \rightarrow (-t, x, -y, -\dot{x},\dot{y}). 
	\end{equation}
	A geometrical interpretation of it is that given an orbit in the configuration space $(x, y)$, the symmetrical orbit with respect to the $x$ axis will also exist.

	\item The simplest solutions are 5 equilibrium points:  the so called collinear ones $L_{i}$, $i=1,2,3$, and the triangular ones $L_{i}$, $i=4,5$. On the plane $(x,y)$, $L_{1,2,3}$  are located on the $x$ axis, with $x_{L_2}<\mu -1<x_{L_1}<\mu <x_{L_3}$ and $L_{4,5}$ forming an equilateral triangle with the primaries.  $C_{L_i}$ will stand for the value of $C$ at $L_i$, $i=1,...,5$.

	\item  Depending on the value of the Jacobi constant $C$, the particle can move on specific regions of the plane $(x,y)$, called Hill regions and defined by
	\begin{equation}
	    \mathcal{R}(C) = \left\{(x, y) \in \mathbb{R}^2\,|\,2\Omega(x,y)\geq C\right\}.
	\end{equation}
\end{enumerate}

In order to deal with the singularity of the primary $P_1=(\mu ,0)$ ($r_1=0$) we will consider the Levi-Civita regularization (see \cite{Szebehely}). The well known  transformation of coordinates and time is given by: 
\[
    \left\{
    \begin{aligned}
        & x = \mu + u^2-v^2,\\
        & y = 2uv,\\
        & \frac{dt}{ds} = 4\left(u^2+v^2\right),
    \end{aligned}
    \right.
\]
and we remark that, taking $\mu \in (0,1)$ we are regularizing the big primary (if $\mu \in (0,1/2]$) or the small one  (if $\mu \in [1/2,1)$). In this new system of coordinates, the solutions of system \eqref{eq:equacions3cos} with Jacobi constant equal to $C$ satisfy:
\begin{equation}\label{LCres}
    \left\{
    \begin{aligned}
        &u'' - 8(u^2+v^2)v' = \left(4\mathcal{U}(u^2+v^2)\right)_u\\
        & \hspace{25mm}= 4\mu u +16\mu u^3 +12(u^2+v^2)^2u + \frac{8\mu u}{r_2} - \frac{8 \mu u (u^2+v^2) (u^2+v^2+1)}{r_2^3}-4Cu,\\
        &v'' + 8(u^2+v^2)u' = \left(4\mathcal{U}(u^2+v^2)\right)_v\\
        & \hspace{25mm}= 4\mu v -16\mu v^3 +12(u^2+v^2)^2v + \frac{8\mu v}{r_2} - \frac{8 \mu v (u^2+v^2) (u^2+v^2-1)}{r_2^3}-4Cv,\\
        & C'=0
    \end{aligned}
    \right.
\end{equation}
where $' = d/ds$ and
\[
     \mathcal{U} = \frac{1}{2}\left[(1-\mu)\left(u^2 + v^2\right)^2+\mu \left((1+u^2-v^2)^2+4u^2v^2\right)\right] + \frac{1-\mu}{u^2 + v^2}+\frac{\mu}{r_2}-\frac{C}{2}.
\]
with $r_2 = \sqrt{(1+u^2-v^2)^2+4u^2v^2}$. 

The system of ODE is now regular everywhere except at the collision with the  primary $P_2$ ($r_2=0$).

We observe that when studying  the system of ODE \eqref{LCres}, a value of a Jacobi constant $C$ is fixed. Thus  to take an initial condition of this system, we will take $(u(0),v(0),u'(0),v'(0),C(0))$. Nevertheless, along the paper, we will actually study system \eqref{LCres} removing the last equation in $C$, and we will consider the corresponding solution for a fixed $C$ and initial condition simply  given by $(u(0),v(0),u'(0),v'(0))$.

In this new system of variables, the previous properties of the RTBP are translated as:
\begin{enumerate}
    \item Jacobi Integral:
    \begin{equation}\label{prjac}
        u'^2+v'^2 = 8\left(u^2+v^2\right)\mathcal{U},
    \end{equation}
    which is regular at the collision with the primary $P_1$. In particular (see \cite{Szebehely}), the velocity at the position of the first primary $(u=0,v=0)$ satisfies:
    \begin{equation} \label{C0}
        u'^2+v'^2 = 8(1-\mu),
    \end{equation}
and therefore the velocities at the collision lie in a circle of radius $\sqrt{8(1-\mu)}$. 
    \begin{figure}[ht!]
        \centering
        \includegraphics[trim= 4mm 0 4mm 1mm, clip, width=0.485\textwidth]{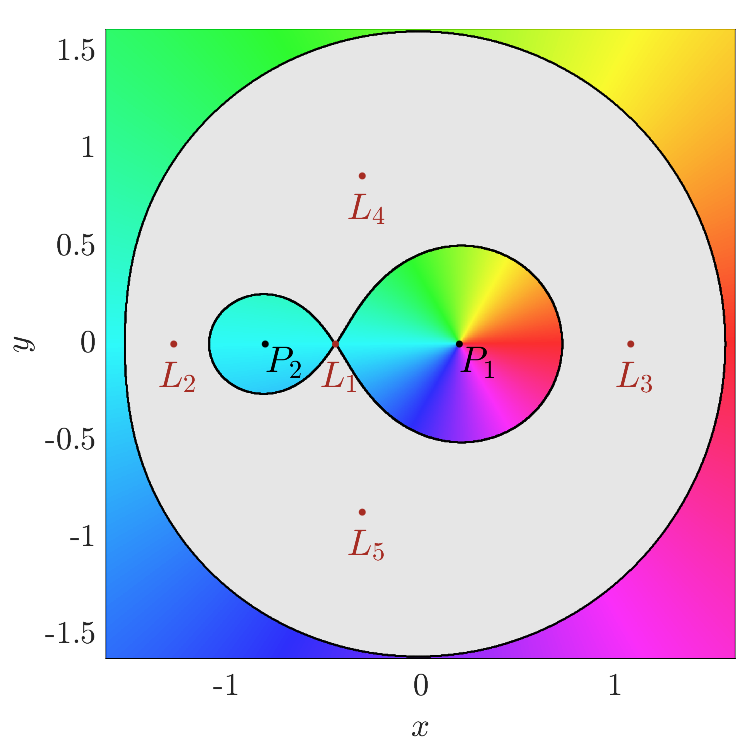}
        \includegraphics[trim= 4mm 0 4mm 1mm, clip, width=0.485\textwidth]{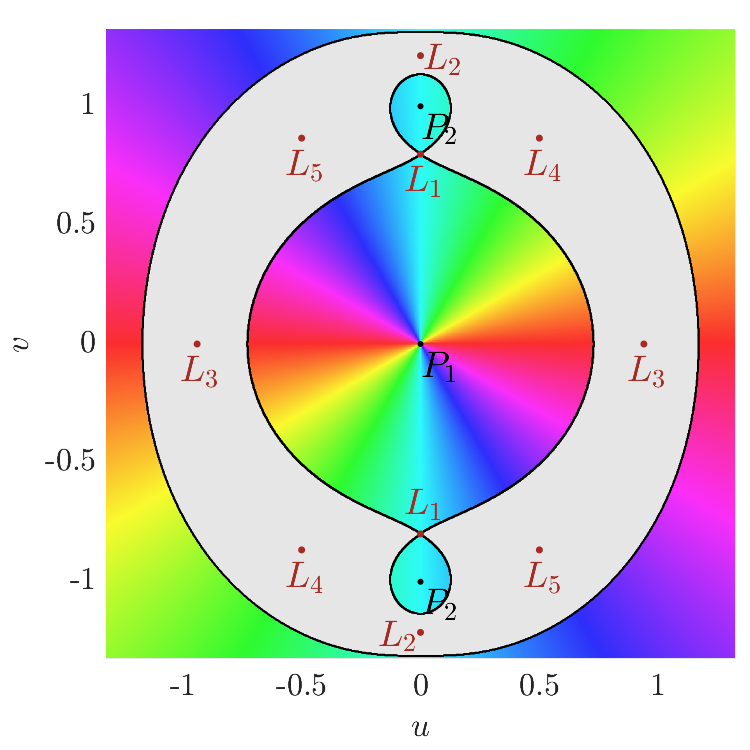}
        \caption{Levi-Civita transformation. Hill's region for $\mu=0.2$ and $C_{L_1}$. Left.  Synodic $(x,y)$ coordinates. Right. Levi-Civita ones $(u,v)$. In grey the forbidden region.}
        \label{fig:Hill}
    \end{figure}
    
    \item As the Levi-Civita transformation duplicates the configuration space (see Figure \ref{fig:Hill}) the equations of motion satisfy two symmetries, \eqref{simetriaLC1} as a  consequence of the duplication of space and \eqref{simetriaLC2} due to \eqref{simetria}:
    \begin{subequations}
        \label{simetriaLC}
        \begin{gather}
            (s,u,v,u',v') \rightarrow (-s,u,-v,-u',v'), \label{simetriaLC1}  \\
            (s,u,v,u',v') \rightarrow (-s,-u,v,u',-v'). \label{simetriaLC2}
        \end{gather}
    \end{subequations}
           
    \item The equilibrium points are now duplicated and they are located on the plane $(u,v)$. In particular, the collinear points now are located in the $u$ axis and in the $v$ axis. See Figure \ref{fig:Hill}.

    \item  Similarly,  given a value of the Jacobi constant $C$,  the Hill’s region in variables $(u,v)$ now becomes
    \begin{equation}\label{hillenuv}
        \mathcal{R}(C) = \left\{(u, v) \in \mathbb{R}^2\,|\,(u^2+v^2)\mathcal{U}\geq0\right\}.
    \end{equation}
    In particular we will consider values of the Jacobi constant $C\geq C_{L_1}$, the value of the Jacobi constant associated to the equilibrium point $L_1$. In this way, it will be enough to regularize only the position of $P_1$ because the Hill's region associated to these values of $C$ avoids collisions with the second primary (assuming the particle moves in a neighbourhood of $P_1$), see Figure \ref{fig:Hill}). 
\end{enumerate}


\section{The Hill problem and the Levi-Civita regularization}\label{sec:Hill}

    The \emph{Hill problem} is a simplified limiting case of the RTBP that allows to study the vicinity of the small primary when this mass tends to 0 (when mass parameter $\mu$ is very small or very close to 1). We can obtain easily the equation of Hill problem making a translation of the small primary (denoted by $P_h$) to the origin, and rescaling the coordinates by a factor $\mu^{1/3}$ if $\mu\rightarrow 0$ or $(1-\mu)^{1/3}$ if $\mu\rightarrow 1$.
    
    For our purpose we will consider this second case, so the first step is to introduce new variables $(x_h,y_h)$ defined by the relation
    \[
        x = \mu + (1-\mu)^{1/3}x_h,\quad\quad y = (1-\mu)^{1/3}y_h.
    \]
    In this way the expression \eqref{eq:oumega} becomes
    \begin{equation}
        \frac{1}{(1-\mu)^{2/3}}\left(\Omega -\frac{3}{2}\right) = \frac{3}{2}x_h^2+\frac{1}{\sqrt{x_h^2+y_h^2}} + \mathcal{O}\left((1-\mu)^{1/3}\right),
    \end{equation}
    and taking the limit $\mu\rightarrow 1$ we obtain the Hill's potential
    \begin{equation}\label{eq:OmegaHill}
        \Psi(x_h,y_h) = \frac{3}{2}x_h^2+\frac{1}{\sqrt{x_h^2+y_h^2}}.
    \end{equation}
    Thus the equations of motion are given by
    \begin{equation} \label{eq:eqHill}
    \left\{
	\begin{aligned}
		&\ddot x_h -2\dot y_h = \Psi_{x_h}(x_h,y_h),\\
		&\ddot y_h +2\dot x_h = \Psi_{y_h}(x_h,y_h).
	\end{aligned}
	\right.
    \end{equation}
    
    The Hill problem also has some interesting properties for our purposes:
    \begin{enumerate}
        \item The system \eqref{eq:eqHill} has a first integral defined by
        \begin{equation} \label{eq:cttHill}
            K = 2 \Psi(x_h,y_h) - \dot{x}_h^2- \dot{y}_h^2,
        \end{equation}
    	where $K$ is related with the Jacobi integral by: 
        \begin{equation}\label{eq:relCiK}
            C = 3\mu + (1-\mu)^{2/3}K + \mathcal{O}(1-\mu).
        \end{equation} 
    
    \item The equations \eqref{eq:eqHill} not only inherit the symmetry of the problem, that is a symmetry with respect to the $x_h$-axis, but also has an extra one with respect to the $y_h$-axis. In this way the system \eqref{eq:eqHill} has the symmetries:
	\begin{subequations}\label{eq:simetriesHill}
        \begin{gather}
            (t,x_h,y_h,\dot{x}_h,\dot{y}_h) \rightarrow (-t,x_h,-y_h,-\dot{x}_h,\dot{y}_h), \label{eq:simetriesHill1}  \\
            (t,x_h,y_h,\dot{x}_h,\dot{y}_h) \rightarrow (-t,-x_h,y_h,\dot{x}_h,-\dot{y}_h). \label{eq:simetriesHill2}
        \end{gather}
    \end{subequations}
    
    \item The Hill problem only preserves two equilibrium points, which are those that are in the vicinity of the small primary $P_h$.
    That is $L_1$ and $L_2$ if we consider $\mu\rightarrow 0$ or $ L_1$ and $L_3$ if  $\mu\rightarrow 1$. 
    For historical consistency, we will call these equilibrium points $L_1$ and $L_2$, which have positions $(\pm1/3^{1/3},0)$ (see Figure \ref{fig:ProblemaHill}) and we will denote by $K_L=3^{4/3}$ the value of $K$ at $L_1$ and $L_2$.
    
    \item In a similar way, from the first integral and taking into account that $2 \Psi(x_h,y_h)-K\ge 0$, given a value of $K$, the motion can only take place in the Hill’s region defined by
    \begin{equation}
    \mathcal{R}_h(K) = \left\{(x_h, y_h) \in \mathbb{R}^2\,|\,2\Psi(x_h,y_h)\geq K\right\}.
    \end{equation}
We notice that, similarly as we do in the RTBP, we will consider values of $K\ge K_L$ to guarantee
that if the particle starts in a region around $P_h$, it will always remain there. 
    
    \end{enumerate}

In order to regularize the Hill problem we only have to consider the Levi-Civita regularization
\[
    \left\{
    \begin{aligned}
    &x_h = u_h^2-v_h^2,\\
    &y_h = 2u_hv_h,\\
    &\frac{dt}{ds} = 4\left(u_h^2+v_h^2\right),
    \end{aligned}
    \right.
\]
and the system \eqref{eq:eqHill} becomes:
\begin{equation}\label{eq:LChill}
	\left\{
	\begin{aligned}
		u_h'' - 8(u_h^2+v_h^2)v_h' &= \left(4\mathcal{U}_h(u_h^2+v_h^2)\right)_{u_h}\\
		&=-4Ku_h + 12\left(2(u_h^4-2u_h^2v_h^2-v_h^4)+(u_h^2+v_h^2)^2\right)u_h,\\[1.2ex]
		v_h'' + 8(u_h^2+v_h^2)u_h' &= \left(4\mathcal{U}_h(u_h^2+v_h^2)\right)_{v_h}\\
		&=-4Kv_h + 12\left(2(v_h^4-2u_h^2u_h^2-u_h^4)+(u_h^2+v_h^2)^2\right)v_h,
		\\
	\end{aligned}
	\right.
\end{equation}
with 
\begin{equation}
    \mathcal{U}_h = \frac{3(u_h^2 - v_h^2)^2}{2} + \frac{1}{u_h^2 + v_h^2} - \frac{K}{2}.
\end{equation}

\begin{figure}[ht!]
    \centering
    \includegraphics[trim= 0mm 0 0mm 0mm, clip, width=0.45\textwidth]{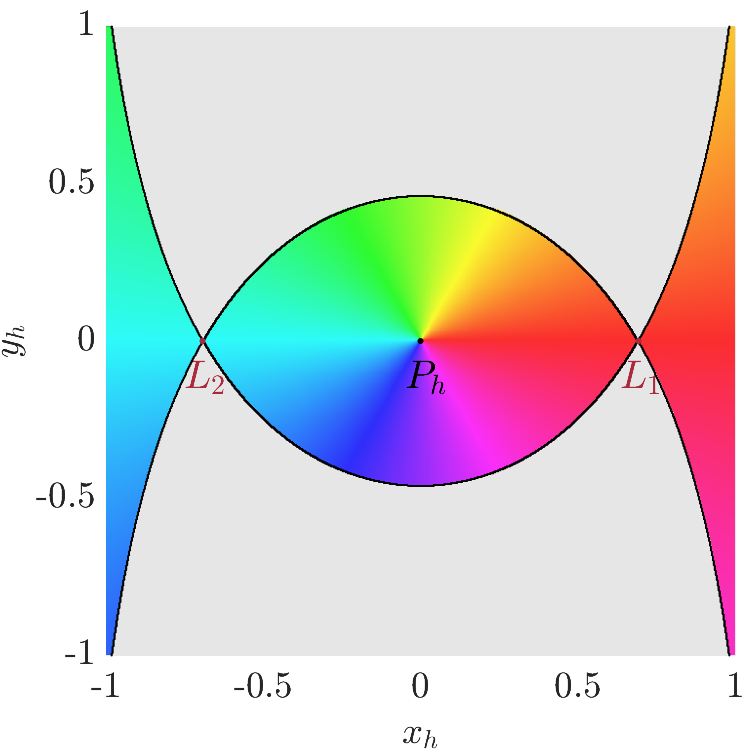}
    \includegraphics[trim= 0mm 0 0mm 0mm, clip, width=0.45\textwidth]{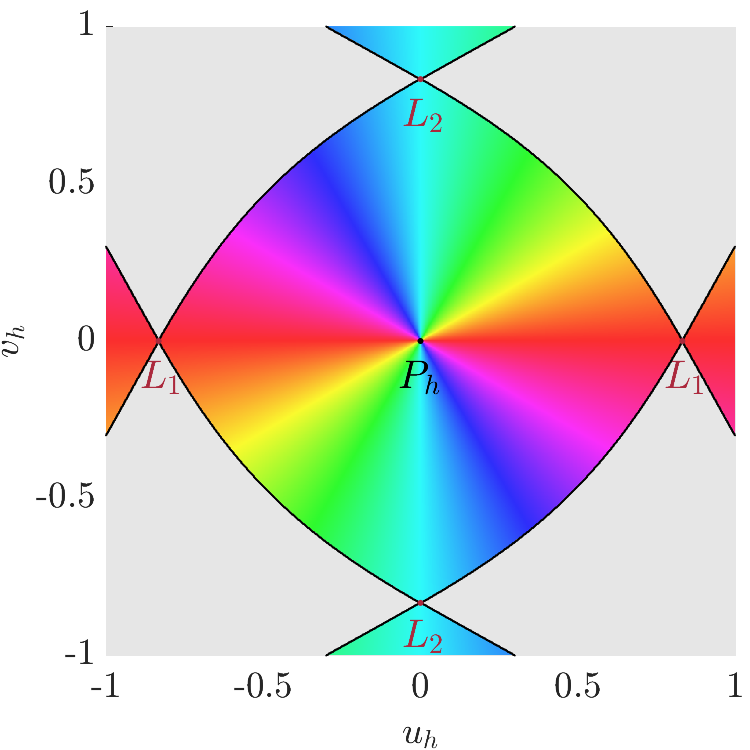}
    \caption{Levi-Civita transformation. Hill's region for $K = K_{L}= 3^{4/3}$. Left.  Synodic $(x_h,y_h)$ coordinates. Right. Levi-Civita ones $(u_h,v_h)$. In grey the forbidden region.}
    \label{fig:ProblemaHill}
\end{figure}

Under this transformation the previous properties of the Hill problem are translated as:

\begin{enumerate}
    \item The first integral of \eqref{eq:LChill} is given by
    \begin{equation}
        u_h'^2+v_h'^2 = 8(u_h^2+v_h^2)\mathcal{U}_h,
    \end{equation}
    which is regular at the collision with $P_h$. In particular the velocity at the position of the primary $(u_h=0,v_h=0)$ satisfies
    \begin{equation}
        u_h'^2+v_h'^2 = 8.
    \end{equation}

    \item As the transformation duplicates the configuration space (see Figure \ref{fig:ProblemaHill}), the equation \eqref{eq:LChill} has an extra symmetry:
    \begin{subequations}\label{eq:simetriesHillLC}
        \begin{gather}
            (s,u_h,v_h,u'_h,v'_h) \rightarrow (-s,u_h,-v_h,-u'_h,v'_h), \label{eq:simetriesHillLC1}  \\
            (s,u_h,v_h,u'_h,v'_h) \rightarrow (-s,-u_h,v_h,u'_h,-v'_h), \label{eq:simetriesHillLC2}  \\
            (s,u_h,v_h,u'_h,v'_h) \rightarrow (-s,v_h,u_h,-v'_h,-u'_h). \label{eq:simetriesHillLC3}
        \end{gather}
    \end{subequations}
    
    \item For the same reason, the equilibrium points are duplicated, and we have $L_1 = (\pm3^{-1/6},0)$ and $L_2 = (0,\pm3^{-1/6})$.

    \item In a similar way, depending on the value of $K$ we can define the valid region of motion (see Figure \ref{fig:ProblemaHill}) in the plane $(u_h, v_h)$ as:
    \begin{equation}
	    \mathcal{R}(K) = \left\{(u_h, v_h) \in \mathbb{R}^2\,|\,(u_h^2+v_h^2)\mathcal{U}_h\geq0\right\}.
	\end{equation}
\end{enumerate}

\section{$n$-EC orbits in the RTBP and main theorems}

In this paper we will focus on a specific type of EC orbits, the $n$-EC orbits, formally defined as
\begin{definition}\label{defECO}
    We call $n$-ejection-collision orbit of a primary, simply noted by $n$-EC orbit, to the orbit that the particle describes when ejects from a primary and reaches $n$ times a relative maximum in the distance with respect to this primary before colliding with it.  
\end{definition}

As we will consider any value of $\mu\in(0,1)$ we will study only the $n$-EC orbits associated to the first primary $P_1$. 
 Notice that from relation \eqref{C0} it is easy to compute the initial conditions of the ejection orbits (and the collision orbits):
\begin{equation}\label{eq:nECOs_ini}
(0, 0, 2\sqrt{2(1 -\mu)}\cos\theta_0, 2\sqrt{2(1 -\mu)}\sin\theta_0),\quad\quad \theta_0\in[0,2\pi) 
\end{equation}
and we can compute the manifold of the ejection (collision) orbits integrating forward (backward) in time. 
Observe that in this case it is enough to consider a value of $\theta_0\in[0,\pi)$ due to the duplication of the configuration plane.

Concerning the existence of 
$n$-EC orbits, we mentioned above that 
in \cite{orspaper2}, the existence of four $n$-EC orbits ejecting from (and colliding with) the big primary for any $n\ge 1$, given  $C$ big enough and   $\mu>0$ small enough, was proved.
The proof was based on a perturbative approach in $\mu$ and assuming that the orbits ejected from the big primary of mass $1-\mu$.

The first goal of this paper is to improve this previous result and prove the existence of four $n$-EC orbits ejecting from (and colliding with)  the big or small primary, for any $n\ge 1$ given and $C$ big enough. So {\sl any value of the mass parameter $\mu \in (0,1)$} is possible in this context.

For analytical and numerical purposes, though, we will use a characterization for an EC orbit, based upon the zero value of its angular momentum, defined from now on as $M:=U\dot V-V\dot U$ (for some suitable variables $(U,V)$ to be defined later), at a minimum distance with the primary the particle ejected from (see Lemma \ref{caraceco} below). 
So in order to obtain an $n$-EC orbit, for $n\ge 1$, $\mu \in(0,1)$ and $C$ given, first we will compute the corresponding ejection solution for each initial condition (that is, for each value $\theta _0$). 
Second we will determine the precise time $\tau ^* = \tau ^*(\theta_0)$ 
when the particle reaches  the $n$-th minimum in the distance to $P_1$. 
At time $\tau^*$  we will compute the value of the angular momentum that is, $(U\dot{V} -V\dot{U})(\tau^*)$.
Varying $\theta _0\in [0,\pi )$ we will obtain the corresponding angular momentum,
that will denote by $M(n,\theta _0)=(U\dot{V} -V\dot{U})(\tau^*)$ (overlooking the additional dependence on $\mu$). 
The zeros of  $M(n,\theta _0)=0$ will provide us with the precise values of $\theta _0$ such that the corresponding ejection orbit is precisely an $n$-EC orbit.
Just to show this idea, we plot in Figure \ref{fig:exMomAng} left the behaviour of the angular momentum $M(1,\theta _0)$ for  $\mu =0.1$ and $C=5$. 
In the right figure we plot the corresponding ejection orbits for three chosen values of  $\theta _0$: the red one and green one plotted for  a range of time $[0,\tau ^* + \delta ]$ (a small suitable $\delta >0$) to see the change of sign in the angular momentum (shown in the zoom area) and the blue one which is a $1$-EC orbit.

\begin{figure}[ht!]
    \centering
    \includegraphics[width=0.475\textwidth]{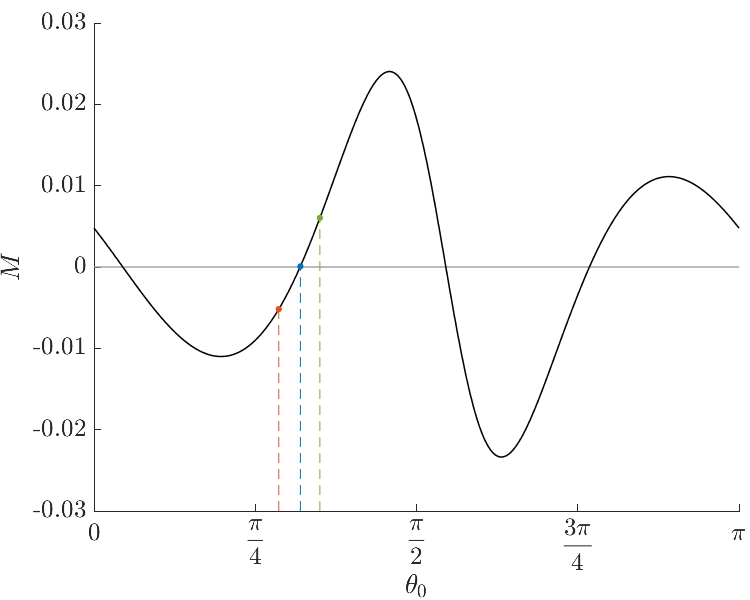}
    \includegraphics[width=0.475\textwidth]{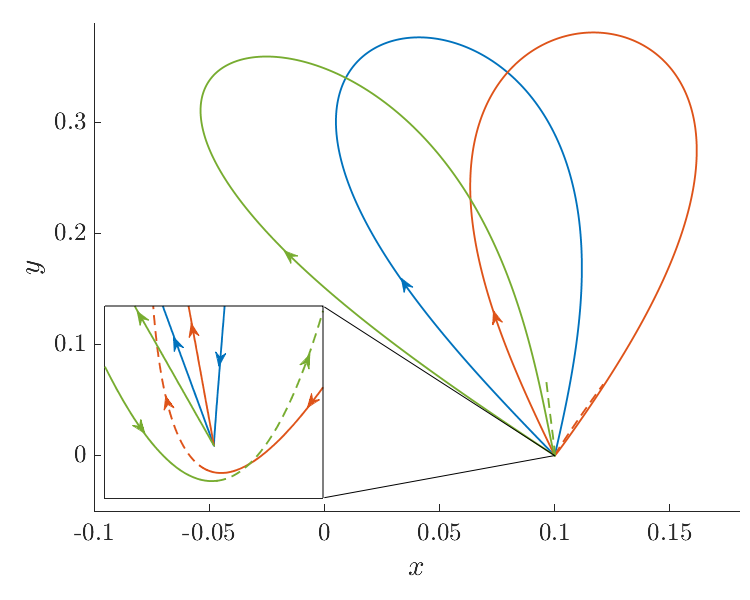}
    \caption{$n=1$, $\mu=0.1$ and $C=5$. Left. Angular momentum  $M(1,\theta _0)$. Right. Three ejection orbits corresponding to the initial values of $\theta _0$ labelled in colours on the left plot. The blue orbit is precisely a $1$-EC orbit.}
    \label{fig:exMomAng}
\end{figure}
Now, we proceed to state the main result of this paper about the existence, the number and the characteristics of the $n$-ejection-collision orbits for any value of the mass parameter and $n\in \mathbb{N}$, for sufficiently restricted Hill regions (i.e. $C$ big enough).
\begin{theorem}\label{maintheoremN}
	There exists an $\hat{L}$ such that for $L\geq\hat{L}$ and for any value of $\mu \in (0,1)$, $n\in\mathbb{N}$ and $C=3\mu +Ln^{2/3}(1-\mu)^{2/3}$, there exist four $n$-EC orbits, which can be characterized by:
	\begin{itemize}
		\item Two $n$-EC orbits symmetric with respect to the $x$  axis.
		\vspace{1mm}
		\item Two $n$-EC orbits, one  symmetric of the other with respect to the $x$ axis.
	\end{itemize}
	The respective families are labelled by 	$\alpha_n$, $\gamma_n$, $\beta_n$ and $\delta_n$ (corresponding to increasing values of $\theta _0$ respectively).
\end{theorem}
In order to prove Theorem \ref{maintheoremN} we will first state a  weaker version of this theorem, Theorem \ref{th:maintheorem2}, but the proof of this second version will provide light on the approach, mainly a suitable scaling in the configuration variables, time and the Jacobi constant $C$, used to prove the more optimal result  in Theorem \ref{maintheoremN}.
\begin{theorem}\label{th:maintheorem2}
    For all $n\in\mathbb{N}$, there exists a  $\hat K(n)$ such that for  $K\ge \hat K (n)$ and for any value of $\mu \in (0,1)$ and  $C=3\mu +K(1-\mu)^{2/3}$, there exist four $n$-EC orbits, which can be characterized in the same way as in Theorem \ref{maintheoremN}.
\end{theorem}
We remark that, in Theorem \ref{th:maintheorem2}, we have a uniform  constant $K=K(n)$ for any value of $\mu \in (0,1)$. This implies that when  $\mu \rightarrow 1$ the value of the Jacobi constant  (for which Theorem \ref{th:maintheorem2} holds) tends to 3, as $C_{L_1}$ does as well. 
Precisely, and as shown in the proof of Theorem  \ref{th:maintheorem2}, the expansion of $C_{L_1}$ was the inspiration to choose a suitable scaling in the variables, time and $C$.
Finally in Theorem \ref{maintheoremN} an expression for $K(n)$ as $Ln^{2/3}$ is provided.

\section{Proof of Theorem \ref{th:maintheorem2}}

In order to prove Theorem \ref{th:maintheorem2}, let us fix $C\geq C_{L_1}$ and consider the following change of variables and time:
\begin{equation}\label{cvarnouUV}
    \left\{
    \begin{aligned}
    	u & = \sqrt{\frac{2(1-\mu)}{C-3\mu}}U,\\
    	v & = \sqrt{\frac{2(1-\mu)}{C-3\mu}}V,\\
    	\tau &= 2\sqrt{C-3\mu} s,
	\end{aligned}
    \right.
\end{equation}
that corresponds to the change that normalizes the linear term of \eqref{LCres} and the initial condition of the ejection orbits.
Denoting by $\dot{} = \frac{d}{d\tau}$ the new time derivative the system \eqref{LCres} transforms to  the following:
\begin{equation}\label{ode_UV}
    \left\{
    \begin{aligned}
	    \ddot{U} & =  -\frac{(C - \mu)U}{C - 3\mu}
	        +\frac{8(1-\mu)\left(U^2+V^2\right)\dot{V}}{(C-3\mu)^{3/2} }
	        +\frac{12{\left(1-\mu\right)}^2{\left(U^2+V^2\right)}^2U}{{\left(C-3\mu \right)}^3}
	        +\frac{8\mu(1-\mu)U^3}{\left(C-3\mu \right)^2}\\
	        & \hspace{4mm}+\frac{2\mu U}{(C-3\mu)R_2}  - \frac{4\mu(1-\mu)U(U^2+V^2)[2(1-\mu)(U^2+V^2)+(C-3\mu)]}{(C-3\mu)^3R_2^3},\\[1.5ex]
	    \ddot{V} & =  -\frac{(C - \mu)V}{C - 3\mu}
	        -\frac{8(1-\mu)\left(U^2+V^2\right)\dot{U}}{(C-3\mu)^{3/2} }
	        +\frac{12{\left(1-\mu\right)}^2{\left(U^2+V^2\right)}^2V}{{\left(C-3\mu \right)}^3}
	        -\frac{8\mu(1-\mu)V^3}{\left(C-3\mu \right)^2}\\
	        & \hspace{4mm}+\frac{2\mu V}{(C-3\mu)R_2}  - \frac{4\mu(1-\mu)V(U^2+V^2)[2(1-\mu)(U^2+V^2)-(C-3\mu)]}{(C-3\mu)^3R_2^3},
	\end{aligned}
    \right.
\end{equation}
where $R_2 =\sqrt{1 + \frac{4(1-\mu)(U^2-V^2)}{C-3\mu} +\frac{4(1-\mu)^2(U^2+V^2)^2}{(C-3\mu)^2}}$. 

It is important to remark that the  properties \eqref{prjac}, \eqref{C0},  \eqref{hillenuv} are preserved (translated to the new variables), and so are the symmetries obtained in the Levi-Civita regularization, i. e.:
\begin{subequations} \label{simetriaNewLC}
    \begin{gather}
        (\tau,U,V,\dot{U},\dot{V}) \rightarrow (\tau,-U,-V,-\dot{U},-\dot{V}), \label{simetriaNewLC1}  \\
        (\tau,U,V,\dot{U},\dot{V}) \rightarrow (-\tau,-U,V,\dot{U},-\dot{V}).  \label{simetriaNewLC2}
    \end{gather}
\end{subequations}

At this point, the two main ideas to prove the theorem are:
 (i) a  perturvative approach taking $\delta = 1/\sqrt{C-3\mu}$  as a small parameter, and (ii) the requirement of the angular momentum to be zero at a minimum distance with the primary the particle ejected from.

First of all we observe that the functions $1/R_2$ and $1/R_2^3$ are are analytic for $U$, $V$ bounded, $0\le \mu\le 1$, and $\delta$ small enough. In fact, the expansions of $1/R_2$ and $1/R_2^3$ are of the form:
\begin{equation}
	\begin{aligned}
	\frac{1}{R_2} &= 
	    1 
	    - 2(1-\mu)(U^2-V^2)\delta^2
	    + 8(1-\mu)^2(11(U^2-V^2)^2-4U^2V^2)\delta^4
	    + \sum_{k\geq3}(1-\mu)^kP_{2k}(U,V)\delta^{2k},\\[1.7ex]
	\frac{1}{R_2^3} &=
	    1 
	    - 6(1-\mu)(U^2-V^2)\delta^2 + \sum_{k\geq2}(1-\mu)^kQ_{2k}(U,V)\delta^{2k},
	\end{aligned}
\end{equation}
where $P_{2k}(U,V)$ and $Q_{2k}(U,V)$ are polynomials sum of monomials of degree $2k$.

So if we expand the system \eqref{ode_UV} with respect to $\delta$  we obtain:
\begin{equation}\label{SistemaSerieDelta}
    \left\{
	\begin{aligned}
	\ddot{U} &= - U + 8(1-\mu)(U^2 + V^2)\dot{V}\delta^3 + 12(1-\mu)^2\left[2\mu\left(U^4-2U^2V^2-V^4\right) + (U^2+V^2)^2\right]U   \delta^6\\
	&\hspace{5mm} + \mu\sum_{k\geq4}(1-\mu)^{k-1}\bar{P}_{2k-1}(U,V)\delta^{2k},\\[1.7ex]
	\ddot{V} &= - V - 8(1-\mu)(U^2 + V^2)\dot{U}\delta^3 + 12(1-\mu)^2\left[2\mu\left(V^4-2U^2V^2-U^4\right) + (U^2+V^2)^2\right]V   \delta^6\\
	&\hspace{5mm} + \mu\sum_{k\geq4}(1-\mu)^{k-1}\bar{Q}_{2k-1}(U,V)\delta^{2k},\\[1.7ex]
	\end{aligned}
	\right.
\end{equation}
which is an analytical system of ODE in $\delta$ and $\bar{P}_{2k-1}(U,V)$ and $\bar{Q}_{2k-1}(U,V)$ are polynomials sum of monomials of degree $2k-1$.

Before proceeding it is important to make two observations:
\begin{enumerate}
    \item 
    We can introduce the parameter $\varepsilon = (1-\mu)^{1/3}\delta$. So we have, using that $\delta=\frac{1}{C-3\mu}$:
    \begin{equation}\label{eq:eps2}
        \varepsilon^2 = (1-\mu)^{2/3}\delta^2 = \frac{(1-\mu)^{2/3}}{C-3\mu}.
    \end{equation}
    \item 
    We also know that $C\geq C_{L_1}(\mu)$ since otherwise the Hill region of motion allows transits between both primaries and, in this sense, Hill's region is not regular anymore. As it is well known the expansion of $C_{L_1}(\mu)$ is (see \cite{Szebehely})
    \begin{equation}
        C_{L_1}(\mu) = 3 +9\left(\frac{1-\mu}{3}\right)^{2/3} -7\frac{1-\mu}{3} + \mathcal{O}((1-\mu)^{4/3}),
    \end{equation}
    therefore, if it is possible we will like to have a uniform parameter $K$ with the same order in $(1-\mu)$. So, introducing the variable $K$ as 
    \begin{equation}\label{eq:C}
        C = 3\mu + K(1-\mu)^{2/3},
    \end{equation}
    we have that the previous expression \eqref{eq:eps2} becomes:
    \[
    \varepsilon^2 
    = \frac{1}{K}.
    \]
\end{enumerate}

The change \eqref{cvarnouUV} using \eqref{eq:C} becomes:
\begin{equation}\label{cvarnouUVC}
    \left\{
    \begin{aligned}
    	u & = \frac{\sqrt{2}(1-\mu)^{1/6}}{\sqrt{K}}U,\\
    	v & = \frac{\sqrt{2}(1-\mu)^{1/6}}{\sqrt{K}}V,\\
    	\tau &= 2\sqrt{K}(1-\mu)^{1/3} s,\\
    	C & = 3\mu + K(1-\mu)^{2/3}.
	\end{aligned}
    \right.
\end{equation}

Note that the value $K$ is related with the Hill constant by \eqref{eq:relCiK} when $\mu$ tends to 1.

So, in terms of $\varepsilon = 1/\sqrt{K}$ the system \eqref{SistemaSerieDelta} has the following expression:
\begin{equation} \label{SistemaSerie}
	\left\{
	\begin{aligned}
	    \ddot{U} &= - U + 8(U^2 + V^2)\dot{V}\varepsilon^3 + 12\left[2\mu\left(U^4-2U^2V^2-V^4\right) + (U^2+V^2)^2\right]U   \varepsilon^6\\
	        &\hspace{5mm} + \mu\sum_{k\geq4}(1-\mu)^{\frac{k-3}{3}}\bar{P}_{2k-1}(U,V)\varepsilon^{2k},\\[1.7ex]
	    \ddot{V} &= - V - 8(U^2 + V^2)\dot{U}\varepsilon^3 + 12\left[2\mu\left(V^4-2U^2V^2-U^4\right) + (U^2+V^2)^2\right]V   \varepsilon^6\\
	        &\hspace{5mm} + \mu\sum_{k\geq4}(1-\mu)^{\frac{k-3}{3}}\bar{Q}_{2k-1}(U,V)\varepsilon^{2k}.\\[1.7ex]
	\end{aligned}
	\right.
\end{equation}

Second let us prove the following characterization for an EC orbit, based upon the zero value of the angular momentum at a minimum distance with the primary.
\begin{lemma}\label{caraceco}
    Assume $C$ large enough. An ejection orbit is an EC orbit if and only if  it satisfies that at  a minimum in the distance (with the primary) the angular momentum $M=U\dot V-V\dot U=0$. 
\end{lemma}

\begin{proof}
    The minimum distance condition is given by:
    \begin{equation}\label{mindist}
        \begin{aligned}
            &U\dot{U} + V\dot{V} = 0,\\
            &U\ddot{U} + \dot{U}^2 + V\ddot{V} + \dot{V}^2 > 0,
        \end{aligned}
    \end{equation}
    and the angular momentum condition $M=U\dot V-V\dot U=0$:
    \begin{equation}\label{maval0}
        U\dot{V} = V\dot{U}.
    \end{equation}
    We will distinguish between two cases:
    \begin{enumerate}
        \item  $\dot{V} \neq 0$. Then, from \eqref{maval0}:
        \[
        U = \frac{V\dot{U}}{\dot{V}}\ \mbox{ and by \eqref{mindist}}\ \Longrightarrow \frac{V\dot{U}}{\dot{V}}\dot{U} + V\dot{V} = 0 \Longrightarrow V\dot{U}^2 + V\dot{V}^2 =  V(\dot{U}^2 + \dot{V}^2)= 0 \Longrightarrow V = 0 ,
        \]
        and, by \eqref{maval0}  also $U=0$.

        \item $\dot{V} = 0$, we will have two subcases:
        \begin{enumerate}
            \item if $\dot U \ne 0$, then  by  \eqref{mindist} and  \eqref{maval0} we get $U=V=0$.
            \item $\dot U=0$  then, using equations \eqref{SistemaSerie}:
            \[ U\ddot{U}+  V\ddot{V} =-(U^2+V^2)\left[1 +\mathcal {O}\left(\varepsilon ^6(|U|^4+|V|^4)\right)\right],\]
            but this quantity is negative for $\varepsilon$ small enough, if $U^2+V^2>0$, which contradicts the second item of \eqref{mindist}. We conclude that $U=V=0$.
        \end{enumerate}
    \end{enumerate}
    On the other hand, it is clear that if a collision takes place, i. e.  $U=V=0$ and $\dot U^2+\dot V^2=\sqrt{8(1-\mu)}$, then conditions \eqref{mindist} and \eqref{maval0} are trivially satisfied.
\end{proof}

Now let us proceed. Using the vectorial notation $\bm{U} =(U,V,\dot U,\dot V)^T$,  the second order system of ODE \eqref{SistemaSerie} can be written as 
\begin{equation}\label{eqU}
        \dot{\bm{U}} = \bm{G}(\bm{U})
            =\bm{G}_0(\bm{U}) +  \varepsilon^3\bm{G}_3(\bm{U}) + \varepsilon^6\bm{G}_6(U,V) + \sum_{k\geq4}\varepsilon^{2k}G_{2k}(U,V), 
\end{equation}
where
\begin{equation}
    \begin{matrix}
        \bm{G}_0(\bm{U}) =
        \left(
        \begin{matrix}
            \dot U\\
            \dot V\\
            -U\\
            -V
        \end{matrix}
        \right),
        \quad\quad
        \bm{G}_3(\bm{U}) =
        8
        \left(
        \begin{matrix}
            0\\
            0\\
            (U^2+V^2)\dot{V}\\
            -(U^2+V^2)\dot{U}
        \end{matrix}
        \right),\\[6ex]
        \bm{G}_6(U,V) = 12
        \left(
        \begin{matrix}
            0\\
            0\\
            \left[2\mu\left(U^4-2U^2V^2-V^4\right) + (U^2+V^2)^2\right]U\\
            \left[2\mu\left(V^4-2U^2V^2-U^4\right) + (U^2+V^2)^2\right]V
        \end{matrix}
        \right),\\[6ex]
        \bm{G}_{2k}(U,V) = \mu(1-\mu)^{\frac{k-3}{3}}
        \left(
        \begin{matrix}
            0\\
            0\\
            \bar{P}_{2k-1}(U,V)\\
            \bar{Q}_{2k-1}(U,V)
        \end{matrix}
        \right),
        \text{ for }k\geq 4.
    \end{matrix}
\end{equation}

We remark that  $\bm{G}_0$ and $\bm{G}_3$ are the only functions that depend on $\dot U$ and $\dot V$, the remaining ones depending only on $U$ and $V$. Moreover we observe that the expressions appearing in the expansions are polynomials. Both properties allow to significantly simplify the computations.

The  next natural step consists on obtaining a solution  $\bm{U}=\bm{U}(\tau )$ as a series expansion in $\varepsilon$:
\begin{equation}
    \bm{U} = \sum_{j\ge 0}\bm{U}_{j} \varepsilon^j.
\end{equation}
As a usual procedure to obtain the functions  $\bm{U}_j$, we plug $\bm{U}$ in system \eqref{eqU}, and comparing the powers in $ \varepsilon$, we obtain a system of ODE for  $\bm{U}_j$.

\subsection*{Computation of the functions $\bm{U}_{j}$} 

Now we proceed to compute the explicit expressions for $\bm{U}_{j}(\tau)=(U_{j}(\tau),V_{j}(\tau),\dot U_{j}(\tau),
\dot V_{j}(\tau))$, for any $j$.
Actually we will show that, in order to prove Theorem \ref{th:maintheorem2}, we only need to find explicitly the functions $\bm{U}_j$ up to order $j=6$. 

From Definition \ref{defECO} and the scaling \eqref{cvarnouUVC}, any ejection orbit $\bm{U}(\tau)=\bm{U}(\tau,\theta_0)$, has the initial condition
\begin{equation}\label{condini}
    \bm{U}(0)=(0,0,\cos\theta_0,\sin\theta_0),\quad \theta_0\in[0,2\pi),
\end{equation}
 so we have
\begin{equation}\label{cidef}
    \bm{U}_{0}(0)=(0,0,\cos\theta_0,\sin\theta_0),\qquad \bm{U}_{j}(0)=\bm{0},  \qquad j\ge 1.
\end{equation}

\subsection*{Solution for $\varepsilon = 0$:}

We must solve the linear  system:
\begin{equation}
\left\{
    \begin{aligned}
        \ddot{U}_0 = -U_0,\\
        \ddot{V}_0 = -V_0,
    \end{aligned}
    \right.
\end{equation}
which is a harmonic oscillator, with initial condition \eqref{condini}. 
Then the ejection orbit $\bm{U}_0$ is given by
$\bm{U}_0=(U_0,V_0,\dot{U}_0,\dot{V}_0)$, with:
\begin{equation}
    \begin{aligned}
        &U_{0}(\tau) = \cos\theta_0\sin\tau,\\
        &V_{0}(\tau) = \sin\theta_0\sin\tau,\\
        &\dot{U}_{0}(\tau) = \cos\theta_0\cos\tau,\\
        &\dot{V}_{0}(\tau) = \sin\theta_0\cos\tau.\\
    \end{aligned}
\end{equation}

\subsection*{Solution for $\varepsilon \neq0$:}

In order to find the functions $\bm{U}_{j}$, we must solve the successive resulting ODE when substituting $\bm{U}$ by the series expansion in \eqref{SistemaSerie} up to the desired order.

We observe that, for $j\ge 1$, the linear non homogeneous system of  ODE to be solved is
\[\dfrac{d\bm{U}_{j}}{d\tau}=D \bm{G}_0(\bm{U}_{0})\bm{U}_{j}+\bm{F}_j(\bm{U}_{0},
\bm{U}_{1},...,\bm{U}_{j-3})
=\bm{G}_0(\bm{U}_{j})+\bm{F}_j(\bm{U}_{0},
\bm{U}_{1},...,\bm{U}_{j-3}),\]
where the homogeneous system is always the same but the independent term changes and increases in complexity with $j$.

Since a fundamental matrix for the homogeneous system (the first order variational equations) is given by
\begin{equation}
    X(\tau) =
    \left(
    \begin{matrix}
        \cos\tau & 0 & \sin\tau & 0 \\[1.1ex]
        0 & \cos\tau & 0 & \sin\tau \\[1.1ex]
        -\sin\tau & 0 & \cos\tau & 0\\[1.1ex]
        0 & -\sin\tau & 0 & \cos\tau\\
    \end{matrix}
    \right),
\end{equation}
and the initial conditions are $\bm{U}_j(0)=\bm{0}$  for $j\ge 1$, we obtain the following well known formula
\begin{equation}
    \bm{U}_{j}(\tau) = X(\tau)\int_0^\tau X^{-1}(s)\bm{F}_j(\bm{U}_{0}(s),...,\bm{U}_{j-3}(s))ds .
\end{equation}

The corresponding explicit expressions are the following:
\begin{equation}
	\begin{aligned}
        U_3(\tau) &= (\tau\sin\tau - \cos\tau\sin\tau)\sin\theta_0,\\[1.5ex]
	    V_3(\tau) &= -(\tau\sin\tau - \cos\tau\sin\tau)\cos\theta_0,\\[1.5ex]
	    U_6(\tau) & = -\frac{(\tau-\cos\tau\sin\tau)^2\sin\tau - \mu( 15\tau\cos\tau - (8 + 9\cos^2\tau - 2\cos^4\tau)\sin\tau)(1-2\cos^4\theta_0)}{2}\cos\theta_0,\\[1.5ex]
	    V_6(\tau) & = -\frac{(\tau-\cos\tau\sin\tau)^2\sin\tau - \mu( 15\tau\cos\tau - (8 + 9\cos^2\tau - 2\cos^4\tau)\sin\tau)(1-2\sin^4\theta_0)}{2}\sin\theta_0,\\[1.5ex]
	\end{aligned}
\end{equation}
with $U_i(\tau) = V_i(\tau) = 0$ for $i=1,2,4,5(,7)$. 
Once we have the ejection solution up to order $j=6$, the next step consists of computing the $n$-th minimum in the distance to the primary (located at the origin) the particle ejected from as a function of the initial $\theta _0$. Equivalently we want to compute the $n$-th minimum of the function $\left(U^2+V^2\right)(\tau)$. This requires to compute the precise time denoted by $\tau ^*$, needed to reach the $n$-th minimum in distance. We apply the Implicit Function Theorem to the function $(U\dot{U} +V\dot{V})(\tau^*) = 0$ in order to obtain an expansion series in $\varepsilon$, i.e.:
\[
    \tau^* = \sum_{i=0}^{6}\tau^*_i\varepsilon^i + \mathcal{O}(\varepsilon^{7}).
\]

We can easily compute $\tau^*_0$, since we have a harmonic oscillator:
\[
    \tau^*_0 = n\pi.
\]

Writing the function $(U\dot{U}+V\dot{V})(\tau)$  as an expansion series in $\varepsilon$  and collecting terms of the same order, we can successively find  the terms $\tau_i^*$ (up to order 6, higher order terms in Appendix \ref{app:valorsUs}):
\begin{equation}
	\tau_6^*(n) = \frac{15\mu n\pi(1+ 3\cos(4\theta_0))}{8},
\end{equation}
with $\tau_i(n,\theta_0) = 0$ for $i=1,2,3,4,5(,7)$.

Now we are ready to compute the angular momentum $M(n,\theta _0)=(U\dot{V}-V\dot{U})(\tau ^*)$ whose expansion is:
\begin{equation}\label{Maordre7}
    M(n,\theta_0) = \mu \varepsilon^6 \left(-\frac{15 n\pi\sin(4\theta_0)}{4} + \mathcal{O}\left(\varepsilon^{2}\right)\right),
\end{equation}
or in short, since we look for the zeros of $M (n,\theta _0) =0$, we write, dividing the previous equation by $\mu \varepsilon^6$,
\begin{equation}\label{eqppal}
    \hat{M}(n,\theta_0) = -\frac{15 n\pi\sin(4\theta_0)}{4} + \mathcal{O}\left(\varepsilon^2\right).
\end{equation}

Now we apply the Implicit Function Theorem  and for $\varepsilon >0$ small enough we obtain that the equation \eqref{eqppal} has four and only four roots in $[0,\pi)$ given by
\begin{equation}\label{zerosTFI}
    \theta_0 =\dfrac{\pi m}{4}+\mathcal{O} (\varepsilon^2),\quad m=0,1,2,3.
\end{equation}
regardless of the value of the parameter $\mu$. It is clear from \eqref{eqppal} that the roots $\theta _0$ are simple.

So we have proved that there exist four $n$-EC orbits. Moreover, applying the symmetries of the system we can conclude that those EC orbits with an intersection angle with $m=0,2$ correspond to symmetric  $n$-EC orbits (in the sense that the $(x,y)$ projection is symmetric with respect to the $x$ axis). 
Such EC orbits belong to families $\gamma_n$ and $\alpha_n$ respectively when varying $\varepsilon$ (or, equivalently, $K$ and therefore $C$ in \eqref{eq:C}). 
Those EC orbits with an intersection angle with $m=1,3$ correspond to symmetric $n$-EC orbits (in the sense that the $(x,y)$ projection is symmetric one with respect to the other one). Such EC orbits belong to families $\delta_n$ and $\beta_n$ respectively.

This finishes the proof of Theorem \ref{th:maintheorem2}.

In order to illustrate the results of Theorem \ref{th:maintheorem2}, in Figure \ref{fig:ex4ECOs} top we plot the function $M(n,\theta_0)$ for $\mu =0.1$, $C=6$ and the values of $n=2$ (continuous line) and $n=4$ (discontinuous line). We remark its sinusoidal behaviour in accordance with equation \eqref{eqppal}. Consistently with Theorem \ref{th:maintheorem2}, the curve $M(n,\theta _0)$ intersects four times $M(n,\theta _0)=0$. The four specific values of $\theta _0$ give rise to four $n$-EC orbits belonging to families $\alpha _n$, $\beta _n$, $\gamma _n$ and $\delta _n$. The corresponding EC orbits are shown in the bottom figure in usual synodical coordinates $(x,y)$.
 
\begin{figure}[ht!]
    \centering
    \includegraphics[width=0.95\textwidth]{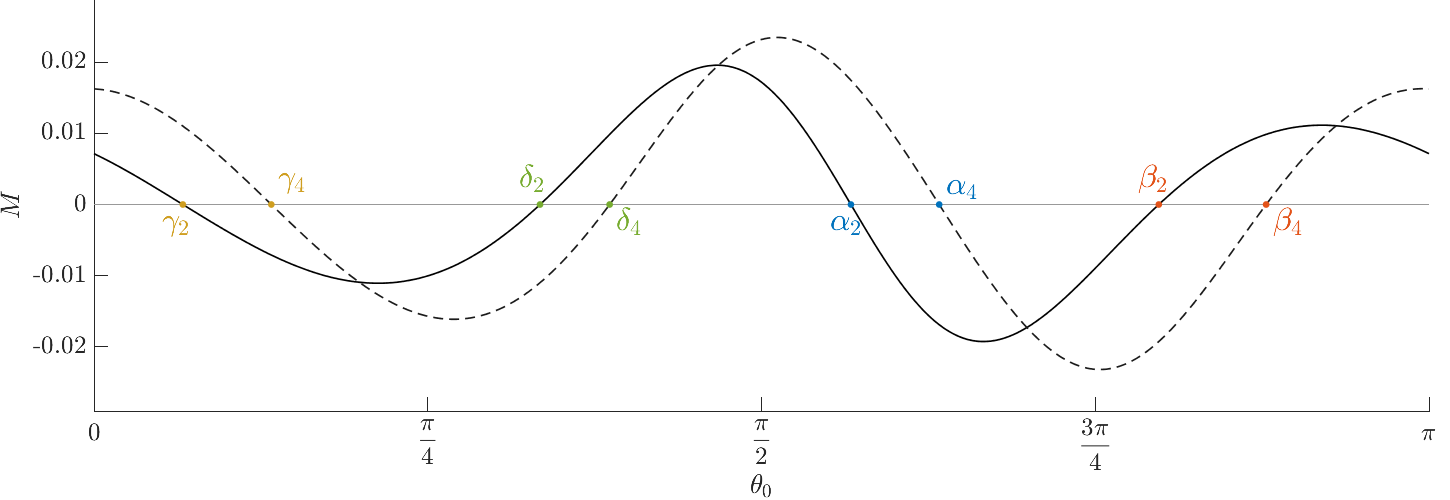}
    \includegraphics[width=0.475\textwidth]{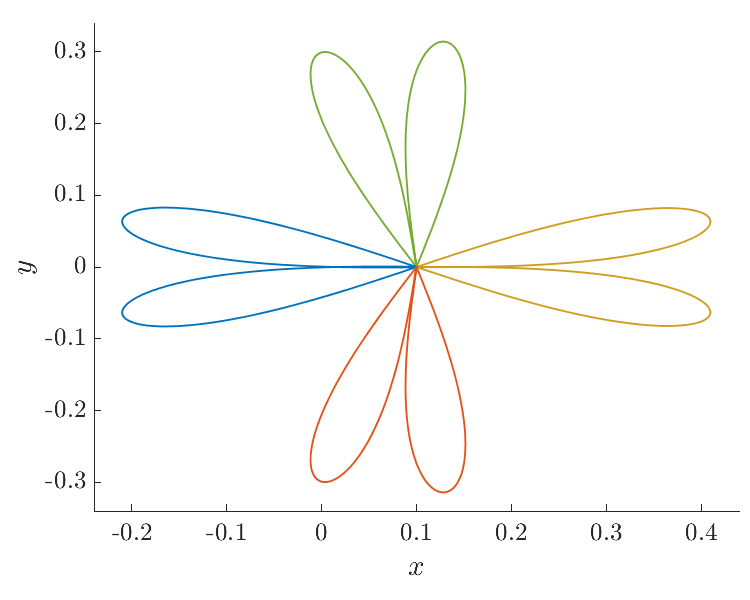}
    \includegraphics[width=0.475\textwidth]{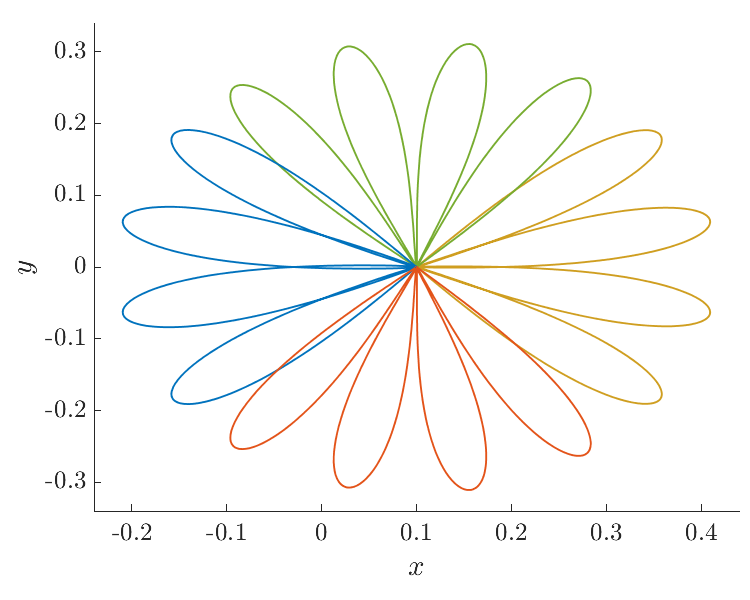}
    \caption{$\mu=0.1$, $C=6$. Top. Angular momentum $M(n,\theta _0)$ for $n=2$ (continuous line) and $n=4$ (discontinuous line). Bottom. The corresponding four $n$-EC orbits in the plane $(x,y)$ (left for $n=2$ and right for $n=4$).}
    \label{fig:ex4ECOs}
\end{figure}

\section{Analysis of Bifurcations}\label{sec_bif}

So far we have applied the Implicit Function Theorem to infer the existence of four and only four $n$-EC orbits, for any value of $\mu$ and $C=3\mu +K(1-\mu)^{1/3}$ (see \eqref{eq:C}) with $K$ big enough, that is $\varepsilon =1/\sqrt{K}$ small enough. 
In this  procedure  the minimum order required  in the $\varepsilon$ expansions for both the functions  $\bm{U}_{j}$ and $\tau_j^*$ was order 6. 
Of course, when $\varepsilon$ becomes bigger, the Implicit Function Theorem may not be applied anymore and bifurcations can appear.  This section is focused on such bifurcations.

We will focus on two purposes: on the one hand, the illustration of the appearance and collapsing of bifurcating families of $n$-EC orbits when doing the continuation of families varying $C$ as parameter; 
and on the other hand, the behavior of $\hat{K}(n)$ and its associated value $\hat C(\mu, n)=3\mu +\hat K(n)(1-\mu)^{1/3}$ provided by Theorem \ref{th:maintheorem2}, for any value of $\mu \in (0,1)$ and varying $n$.

\subsection{Bifurcating families}

The first task is to compute the angular momentum $M (n, \theta _0)$ to higher order.
To do so we need higher order terms for both the functions  $\bm{U}_{j}$ and  $\tau_j^*$. We have proceeded as in the previous Section; however,  there, expressions up to order $6$ were enough. To analyze the bifurcations, we provide their expressions for $j$ up to order 10 in the Appendix.
Now we are ready to compute the explicit expression for the angular momentum $M(n,\theta _0)=(U\dot{V}-V\dot{U})(\tau ^*)$ up to order 10 which is the following:

\begin{equation}\label{Maordre10}
    \begin{aligned}
	    M(n,\theta_0) &= -\frac{15\mu n\pi\sin(4\theta_0)}{4}\varepsilon^6 + \frac{105\mu(1-\mu)^{1/3}n\pi\left(\sin(2\theta_0) + 5\sin(6\theta_0)\right)}{64}\varepsilon^8 +\frac{15\mu n^2\pi^2\cos(4\theta_0)}{2}\varepsilon^9\\
	    &\hspace{6mm}- \frac{315\mu(1-\mu)^{2/3}n\pi(2\sin(4\theta_0) + 7\sin(8\theta_0))}{128}\varepsilon^{10} +\mathcal{O}(\varepsilon^{11}).
    \end{aligned}
\end{equation}

It is clear that if $\varepsilon$ is small enough, the dominant term is $\varepsilon ^6$, and the zeros of $M(n,\theta _0)$ are related to the term $\sin (4\theta _0)$. Therefore we obtain four $n$-EC orbits.

\begin{figure}[ht!]
    \centering
    \includegraphics[width=0.95\textwidth]{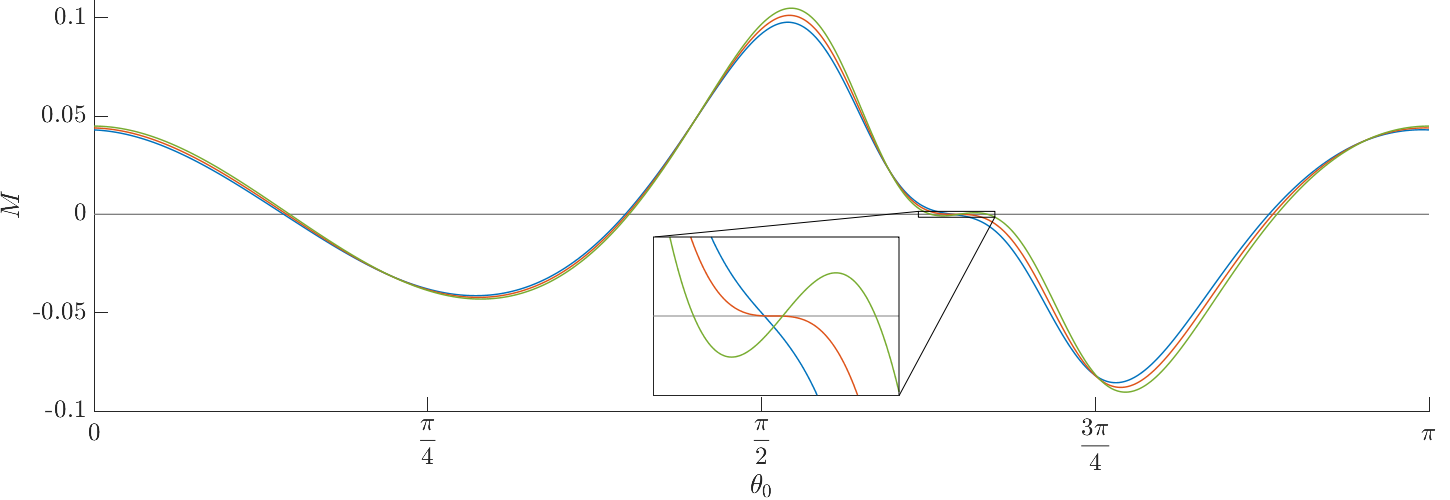}
    \includegraphics[width=0.32\textwidth]{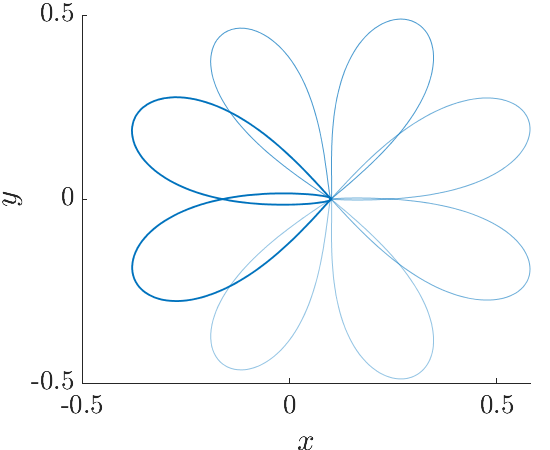}
    \includegraphics[width=0.32\textwidth]{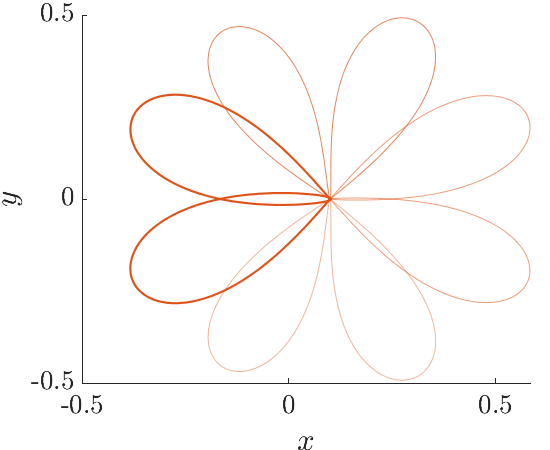}
    \includegraphics[width=0.32\textwidth]{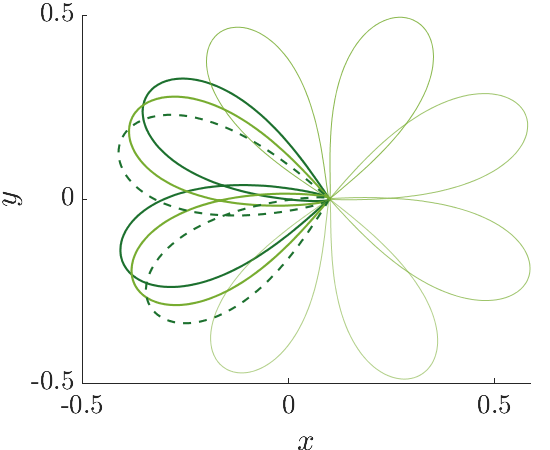}
    \caption{$\mu=0.1$, Top. We plot the angular momentum $M(2,\theta _0)$. Notice the zoom area where the appearance of two new bifurcating orbits (in green), besides the family $\alpha _2$ is observed when decreasing $C$. Bottom. Left, middle and right. Four $2$-EC orbits (the colour code corresponds to the top figure) for  $C = 3.76$ (in blue), $C_{bif} = 3.72442505$ (the bifurcating value, in red),   $C = 3.69$ (in green).  Darker colour: those  $2$-EC orbits belonging to family $\alpha _2$.  In the right plot, also the two new bifurcated $2$-EC orbits are plotted.}
    \label{fig:bif2ECOs}
\end{figure}

\begin{figure}[ht!]
    \centering
    \includegraphics[width=0.95\textwidth]{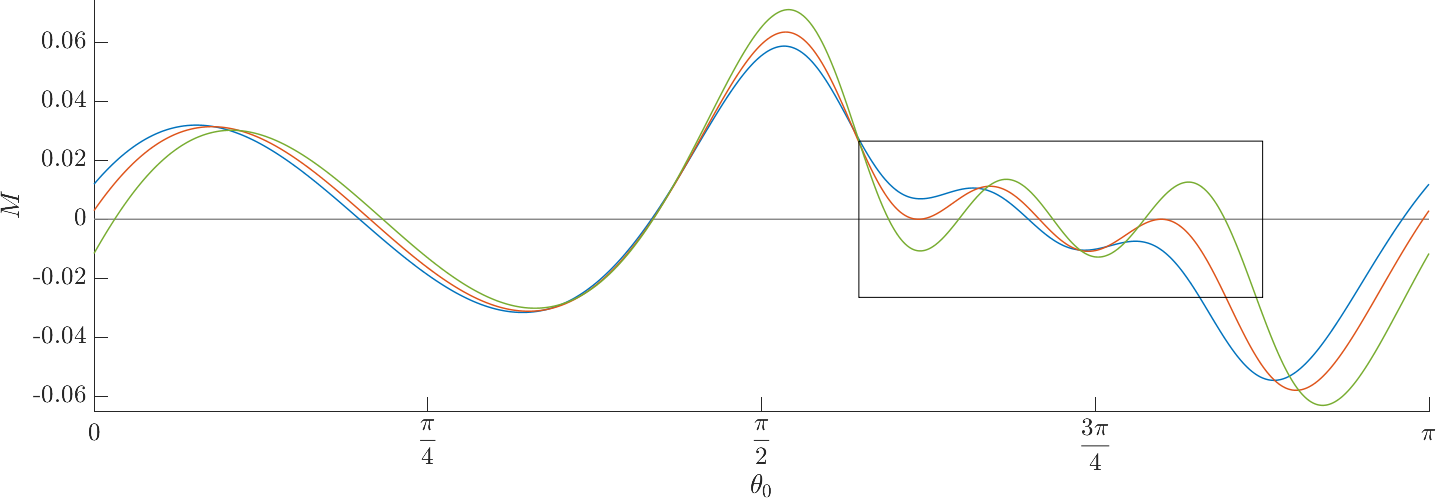}
    \includegraphics[width=0.32\textwidth]{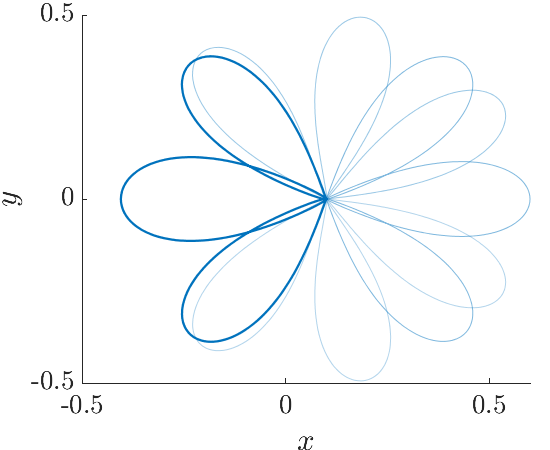}
    \includegraphics[width=0.32\textwidth]{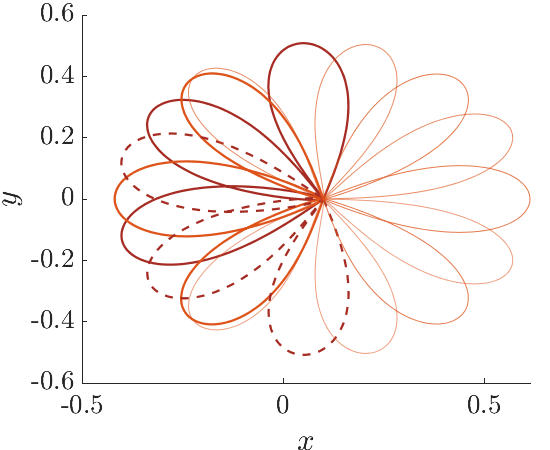}
    \includegraphics[width=0.32\textwidth]{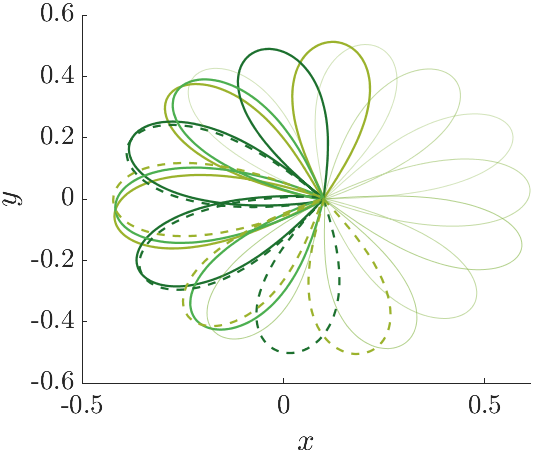}
    \caption{$\mu=0.1$, $n=3$. Top. We plot the angular momentum $M(n,\theta _0)$. Notice the zoom area where the appearance of four new bifurcating orbits (in green), besides the family $\alpha _3$ is observed when decreasing $C$. Bottom. Left, middle and right. Four $3$-EC orbits (the colour code corresponds to the top figure) for  $C = 3.9$ (in blue), $C_{bif} = 3.80644009$ (the bifurcating value, in red), $C = 3.7$ (in green). Darker colour: those  $3$-EC orbits belonging to family $\alpha _3$. In the middle plot, also the two tangent new bifurcated $3$-EC orbits are plotted. In the right plot, also the four new bifurcated $3$-EC orbits are plotted.}
    \label{fig:bif3ECOs}
\end{figure}

However let us discuss what happens for bigger values of $\varepsilon$, or equivalently for smaller values of $C$.  
We will illustrate two different kind of bifurcations that take place when doing the continuation of families of $n$-EC orbits and that can be explained precisely from the analytical expression of $M(n,\theta _0)$ to higher order just obtained.

The first kind of bifurcation can be inferred just taking into account the terms of $M(n,\theta _0)$ up to order 8 in \eqref{Maordre10}. The bifurcation is associated with the term $\sin (6\theta _0)$. See Figure \ref{fig:bif2ECOs} top for $\mu =0.1$ and $n=2$. 
We can clearly see how increasing $\varepsilon$ (decreasing $C$), the bifurcation takes place. 
Let us describe the bifurcation close to the $2$-EC orbit belonging to family $\alpha _2$.
See the zoom area in Figure \ref{fig:bif2ECOs} top. 
Locally, at a neighbourhood of the value of $\theta _0$ of such EC orbit, for some value of $C$ the angular momentum has a unique transversal intersection with the $x$-axis (that is $M(2,\theta _0)=0$,  $M'(2,\theta _0)=0$). 
For $C=3.76$ this intersection corresponds to the $2$-EC orbit belonging to the family $\alpha _2$ (see the blue curve). 
For the bifurcating value $C_{bif}=3.72442505$,  $M(2, \theta _0)$ crosses tangently the $x$ axis (see the red curve).
For smaller values of $C$, $M(2, \theta _0)$ crosses the $x$ axis three times, giving rise to two new bifurcating families of $2$-EC orbits (see the green curve) besides family $\alpha _2$ which persists. 
The new $2$-EC orbits are (obviously due to symmetry  \eqref{simetria}) one symmetric with respect to the other. 
From a global point of view, for a range $C<C_{bif}$,  varying $\theta _0\in [0,\pi )$, $M(2, \theta _0)$ crosses six times,  that is, we obtain six $2$-EC orbits, and this is related to the term $\sin (6\theta _0)$, which becomes the dominant term in $M(2, \theta _0)$. 
We show these $2$-EC orbits in Figure \ref{fig:bif2ECOs} bottom. More precisely, on the three plots, the four $2$-EC orbits are shown (in the plane $(x,y)$) and those $2$-EC orbits of family $\alpha _2$ are plotted in a darker colour. Since the family $\alpha _2$ persists after the bifurcation, the $2$-EC orbits are plotted in the left, middle and right plots.  The two new bifurcating $2$-EC orbits after the bifurcation are also shown on the right plot.

The second kind of bifurcation can be inferred from the expression of $M(n, \theta _0)$ up to order 10 given in \eqref{Maordre10}. The bifurcation is associated with the term $\sin (8\theta _0)$. See Figure \ref{fig:bif3ECOs} top for $\mu =0.1$ and $n=3$.
We can clearly see how increasing $\varepsilon$ (decreasing $C$), the angular momentum $M(3, \theta _0)$ typically crosses four times the $x$-axis (for $\theta _0\in [0,\pi )$), as expected (see the blue curve in the top figure).
However at some bifurcating value $C_{bif}$ there appear two tangencies (say from nowhere, see the red curve in
the zoom area in Figure  \ref{fig:bif3ECOs} top); each tangency  gives rise to two families
when doing the continuation of families decreasing $C$. See the green curve in the zoom area in  Figure  \ref{fig:bif3ECOs} top. 
So from a global point of view, for a range of $C<C_{bif}$ and $\theta _0\in [0,\pi )$,  the angular momentum $M(3, \theta _0)=0$ crosses eight times the $x$-axis, giving rise to eight $3$-EC orbits related to the term  $\sin (8\theta _0)$.
We show these $3$-EC orbits in Figure \ref{fig:bif3ECOs} bottom. More specifically, on the three plots, the four $3$-EC orbits are shown (in the plane $(x,y)$) and those $3$-EC orbits of family $\alpha _3$ are plotted in a darker colour. 
The two $3$-EC orbits that appear due to the tangency of $M(3, \theta _0)$ with the $x$-axis are also plotted on the  middle  plot.
Moreover, the four new bifurcating $3$-EC orbits after the bifurcation are also shown on the right plot.
A continuous and discontinuous line with the same colour correspond to EC orbits that are symmetric one with respect to the other one. In fact, due to the symmetry of the problem, we might only consider the two intersection points  (those on the left hand side or on the right one of the value of $\theta_0$ in $\alpha_3$), and the other two intersection points would be obtained by symmetry.

\begin{figure}[ht!]
    \centering
    \includegraphics[width=0.44\textwidth]{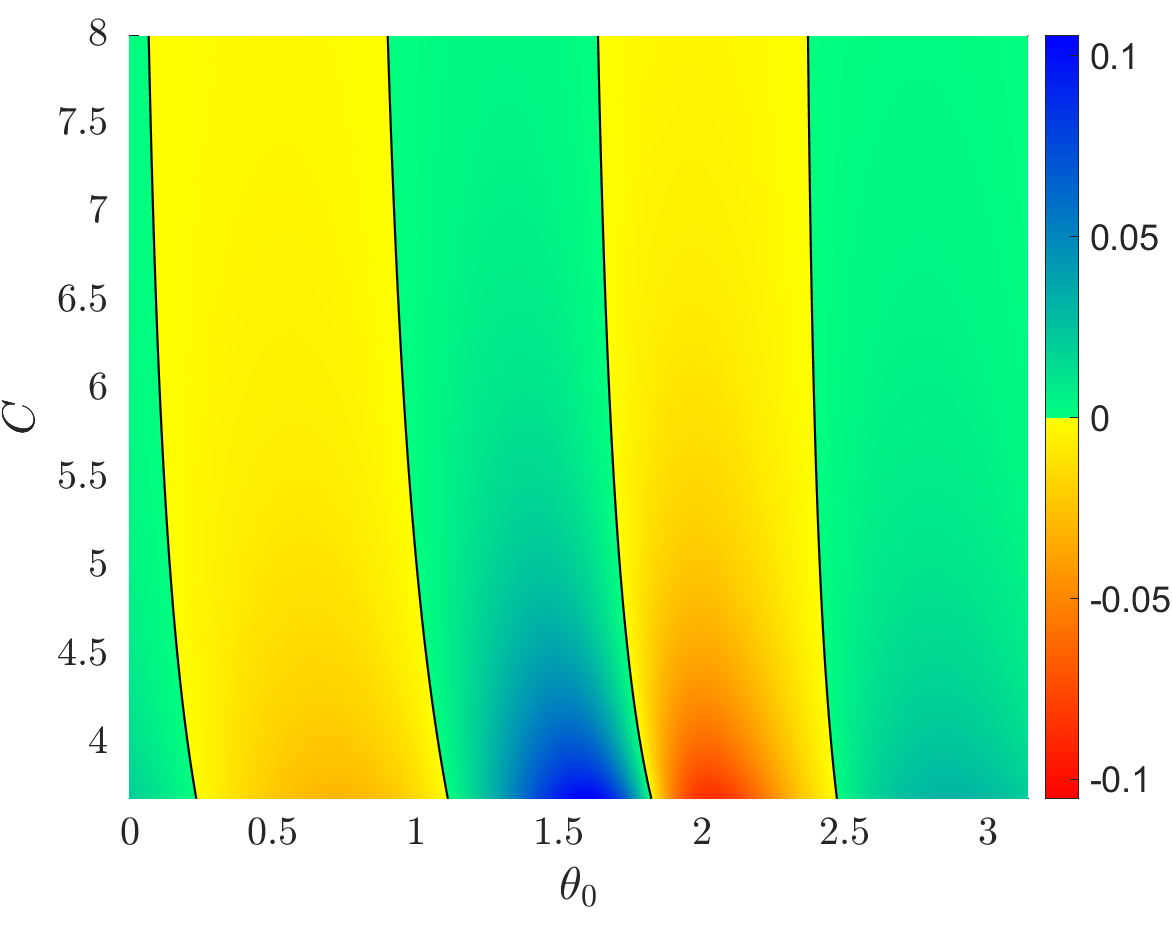}
    \includegraphics[width=0.44\textwidth]{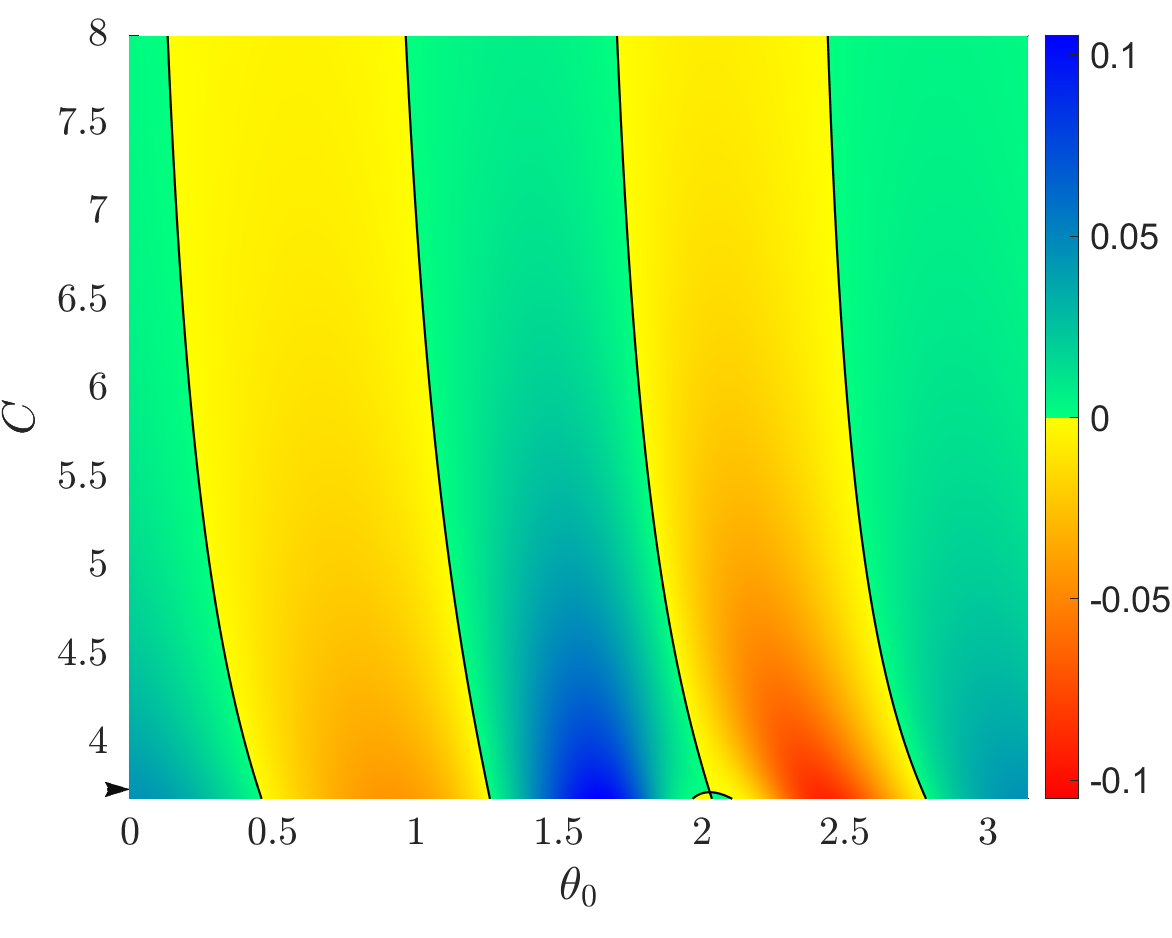}\\
    \vspace{-2mm}
    \includegraphics[width=0.44\textwidth]{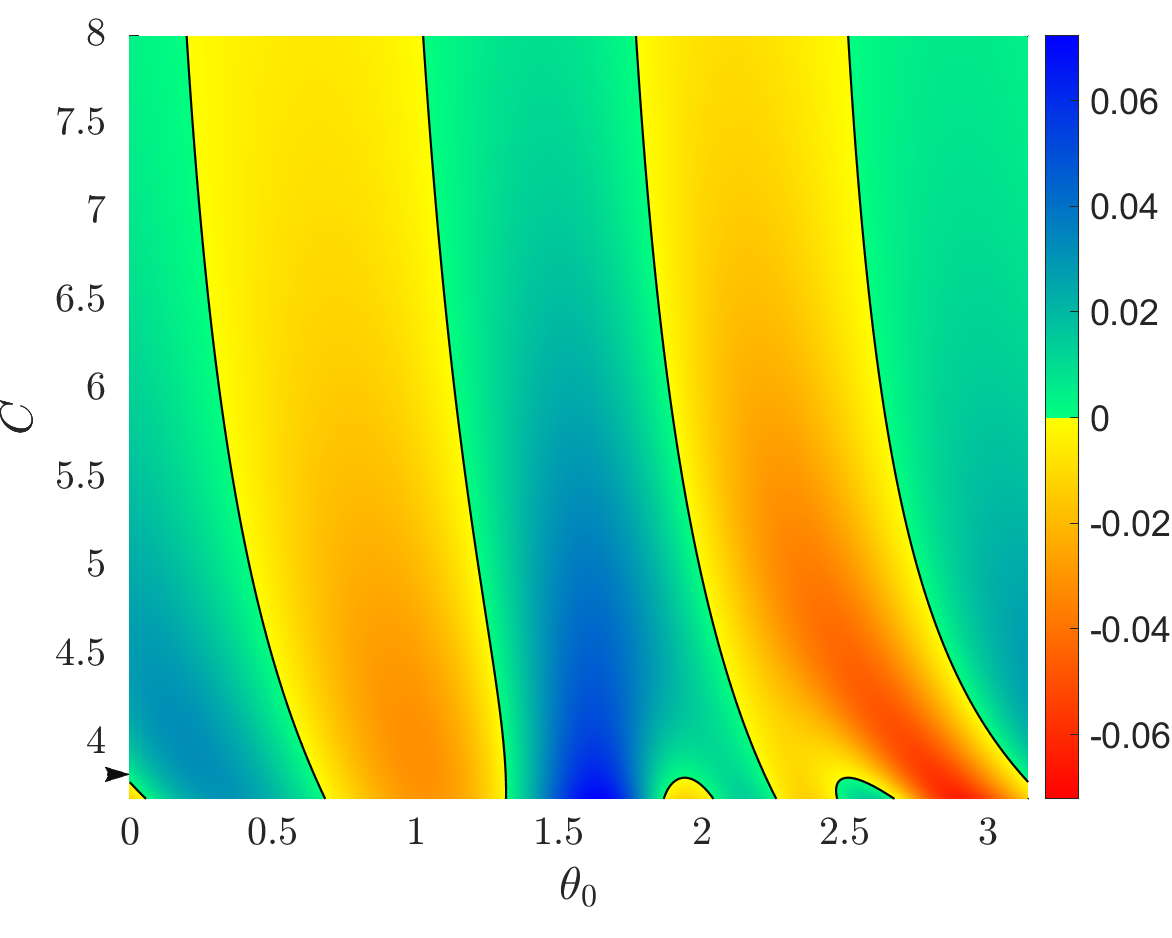}
    \includegraphics[width=0.44\textwidth]{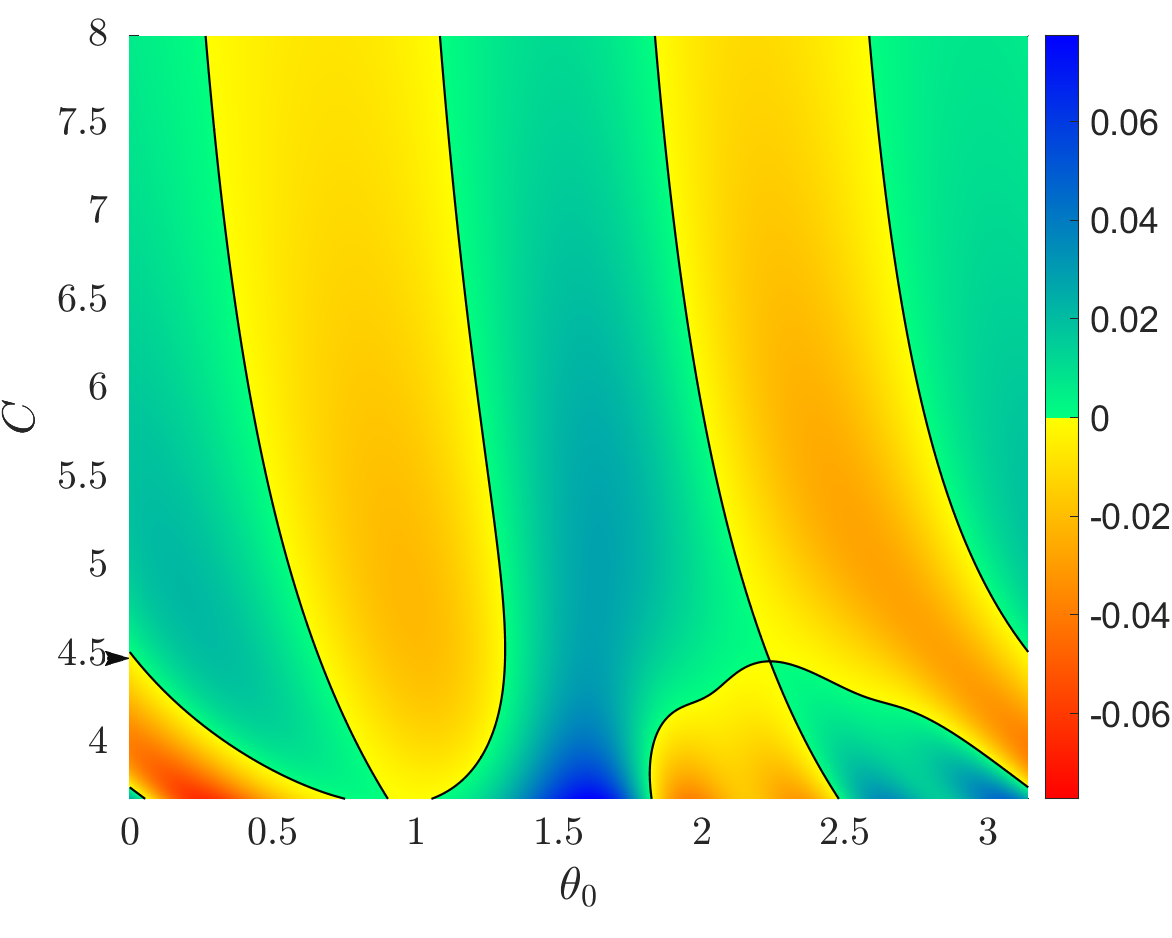}\\
    \vspace{-2mm}
    \includegraphics[width=0.44\textwidth]{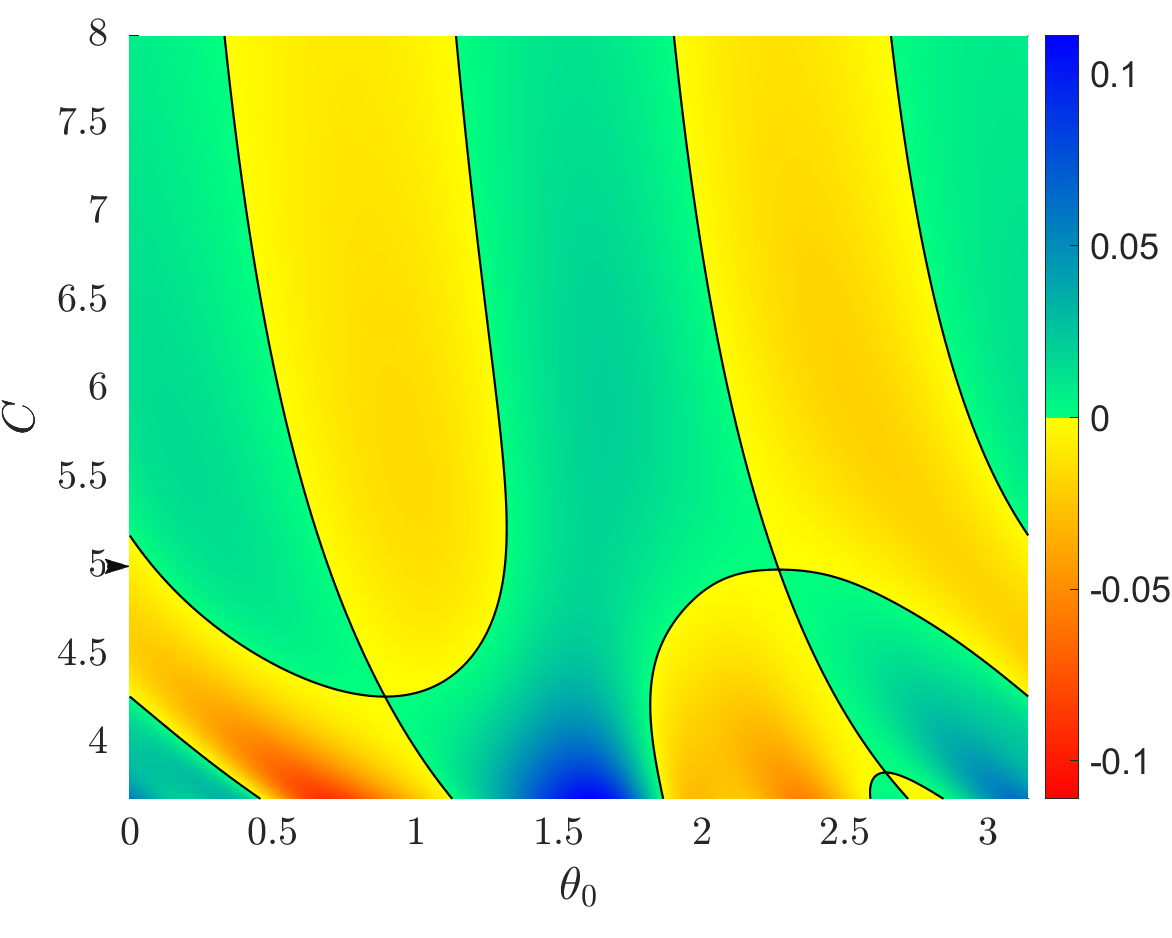}
    \includegraphics[width=0.44\textwidth]{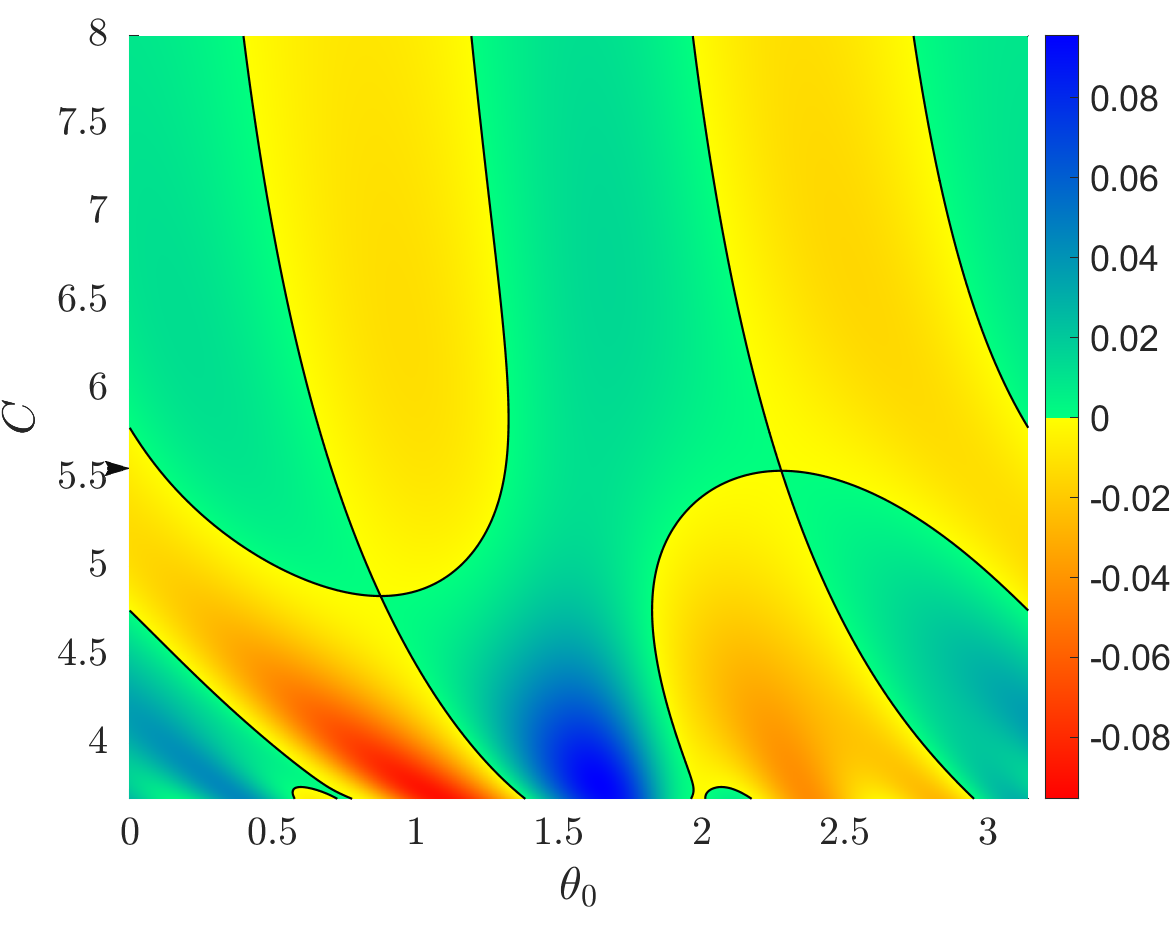}\\
    \vspace{-2mm}
    \includegraphics[width=0.44\textwidth]{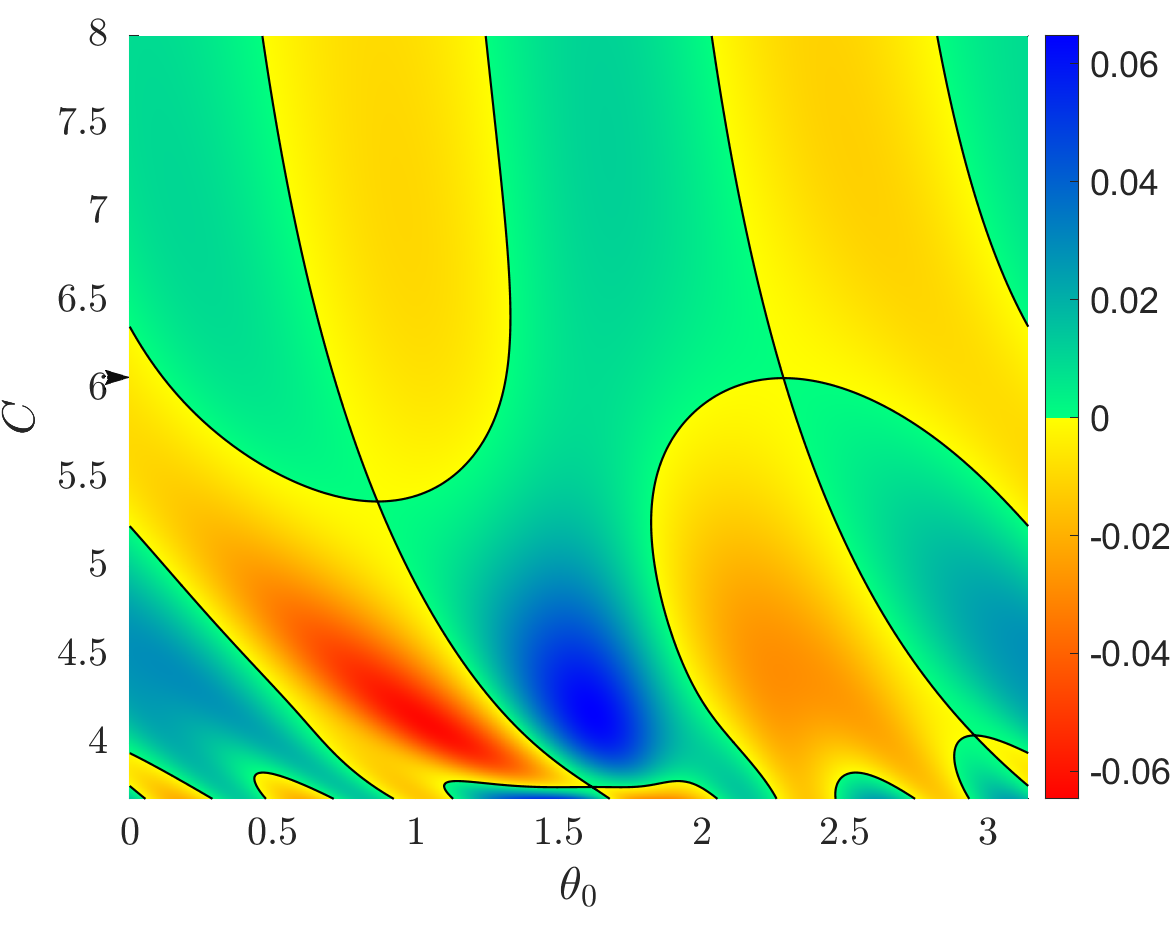}
    \includegraphics[width=0.44\textwidth]{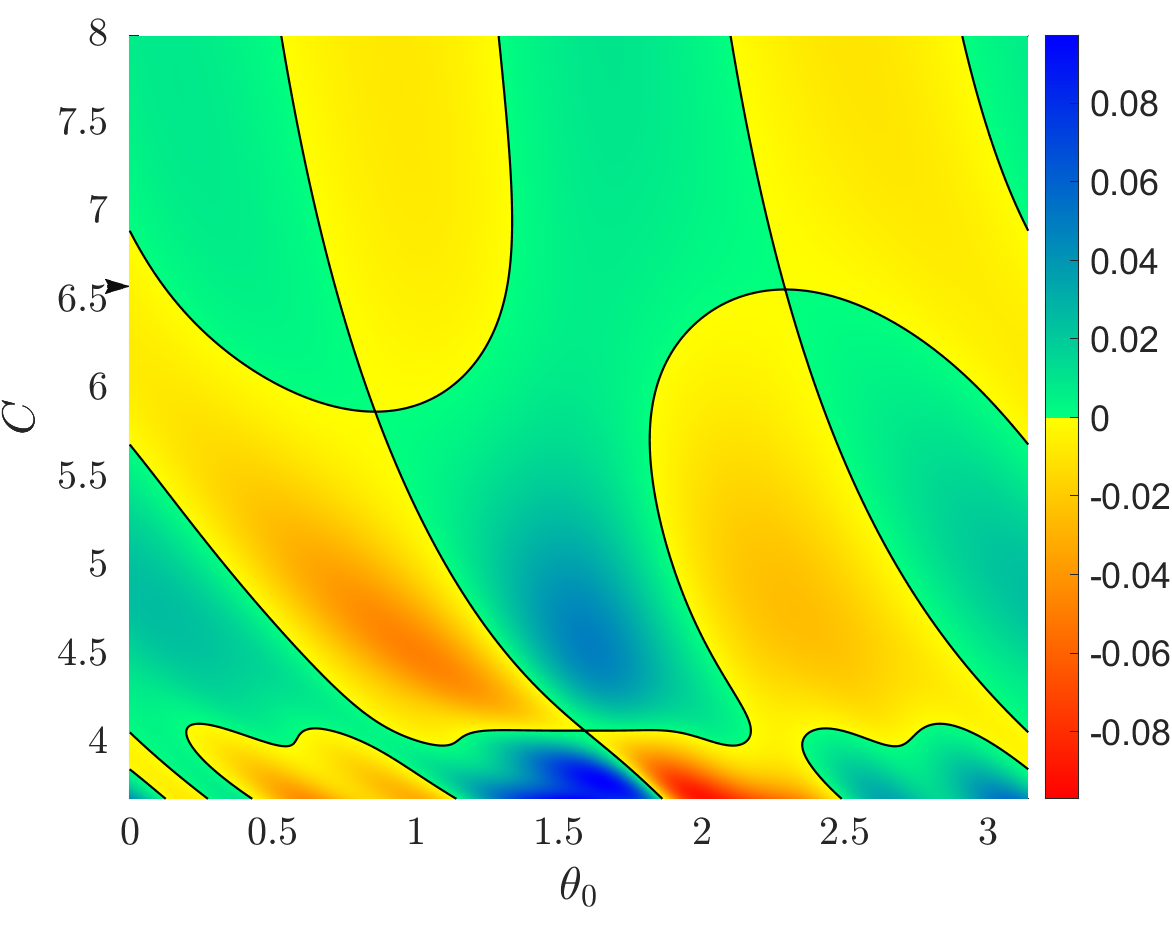}\\
    \vspace{-4mm}
    \caption{Bifurcation diagrams for $\mu=0.1$, $n=1,...,8$ and $C$ in $[C_{L_1},8]$. The value of $\hat{C}(\mu,n)$, for $\mu =0.1$ is also indicated in each plot  with an arrow in the vertical axis.}
    \label{fig:evolucioM}
\end{figure}

\begin{figure}[ht!]
    \centering
    \includegraphics[width=0.44\textwidth]{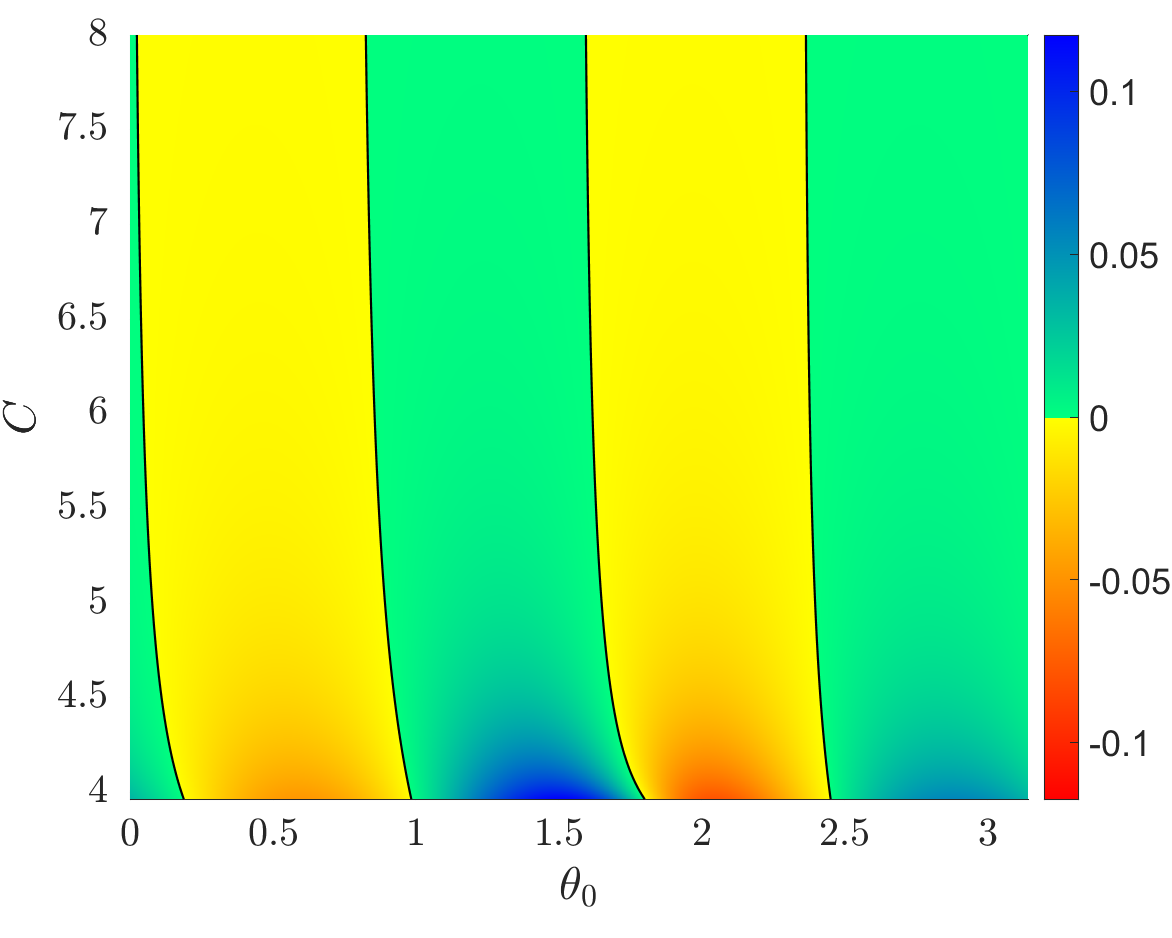}
    \includegraphics[width=0.44\textwidth]{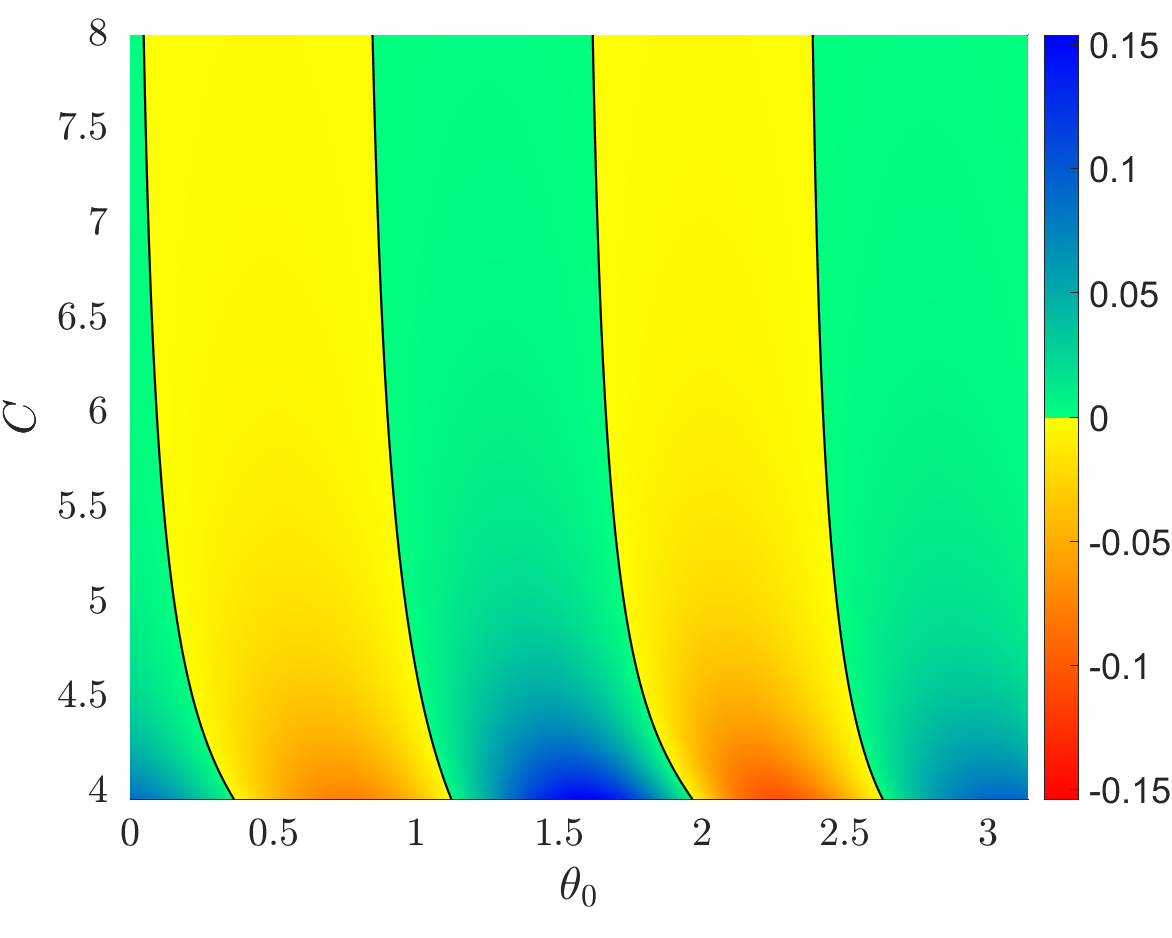}\\
    \vspace{-2mm}
    \includegraphics[width=0.44\textwidth]{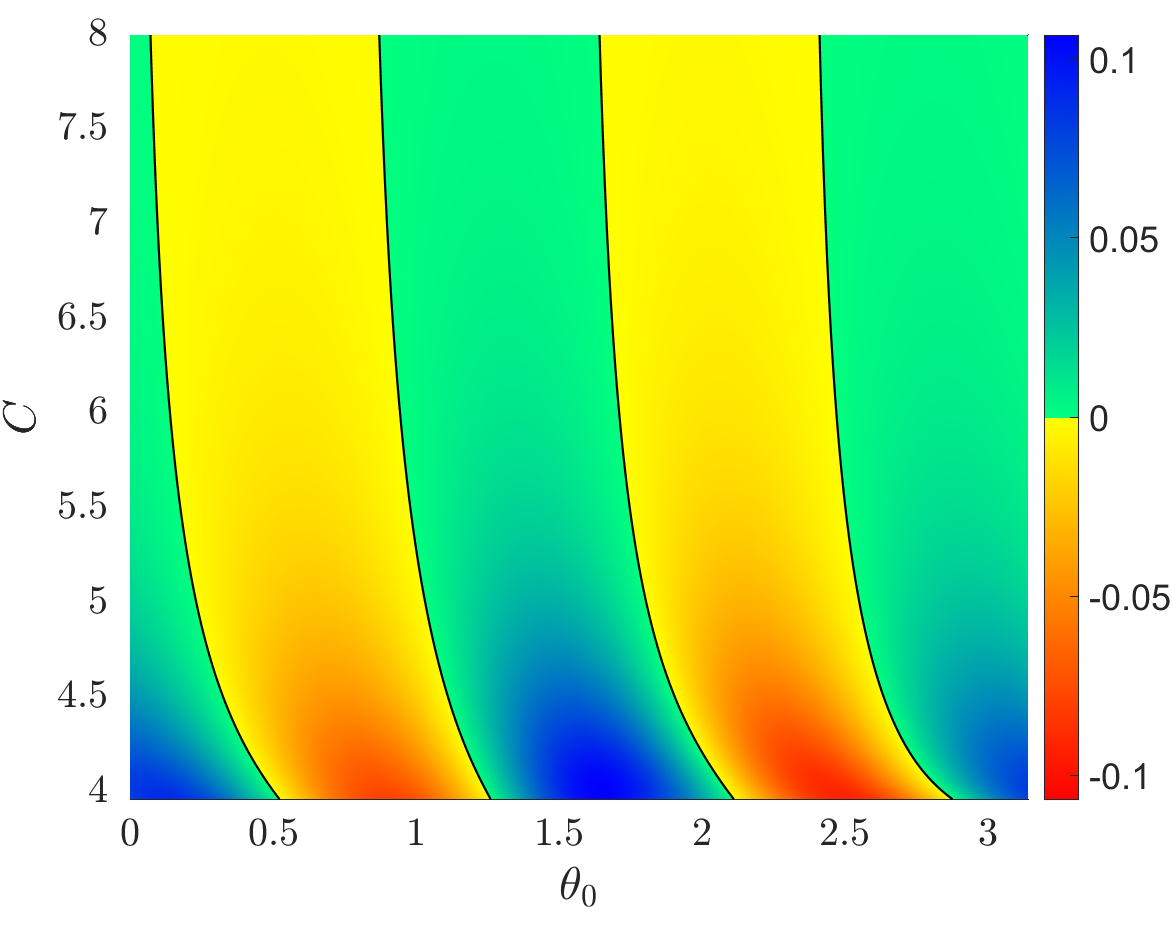}
    \includegraphics[width=0.44\textwidth]{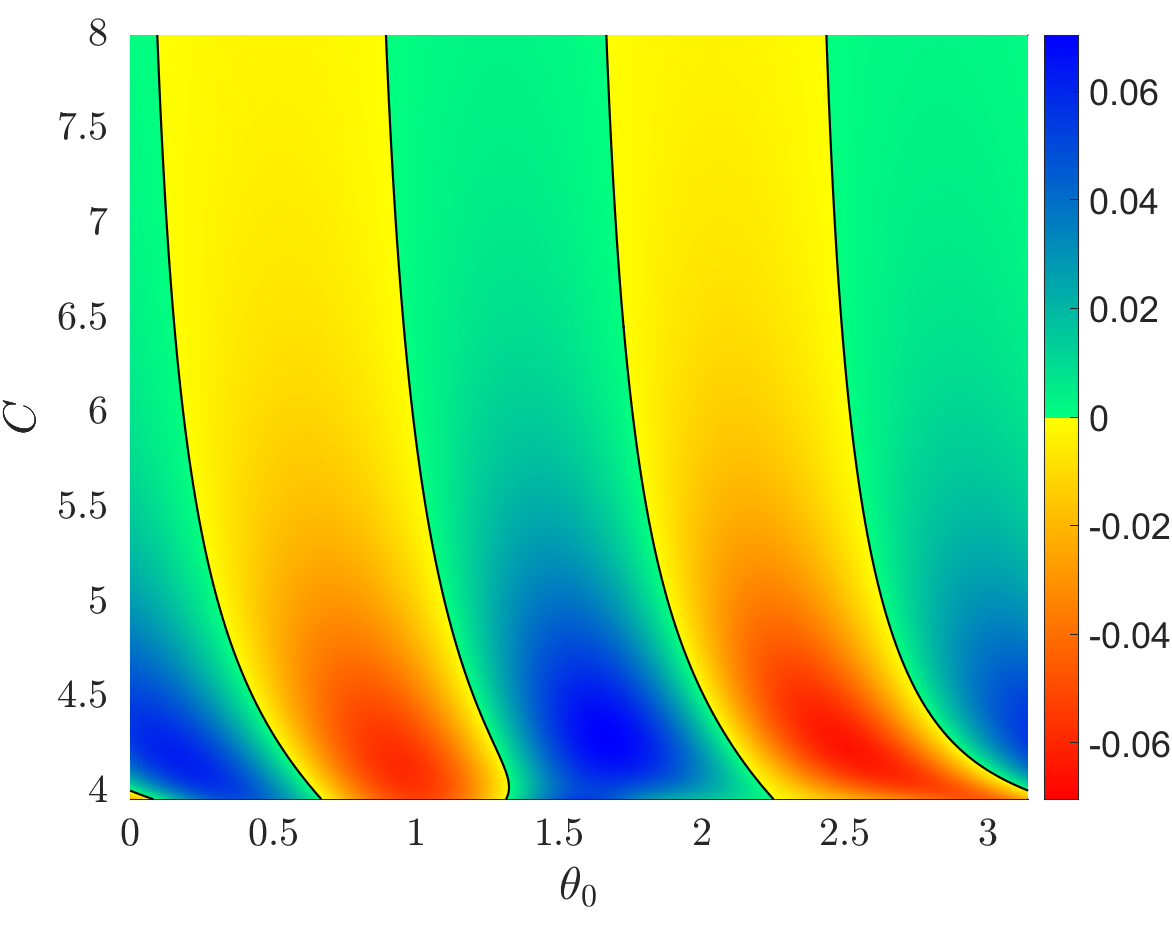}\\
    \vspace{-2mm}
    \includegraphics[width=0.44\textwidth]{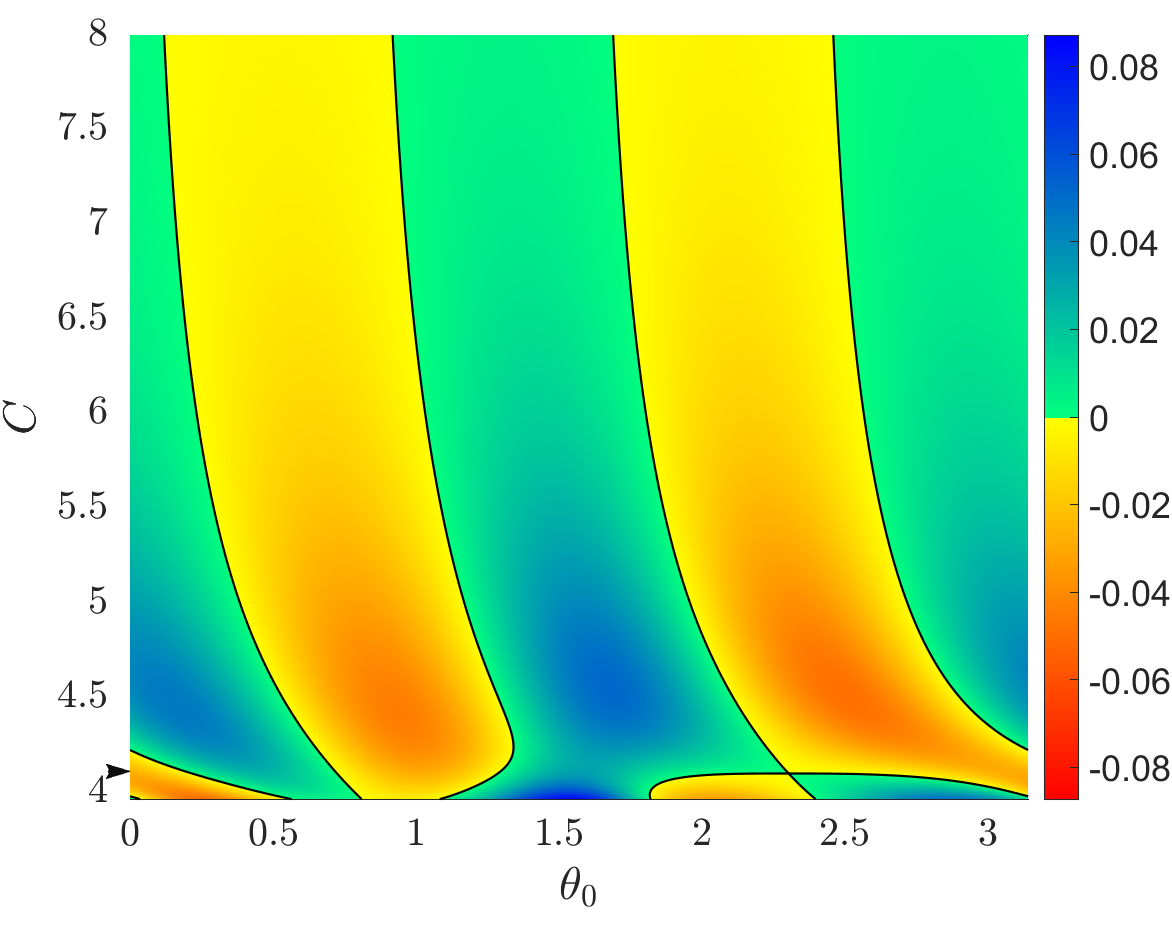}
    \includegraphics[width=0.44\textwidth]{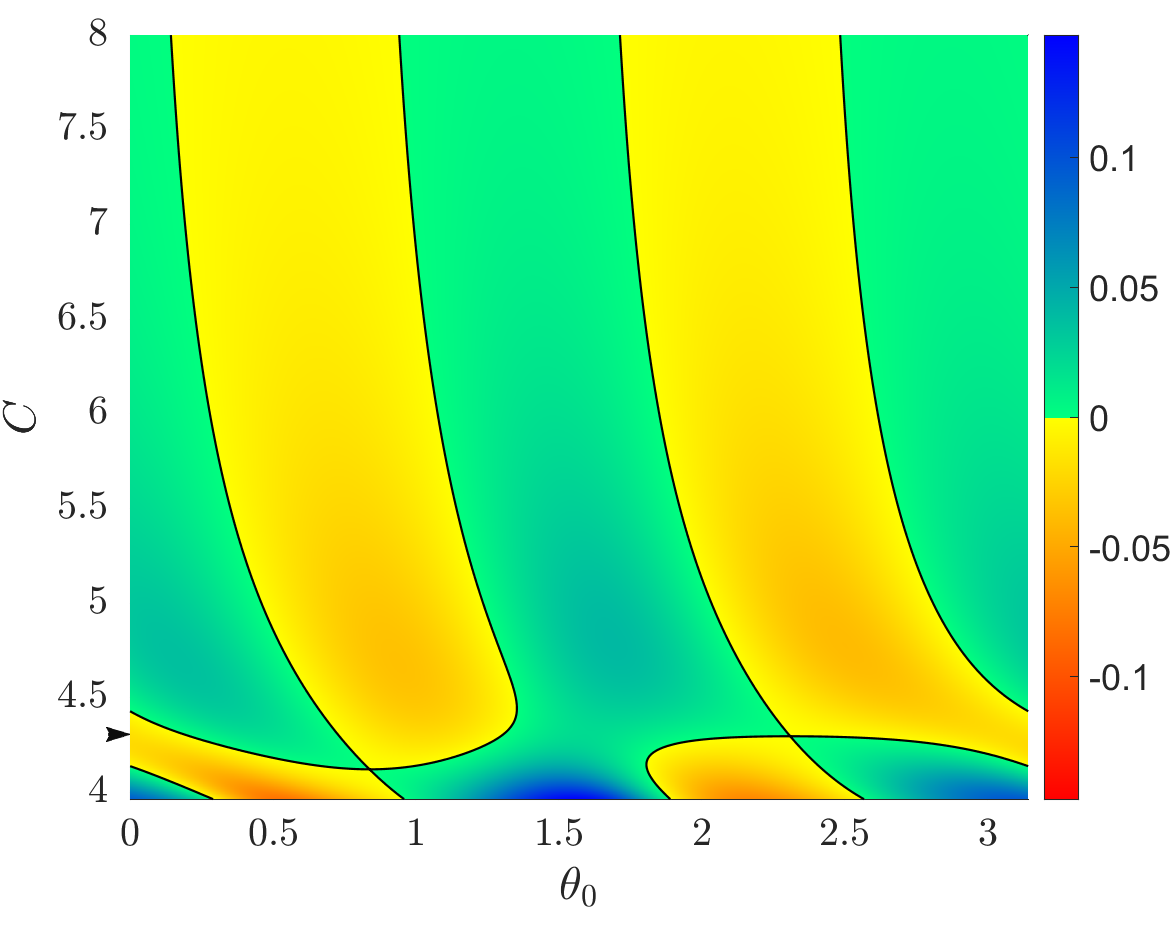}\\
    \vspace{-2mm}
    \includegraphics[width=0.44\textwidth]{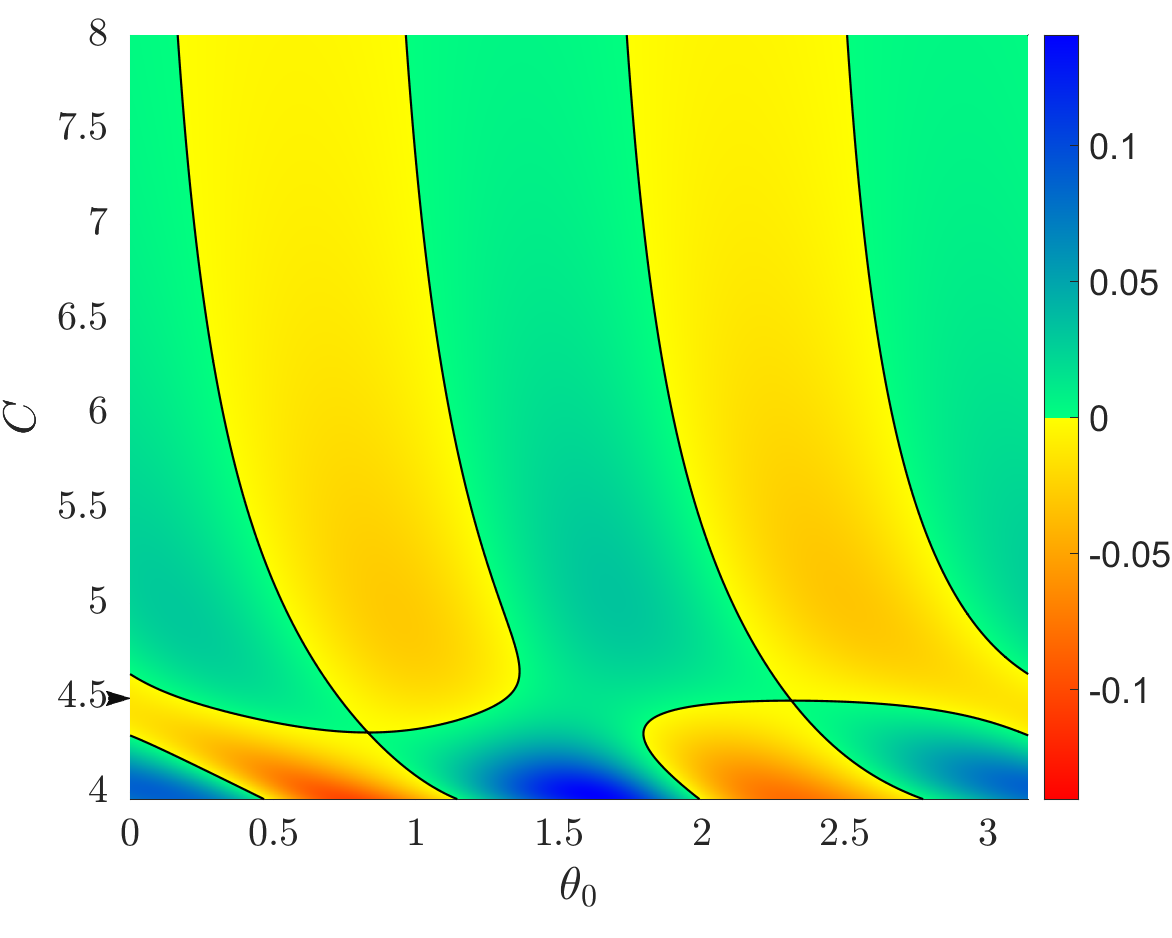}
    \includegraphics[width=0.44\textwidth]{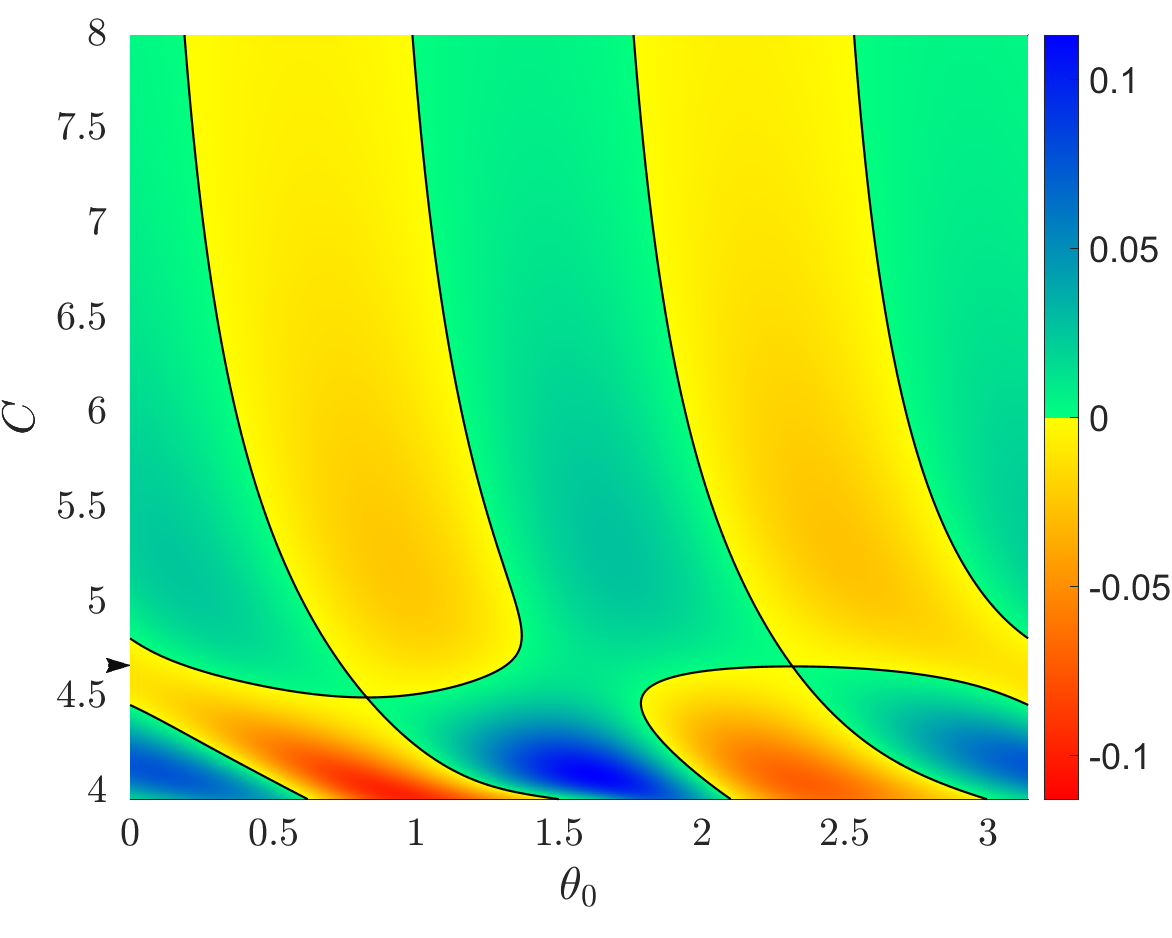}\\
    \vspace{-4mm}
    \caption{Bifurcation diagrams for $\mu=0.8$, $n=1,...,8$ and $C$ in $[C_{L_1},8]$. The value of $\hat{C}(\mu,n)$, for $\mu =0.8$ is also indicated in each plot with an arrow in the vertical axis.}
    \label{fig:evolucioMmu08}
\end{figure}

So far we have described two specific kinds of bifurcations that take place for $n=2$ and $n=3$, for $\mu =0.1$. But from the expression of the angular momentum \eqref{Maordre10} and the previous discussion, we can foresee a great and rich variety of bifurcations. To have a global and exhaustive insight, we have done massive numerical simulations in the following sense: we have fixed a value of $\mu$, and, for a range of values of $C\ge C_{L_1}$ (for example $C\in [C_{L_1},8]$) we have computed the function $M(n,\theta _0)$ for $n=1,...,8$. In Figures \ref{fig:evolucioM} and \ref{fig:evolucioMmu08} we plot the obtained results for $\mu=0.1$ and $\mu=0.8$, what we call {\sl bifurcation diagrams}. For $n$ fixed, we plot the diagram $(\theta _0,C)$ and the colour standing for the value of $M(n, \theta _0)$. The drastic change of colour (from yellow to green) describes the change of sign of  $M(n, \theta _0)$ and therefore  the existence of an $n$-EC orbit. So for any $C$ fixed, we clearly see the number of $n$-EC orbits. Some comments about Figure \ref{fig:evolucioM}  must be made: 
(i) for big values of $C$, the bifurcation diagrams show clearly four $n$-EC orbits for any value of $n$, in accordance with Theorem \ref{th:maintheorem2}. See any plot in the figure. 
(ii) In the first row, right plot, and $C$ close to 3.7, we see the first kind of bifurcation described above for $n=2$. In the second row, left plot, and $C$ close to 3.9, we recognise the second kind of bifurcation described above for $n=3$. 
(iii) It is clear from such diagrams that, when we decrease $C$ and increase the value of $n$,  several phenomena of collapse of families and bifurcation of new families are more visible.  See for example the third row plots, when decreasing $C$, for $\theta _0<\pi /2$, the collapse of two families, and the appearance of two new ones for $\theta _0>\pi /2$. Even richer are the diagrams on the last row of the figure. 
(iv) We have also plotted on each bifurcation diagram the value of the first bifurcation value of $C$ (decreasing $C$), which is precisely  the value  $\hat{C}(\mu,n)$, for $\mu =0.1$ mentioned in Theorem \ref{th:maintheorem2}. We notice in the plot how this value $\hat{C}(\mu,n)$ increases when $n$ increases.

When we take a bigger value of $\mu$, for example, $\mu=0.8$, we obtain Figure \ref{fig:evolucioMmu08}.  Comparing the plots obtained with those of Figure \ref{fig:evolucioM}, we observe two effects: the value of $\hat{C}(\mu,n)$ is smaller, for  the same value of $n$, and moreover, for $n=2,3,4$, a value of $\hat{C}(\mu,n)$ really smaller than $C_{L_1}$ is required (compare the four first plots in Figures \ref{fig:evolucioM} and \ref{fig:evolucioMmu08}). For bigger values of $\mu$ and for the same value of $C\ge C_{L_1}$, the the Hill region gets really smaller, when increasing $\mu$, so quite naturally, the probability of bifurcations decreases. On the other hand, taking $C<C_{L_1}$ represents an enlarging of the Hill's region and therefore a more powerful influence of the big primary, so an easier scenario to have bifurcations. 

{\bf Remark}. We notice that Lemma \ref{caraceco} provides a characterization of an EC orbit if $C$ is large enough. Along the numerical simulations done, where the values of $C$ are not so large, we have  also used the same characterization, but additionally checking that when $M=0$ at a minimum distance, then $U=V=0$.

\subsection{Behaviour of $\hat{C}(\mu,n)$}

As a final goal, we want to describe (numerically) the behaviour of $\hat{C}(\mu,n)=3\mu + \hat{K} (1-\mu)^{2/3}$ for any value of $\mu \in (0,1)$ and $n$. More precisely, for each value of $\mu$ and $n$, and $C$  big enough, Theorem 1 claims that there exist exactly four families of $n$-EC orbits. As discussed in the previous subsection, when decreasing $C$ bifurcations appear in a natural way. So for fixed $\mu$ and $n$, the first value of $C$ (decreasing $C$) such that there appear more than four $n$-EC orbits is precisely the value $\hat{C}(\mu,n)$ formulated in Theorem 1.

\begin{figure}[ht!]
    \centering
    \includegraphics[width=0.95\textwidth]{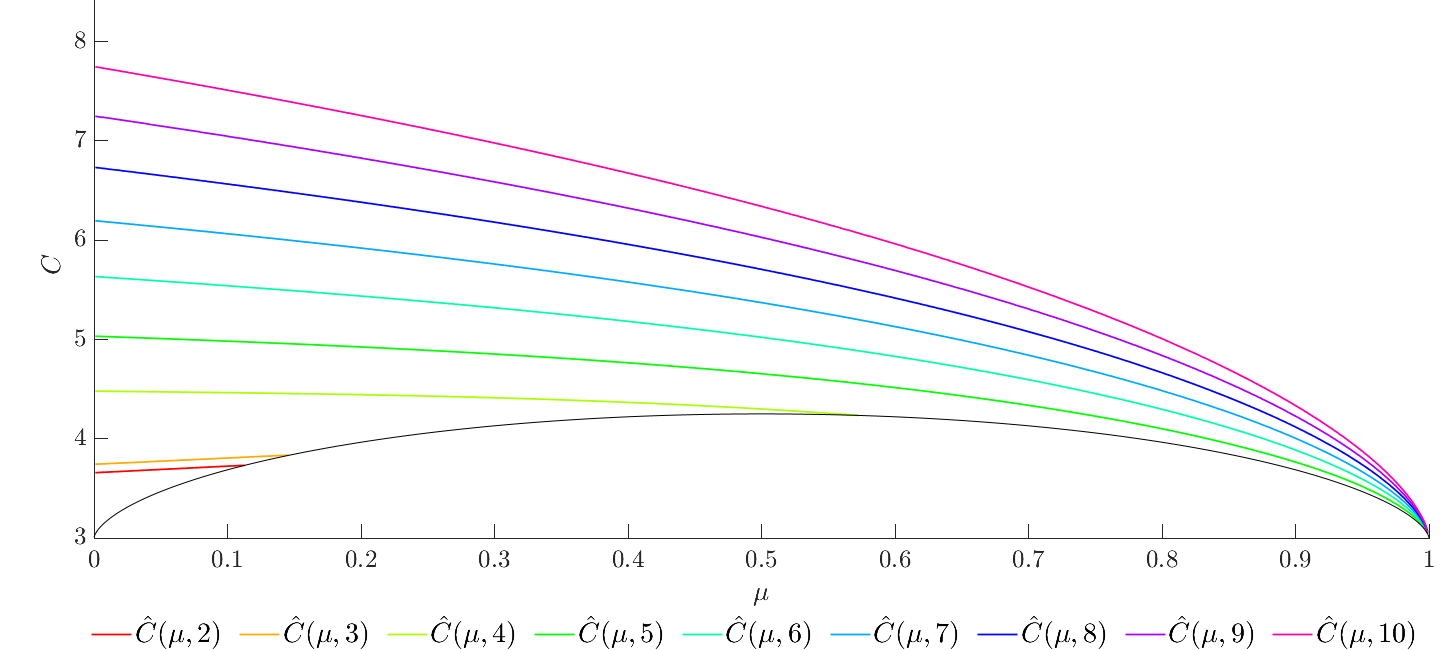}
    \caption{$\hat{C}(\mu,n)$ }
    \label{fig:CsLimit}
\end{figure}

In the previous subsection, we have computed the value $\hat{C}(\mu,n)$, just for $\mu =0.1$ and $n=1,...,8$. Our purpose now is to compute $\hat{C}(\mu,n)$ for any  $\mu \in (0,1)$ and $n$. We will always assume that any value of $C$ considered satisfies $C\ge C_{L_1}$ (recall the Hill regions in Figure \ref{fig:Hill}, there is no possible connection between $P_1$ and $P_2$, and therefore the dynamics around each primary is the simplest possible).

The strategy to  compute numerically  $\hat {C}$, for a fixed $\mu \in (0,1)$ and given $n$, is the following: we take the interval $I=[C_{L_1},C_b]$ of values of $C$, and for each $C\in I$, (starting at $C_b$) we vary $\theta _0\in [0,\pi )$ (that defines the initial conditions of an ejection orbit in synodical Levi-Civita variables) and find the four specific values of $\theta _0$ (such that $M(n,\theta _0)=0$) corresponding to the expected four $n$-EC orbits. So we have four $n$-EC orbits for that value of $C$ and decreasing $C$ we obtain four families of $n$-EC orbits. However as we decrease $C$, we find a value of $C\in I$  such that more than four $n$-EC orbits are found. This means that new families have bifurcated. Next we refine the value of $C$ such that it is the frontier before appearing new families of $n$-EC orbits. That is precisely the specific value of $\hat {C}$.

In Figure \ref{fig:CsLimit} we show the results obtained for $\mu \in (0,1)$ and $n=2,...,10$. Also the curve $(\mu ,C_{L_1})$ has been plotted (in black). Recall that, as mentioned above, we are focussed
on values of $C\ge C_{L_1}$. 
We remark that for $n=1$, the value $\hat{C}(\mu ,1)$ is less than $C_{L_1}$ and therefore is not considered.  Moreover for the specific values of $\mu =0.1$ and $\mu =0.8$ we recover the indicated values in Figures \ref{fig:evolucioM} and \ref{fig:evolucioMmu08} respectively.  From Figure \ref{fig:CsLimit}  it is clear that 
 the value of $\hat {C}(\mu ,n)$ increases when $n$ increases. This means that for higher values of $n$, that is longer time spans integrations, the effect of the other primary is more visible.

We remark that the shape of the curves in Figure \ref{fig:CsLimit} 
provide a hint about the dependence of $K(n)$ in the expression 
$\hat{C}(\mu,n)=3\mu + \hat{K}(n)(1-\mu)^{2/3}$ obtained in
Theorem \ref{th:maintheorem2}, more specifically $\hat{K}(n)=\hat{L}n^{2/3}$.
We will prove rigorously this dependence in Theorem \ref{maintheoremN} in the next Section.

\section{Proof of Theorem 1}

The proof of Theorem \ref{maintheoremN} is also based on a perturvative approach. let us introduce a new parameter $L$ defined as $K=Ln^{2/3}$ in \eqref{eq:C}. In this way, we perform the change \eqref{cvarnouUVC} and a new time $\hat{\mathcal{T}}=\tau/n$: 
\begin{equation}
\left\{
        \begin{aligned}
        u & = \frac{\sqrt{2}(1-\mu)^{1/6}}{\sqrt{L}n^{1/3}}\mathcal{U},\\
        v & = \frac{\sqrt{2}(1-\mu)^{1/6}}{\sqrt{L}n^{1/3}}\mathcal{V},\\
        \hat{\mathcal{T}} &= \frac{2\sqrt{L}(1-\mu)^{1/3}}{n^{2/3}} s=\frac{\tau}{n},\\
        C &= 3\mu + Ln^{2/3}(1-\mu)^{2/3},
        \end{aligned}
\right.
\end{equation}
where we introduce the functions in the new time: 
\[\mathcal{U}(\hat{\mathcal{T}})=U(\tau), \quad \mathcal{V}(\hat{\mathcal{T}})=V(\tau)\]

The system \eqref{LCres} becomes, denoting $\dot {}=\frac{d}{d \tau}$
\begin{equation}\label{eq:noseries}
    \left\{
    \begin{aligned}
    \ddot{\mathcal{U}} =&
     -n^2 \mathcal{U}
     +\frac{8(\mathcal{U}^2+\mathcal{V}^2)\dot{\mathcal{V}}}{L^{3/2}}
     +\frac{12(\mathcal{U}^2+\mathcal{V}^2)^2\mathcal{U}}{L^3}\\
     &+2\mu \left[
     \frac{n^{4/3}}{L(1-\mu)^{2/3}}\left(\frac{1}{r_2}-1\right)
     - \frac{4(\mathcal{U}^2+\mathcal{V}^2)^2}{L^3r_2^3}
     -\frac{2n^{2/3}(\mathcal{U}^2+\mathcal{V}^2)}{L^2(1-\mu)^{1/3}r_2^3}
     +\frac{4n^{2/3}\mathcal{U}^2}{L^2(1-\mu)^{1/3}}
     \right] \mathcal{U},\\[1.8ex]
     \ddot{\mathcal{V}} =&
     -n^2 \mathcal{V}
     -\frac{8(\mathcal{U}^2+\mathcal{V}^2)\dot{\mathcal{U}}}{L^{3/2}}
     +\frac{12(\mathcal{U}^2+\mathcal{V}^2)^2\mathcal{V}}{L^3}\\
     & + 2\mu\left[
     \frac{n^{4/3}}{L(1-\mu)^{2/3}}\left(\frac{1}{r_2}-1\right)
     - \frac{4(\mathcal{U}^2+\mathcal{V}^2)^2}{L^3r_2^3}
     +\frac{2n^{2/3}(\mathcal{U}^2+\mathcal{V}^2)}{L^2(1-\mu)^{1/3}r_2^3}
     - \frac{4n^{2/3}\mathcal{V}^2}{L^2(1-\mu)^{1/3}}
    \right]
    \mathcal{V},
    \end{aligned}
    \right.
\end{equation}
with 
$ r_2 = \sqrt{1 + \frac{4(1-\mu)^{1/3}(\mathcal{U}^2-\mathcal{V}^2)}{Ln^{2/3}} + \frac{4(1-\mu)^{2/3}(\mathcal{U}^2+\mathcal{V}^2)^2}{L^2n^{4/3}}}
$.
Let us introduce the parameter $\xi = 1/\sqrt{L}$, in this way the system \eqref{eq:noseries} becomes
\begin{equation}\label{eq:noseriesxi}
    \left\{
    \begin{aligned}
    \ddot{\mathcal{U}} =&
     -n^2 \mathcal{U}
     +8(\mathcal{U}^2+\mathcal{V}^2)\dot{\mathcal{V}}\xi^3
     +12(\mathcal{U}^2+\mathcal{V}^2)^2\mathcal{U}\xi^6\\
     &+2\mu \left[
     \frac{n^{4/3}}{(1-\mu)^{2/3}}\left(\frac{1}{r_2}-1\right)\xi^2
     - \frac{4(\mathcal{U}^2+\mathcal{V}^2)^2}{r_2^3}\xi^6
     -\frac{2n^{2/3}(\mathcal{U}^2+\mathcal{V}^2)}{(1-\mu)^{1/3}r_2^3}\xi^4
     +\frac{4n^{2/3}\mathcal{U}^2}{(1-\mu)^{1/3}}\xi^4
     \right] \mathcal{U},\\[1.8ex]
     \ddot{\mathcal{V}} =&
     -n^2 \mathcal{V}
     -8(\mathcal{U}^2+\mathcal{V}^2)\dot{\mathcal{U}}\xi^3
     +12(\mathcal{U}^2+\mathcal{V}^2)^2\mathcal{V}\xi^6\\
     & + 2\mu\left[
     \frac{n^{4/3}}{(1-\mu)^{2/3}}\left(\frac{1}{r_2}-1\right)\xi^2
     - \frac{4(\mathcal{U}^2+\mathcal{V}^2)^2}{r_2^3}\xi^6
     +\frac{2n^{2/3}(\mathcal{U}^2+\mathcal{V}^2)}{(1-\mu)^{1/3}r_2^3}\xi^4
     - \frac{4n^{2/3}\mathcal{V}^2}{(1-\mu)^{1/3}}\xi^4
    \right]
    \mathcal{V},
    \end{aligned}
    \right.
\end{equation}
with 
\begin{equation}\label{eq:r2nou}
r_2 = \sqrt{1 + \frac{4(1-\mu)^{1/3}(\mathcal{U}^2-\mathcal{V}^2)}{n^{2/3}}\xi^2 + \frac{4(1-\mu)^{2/3}(\mathcal{U}^2+\mathcal{V}^2)^2}{n^{4/3}}\xi^4}.
\end{equation}

Let us introduce the vectorial notation $\bm{\mathcal{U}} = (\mathcal{U},\mathcal{V},\dot{\mathcal{U}},\dot{\mathcal{V}})$. The system \eqref{eq:noseriesxi} can be written as
\begin{equation}\label{eqtbs}
    \dot{\bm{\mathcal{U}}} = \bm{\mathcal{F}}_0(\bm{\mathcal{U}}) + \mu\bm{\mathcal{F}}_1(\bm{\mathcal{U}}),
\end{equation}
where
\begin{equation}\label{eq:F01}
    \begin{aligned}
        \bm{\mathcal{F}}_0(\bm{\mathcal{U}}) & =
        \left(
        \begin{matrix}
            \dot{\mathcal{U}}\\
            \dot{\mathcal{V}}\\
            -n^2 \mathcal{U}
            +8(\mathcal{U}^2+\mathcal{V}^2)\dot{\mathcal{V}}\xi^3
            +12(\mathcal{U}^2+\mathcal{V}^2)^2\mathcal{U}\xi^6\\
            -n^2 \mathcal{V}
            -8(\mathcal{U}^2+\mathcal{V}^2)\dot{\mathcal{U}}\xi^3
            +12(\mathcal{U}^2+\mathcal{V}^2)^2\mathcal{V}\xi^6
        \end{matrix}
        \right),\\
        \bm{\mathcal{F}}_1(\bm{\mathcal{U}}) & =
        \left(
        \begin{matrix}
            0\\
            0\\
            2\left[
            \frac{n^{4/3}}{(1-\mu)^{2/3}}\left(\frac{1}{r_2}-1\right)\xi^2
            - \frac{4(\mathcal{U}^2+\mathcal{V}^2)^2}{r_2^3}\xi^6
            -\frac{2n^{2/3}(\mathcal{U}^2+\mathcal{V}^2)}{(1-\mu)^{1/3}r_2^3}\xi^4
            +\frac{4n^{2/3}\mathcal{U}^2}{(1-\mu)^{1/3}}\xi^4
            \right] \mathcal{U}\\[1ex]
            2\left[
            \frac{n^{4/3}}{(1-\mu)^{2/3}}\left(\frac{1}{r_2}-1\right)\xi^2
            - \frac{4(\mathcal{U}^2+\mathcal{V}^2)^2}{r_2^3}\xi^6
            +\frac{2n^{2/3}(\mathcal{U}^2+\mathcal{V}^2)}{(1-\mu)^{1/3}r_2^3}\xi^4
            - \frac{4n^{2/3}\mathcal{V}^2}{(1-\mu)^{1/3}}\xi^4
            \right]
            \mathcal{V}
        \end{matrix}
        \right).\\
    \end{aligned}
\end{equation}
{\bf Remark.}
Note that $\bm{\mathcal{F}}_1(\bm{\mathcal{U}})$ only depends on the position variables,  $\bm{\mathcal{F}}_1(\bm{\mathcal{U}}) = \bm{\mathcal{F}}_1(\mathcal{U},\mathcal{V})$.

At this point, our next goal is to find the solution as $\bm{\mathcal{U}}= \bm{\mathcal{U}}_0+\bm{\mathcal{U}}_1$ where
\begin{subequations}
    \begin{equation} \label{eq:u0sys}
        \dot{\bm{\mathcal{U}}}_0 = \bm{\mathcal{F}}_0(\bm{\mathcal{U}}_0),
    \end{equation}
    \begin{equation} \label{eq:u1sys}
        \dot{\bm{\mathcal{U}}}_1 = \mu\bm{\mathcal{F}}_1(\bm{\mathcal{U}}_0+\bm{\mathcal{U}}_1) + \bm{\mathcal{F}}_0(\bm{\mathcal{U}}_0+\bm{\mathcal{U}}_1) - \bm{\mathcal{F}}_0(\bm{\mathcal{U}}_0).
    \end{equation}
\end{subequations}

Notice that $\bm{\mathcal{U}}_0$ is the
solution of the 2-body problem ($\mu =0$) 
 in synodical (rotating) Levi-Civita coordinates.
 That is, we consider system \eqref{eqtbs} as a perturbation of the 2-body problem 
\eqref{eq:u0sys} where the perturbation parameter is $\xi$ which will be small enough
 and for any value of $\mu \in (0,1)$.
Roughly speaking, for big values of the Jacobi constant the problem is close the two body problem of the mass-less body and the collision primary, regardless the value of the mass parameter $\mu$.

Note that we are interested only in the ejection orbits $\bm{\mathcal{U}}^e=\bm{\mathcal{U}}^e_0+\bm{\mathcal{U}}^e_1$ and the initial conditions of these orbits are given by
\begin{equation}\label{ic}
    \bm{\mathcal{U}}_{0}^e(0) = (0,0,n\cos\theta_0,n\sin\theta_0)\quad \text{ and } \quad
    \bm{\mathcal{U}}_{1}^e(0) = \bm{0}.
\end{equation}

To prove the theorem we will use the same strategy of computing the angular momentum $\mathcal{M}(n,\theta_0)$ at the $n$-th minimum of the distance to the origin and find the values of $\theta_0$ such that $\mathcal{M}(n,\theta_0) = 0$. 

Thus, we will compute  $\bm{\mathcal{U}}^e$ and the time needed to reach $n$-th minimum solving $ \left[\mathcal{U}^e\dot{\mathcal{U}}^e + \mathcal{V}^e \dot{\mathcal{V}}^e\right](\theta_0,\hat{\mathcal{T}}^*) = 0$.
The last step will be to calculate $\mathcal{M}(n,\theta_0)$.

\subsection*{The unperturbed system}

As a first step we must solve system \eqref{eq:u0sys}
\begin{equation} \label{eq:KeplerNouFort}
        \left\{
        \begin{aligned}
        \ddot{\mathcal{U}}_0 &= - n^2\mathcal{U}_0 + 8(\mathcal{U}_0^2 + \mathcal{V}_0^2)\dot{\mathcal{V}}_0\xi^3 + 12 (\mathcal{U}_0^2+\mathcal{V}_0^2)^2\mathcal{U}_0   \xi^6,\\[1.7ex]
        \ddot{\mathcal{V}}_0 &= - n^2\mathcal{V}_0 - 8(\mathcal{U}_0^2 + \mathcal{V}_0^2)\dot{\mathcal{U}}_0\xi^3 + 12 (\mathcal{U}_0^2+\mathcal{V}_0^2)^2\mathcal{V}_0   \xi^6,
        \end{aligned}
        \right.
\end{equation}
with initial conditions $\bm{\mathcal{U}}_{0} (0)= (0,0,n\cos\theta_0,n\sin\theta_0)$.

In order to obtain the solution of this system, we first consider the 2-body problem in sidereal coordinates:
\begin{equation}\label{eq:KeplerNou}
\left\{
    \begin{aligned}
        \ddot{\bar{\mathcal{U}}}_0 &=
        -\left[n^2-4(\bar{\mathcal{U}}_0\dot{\bar{\mathcal{V}}}_0  - \bar{\mathcal{V}}_0\dot{\bar{\mathcal{U}}}_0)\xi^3 \right]\bar{\mathcal{U}}_0,\\[1.2ex]
        \ddot{\bar{\mathcal{V}}}_0 &=
        -\left[n^2-4(\bar{\mathcal{U}}_0\dot{\bar{\mathcal{V}}}_0  - \bar{\mathcal{V}}_0\dot{\bar{\mathcal{U}}}_0)\xi^3 \right]\bar{\mathcal{V}}_0,
    \end{aligned}
    \right.
\end{equation}
and the change of time
\begin{equation}\label{dtdtau}
    \frac{dt}{d\hat{\mathcal{T}}} = 4(\bar{\mathcal{U}}_0^2 + \bar{\mathcal{V}}_0^2)\xi^3.
\end{equation}
being $(\bar{\mathcal{U}}_0(\hat{\mathcal{T}}),\bar{\mathcal{V}_0}(\hat{\mathcal{T}})$ the associated solutions.

We recall that for the two body problem the angular momentum is constant, consequently
\begin{equation}
    \left[\bar{\mathcal{U}}_0\dot{\bar{\mathcal{V}}}_0  - \bar{\mathcal{V}}_0\dot{\bar{\mathcal{U}}}_0\right](\hat{\mathcal{T}}) = \left(\bar{\mathcal{U}}_{0}\dot{\bar{\mathcal{V}}}_{0}  - \bar{\mathcal{V}}_{0}\dot{\bar{\mathcal{U}}}_{0}\right)(0),
\end{equation}
and therefore the solution of \eqref{eq:KeplerNou} is given by
\begin{equation}\label{eq:sol2cside}
\left\{
    \begin{aligned}
        \bar{\mathcal{U}}_0(\hat{\mathcal{T}}) &=
        \bar{\mathcal{U}}_{0}(0)\cos(\omega\hat{\mathcal{T}}) + \frac{\dot{\bar{\mathcal{U}}}_{0}(0)}{\omega}\sin(\omega\hat{\mathcal{T}}),\\
        \bar{\mathcal{V}}_0(\hat{\mathcal{T}}) &=
        \bar{\mathcal{V}}_{0}(0)\cos(\omega\hat{\mathcal{T}}) + \frac{\dot{\bar{\mathcal{V}}}_{0}(0)}{\omega}\sin(\omega\hat{\mathcal{T}}),
    \end{aligned}
\right.
\end{equation}
where $\omega = \sqrt{n^2-4(\bar{\mathcal{U}}_{0}\dot{\bar{\mathcal{V}}}_{0}  - \bar{\mathcal{V}}_{0}\dot{\bar{\mathcal{U}}}_{0})(0)\xi^3}$.
Moreover the value $t(\hat{\mathcal{T}})$ is simply obtained from \eqref{dtdtau} and \eqref{eq:sol2cside}:
\begin{equation}\label{eq:timet}
\begin{aligned}
    t(\hat{\mathcal{T}}) &= 2\Bigg[\left(\bar{\mathcal{U}}_{0}^2+\bar{\mathcal{V}}_{0}^2\right)(0)
    \left(\hat{\mathcal{T}}+\frac{\cos(\omega\hat{\mathcal{T}})\sin(\omega\hat{\mathcal{T}})}{\omega}\right)
    +\frac{2\big(\bar{\mathcal{U}}_{0}\dot{\bar{\mathcal{U}}}_{0}+\bar{\mathcal{V}}_{0}\dot{\bar{\mathcal{V}}}_{0}\big)(0)}{\omega^2}\sin^2(\omega\hat{\mathcal{T}})\\
    &\hspace{10mm}+ \frac{\big(\dot{\bar{\mathcal{U}}}_{0}^2+\dot{\bar{\mathcal{V}}}_{0}^2\big)(0)}{\omega^2}
    \left(\hat{\mathcal{T}}-\frac{\cos(\omega\hat{\mathcal{T}})\sin(\omega\hat{\mathcal{T}})}{\omega}\right)
    \Bigg]\xi^3,
\end{aligned}
\end{equation}

Now we apply 
the rotation transformation to \eqref{eq:sol2cside} 
to obtain the solution of system in synodical coordinates \eqref{eq:KeplerNouFort},
\begin{equation} \label{eq:solucioUnewRot}
\left\{
    \begin{aligned}
        \mathcal{U}_0(\hat{\mathcal{T}}) &= \bar{\mathcal{U}}_0(\hat{\mathcal{T}})\cos(-t/2) - \bar{\mathcal{V}}_0(\hat{\mathcal{T}})\sin(-t/2),\\
        \mathcal{V}_0(\hat{\mathcal{T}}) &= \bar{\mathcal{U}}_0(\hat{\mathcal{T}})\sin(-t/2) + \bar{\mathcal{V}}_0(\hat{\mathcal{T}})\cos(-t/2),\\
        \dot{\mathcal{U}}_0(\hat{\mathcal{T}}) &= \left[ \dot{\bar{\mathcal{U}}}_0 + 2(\bar{\mathcal{U}}_0^2 + \bar{\mathcal{V}}_0^2)\bar{\mathcal{V}}_0\xi^3 \right]
        \cos(-t/2)
        -
        \left[ \dot{\bar{\mathcal{V}}}_0 - 2(\bar{\mathcal{U}}_0^2 + \bar{\mathcal{V}}_0^2)\bar{\mathcal{U}}_0\xi^3 \right]
        \sin(-t/2),\\
        \dot{\mathcal{V}}_0(\hat{\mathcal{T}}) &= \left[ \dot{\bar{\mathcal{U}}}_0 + 2(\bar{\mathcal{U}}_0^2 + \bar{\mathcal{V}}_0^2)\bar{\mathcal{V}}_0 \xi^3\right]
        \sin(-t/2)
        +
        \left[ \dot{\bar{\mathcal{V}}}_0 - 2(\bar{\mathcal{U}}_0^2 + \bar{\mathcal{V}}_0^2)\bar{\mathcal{U}}_0 \xi^3\right]
        \cos(-t/2),
    \end{aligned}
    \right.
\end{equation}

Notice that the relation between the sidereal initial conditions and the synodical ones
$({\mathcal{U}}_{0}{\mathcal{V}}_{0},\dot{{\mathcal{U}}}_{0}
 \dot{{\mathcal{V}}}_{0})(0)$,
are obtained simply  from \eqref{eq:solucioUnewRot} putting $\hat{\mathcal{T}}= 0$
\begin{equation} \label{eq:condicionsInicialNewU}
\left\{
    \begin{aligned}
        \mathcal{U}_{0}(0) &= \bar{\mathcal{U}}_{0}(0),\\
        \mathcal{V}_{0}(0) &=  \bar{\mathcal{V}}_{0}(0),\\
        \dot{\mathcal{U}}_{0}(0) &=\dot{\bar{\mathcal{U}}}_{0}(0) + 2(\bar{\mathcal{U}}_{0}^2 + \bar{\mathcal{V}}_{0}^2)(0)\bar{\mathcal{V}}_{0}(0)\xi^3, \\
        \dot{\mathcal{V}}_{0}(0) &= \dot{\bar{\mathcal{V}}}_{0}(0) - 2(\bar{\mathcal{U}}_{0}^2 + \bar{\mathcal{V}}_{0}^2)(0)\bar{\mathcal{U}}_{0}(0) \xi^3,
    \end{aligned}
    \right.
\qquad\qquad
\left\{
    \begin{aligned}
        \bar{\mathcal{U}}_{0}(0) &= \mathcal{U}_{0}(0),\\
        \bar{\mathcal{V}}_{0}(0) &=  \mathcal{V}_{0}(0),\\
        \dot{\bar{\mathcal{U}}}_{0}(0) &=\dot{\mathcal{U}}_{0}(0) - 2(\mathcal{U}_{0}^2 + \mathcal{V}_{0}^2)(0)\mathcal{V}_{0}(0)\xi^3, \\
        \dot{\bar{\mathcal{V}}}_{0}(0) &= \dot{\mathcal{V}}_{0}(0) + 2(\mathcal{U}_{0}^2 + \mathcal{V}_{0}^2)(0)\mathcal{U}_{0}(0) \xi^3.
    \end{aligned}
    \right.
\end{equation}

Since we are interested in the particular case of ejection orbits, which have as their initial condition
\begin{equation}
    \bar{\bm{\mathcal{U}}}_{0}(0) = (0,0,n\cos\theta_0,n\sin\theta_0),
\end{equation}
the corresponding ejection solution is given by $\bm{\mathcal{U}}_{0}^e =(\mathcal{U}_{0}^e,\mathcal{V}_{0}^e, \dot{\mathcal{U}}_{0}^e,\dot{\mathcal{V}}_{0}^e)$, where:
\begin{equation}\label{eq:solucioEcosKeplerRaro}
\left\{
\begin{aligned}
    \mathcal{U}_{0}^e(\theta_0,\hat{\mathcal{T}}) &= \big[\cos\theta_0
    \cos\left(-t/2\right)
    -
    \sin\theta_0
    \sin\left(-t/2\right)\big]\sin(n\hat{\mathcal{T}})=\cos (\theta_0-t/2)  \sin(n\hat{\mathcal{T}}),\\[1.5ex]
    \mathcal{V}_{0}^e(\theta_0,\hat{\mathcal{T}}) &= \big[\cos\theta_0
    \sin\left(-t/2\right)
    +
    \sin\theta_0
    \cos\left(-t/2\right)\big]\sin(n\hat{\mathcal{T}})=\sin (\theta_0-t/2)  \sin(n\hat{\mathcal{T}}),
\end{aligned}
\right.
\end{equation}
with
\begin{equation}\label{eq:tempsRotacio}
    t = 2\left[\hat{\mathcal{T}}-\frac{\cos(n\hat{\mathcal{T}})\sin(n\hat{\mathcal{T}})}{n}\right]\xi^3.
\end{equation}
If we denote by $\hat{\mathcal{T}}_0^*$ the time needed by $\bm{\mathcal{U}}_0^e(\hat{\mathcal{T}})$ to reach the $n$-th minimum distance to the origin, it is clear from \eqref{eq:solucioEcosKeplerRaro} that 
\begin{equation}
    \hat{\mathcal{T}}_0^* = \pi.
\end{equation}

\subsection*{The perturbed system}

In order to solve the perturbed problem  (i.e. $\mu\neq0$)  we rewrite system \eqref{eq:u1sys} as
\begin{equation} \label{eq:u1sysV2}
    \begin{aligned}
        \dot{\bm{\mathcal{U}}}_{1} &=
     D\bm{\mathcal{F}}_0(\bm{\mathcal{U}}_{0})\bm{\mathcal{U}}_{1}
        + \bm{\mathcal{G}}(\bm{\mathcal{U}}_{1}),
    \end{aligned}
\end{equation}
where $\bm{\mathcal{U}}_{0}=\bm{\mathcal{U}}_{0}^e$ is the ejection solution \eqref{eq:solucioEcosKeplerRaro} of the two body problem and 
\begin{equation}\label{defG}
    \bm{\mathcal{G}}(\bm{\mathcal{U}}_{1}) = \mu\bm{\mathcal{F}}_1(\bm{\mathcal{U}}_{0}^e+\bm{\mathcal{U}}_{1})
    + \bm{\mathcal{F}}_0(\bm{\mathcal{U}}_{0}^e+\bm{\mathcal{U}}_{1})
    - \bm{\mathcal{F}}_0(\bm{\mathcal{U}}_{0}^e)
    - D\bm{\mathcal{F}}_0(\bm{\mathcal{U}}_{0}^e)\bm{\mathcal{U}}_{1}.
\end{equation}
Note that the ejection solution $\bm{\mathcal{U}}_{1}^e $ has zero initial condition and therefore  is the solution of the implicit equation
\begin{equation}\label{eqH}
\bm{\mathcal{U}}_{1}^e=\bm{\mathcal{H}}\{\bm{\mathcal{U}}_{1}^e\},
\end{equation}
where we define
\begin{equation}\label{defH}
\bm{\mathcal{H}}\{\bm{\mathcal{U}}\}(\hat{\mathcal{T}})
    = X(\hat{\mathcal{T}})\int_0^{\hat{\mathcal{T}}} X^{-1}(\hat{\mathcal{T}})\bm{\mathcal{G}}(\bm{\mathcal{U}}(\hat{\mathcal{T}}))\,d\hat{\mathcal{T}},
\end{equation}
and $X(\hat{\mathcal{T}})$ is the fundamental matrix of the linear system:
\begin{equation} \label{eq:u1sysV2linear}
\dot{\bm{\mathcal{U}}}_{1} =
D\bm{\mathcal{F}}_0(\bm{\mathcal{U}}_{0}^e)\bm{\mathcal{U}}_{1}.
\end{equation}
We will apply a Fixed Point Theorem to prove the existence of the solution $\bm{\mathcal{U}}_1^e$.
 Thus we consider the space
\[
\mathcal{\chi}=\{\bm{\mathcal{U}}:[0,T]\longrightarrow \mathbb{R}^4,\quad   
\bm{\mathcal{U}}\quad \mbox{continuous} \},
\]
 for a given $T$, for example $T=2\pi$. 

For a given function
 ${\bm{\mathcal{U}}} = (\mathcal{U},\mathcal{V},\dot{\mathcal{U}},\dot{\mathcal{V}})\in \chi$  we consider the norm:
 \begin{equation}\label{eq:norm}
 || {\bm{\mathcal{U}}}|| = \sup _{\hat{\mathcal{T}}\in[0,T]}(n|\mathcal{U}(\hat{\mathcal{T}})| + n|\mathcal{V}(\hat{\mathcal{T}}))| 
+ |\dot{\mathcal{U}}(\hat{\mathcal{T}})| + |\dot{\mathcal{V}}(\hat{\mathcal{T}})|.     
 \end{equation}
With this norm $\mathcal{\chi}$ is a Banach space.

As usual, given an $R>0$, we define the ball 
$ B_R(\bm{0}) \subset \chi $ 
as the functions $\bm{\mathcal{U}}\in \chi$ such that   $\lVert\bm{\mathcal{U}}\rVert \leq R$.

Next lemmas show  that the required hypotheses for the Fixed Point Theorem to be applied  are satisfied. 


\begin{lemma}\label{Hfitatlema}
    There exist $\xi_0>0$ and a constant $M_1>0$
    such that, for $0<\xi<\xi_0$,  $0< \mu<  1$ and $n\in\mathbb{N}$, 
    \[
        \lVert \bm{\mathcal{H}}\{\bm{0}\} \rVert \leq M_1\mu \xi^6.
    \]
\end{lemma}
\begin{proof}
    See Appendix \ref{app:Lema2}.
\end{proof}

\begin{lemma}\label{le:LipschitzLema}
    There exist $0<\xi_1\le \xi_0$ and a constant $M_2\ge M_1$
    such that, for $0<\xi<\xi_1$, $0< \mu < 1$ and $n\in \mathbb{N}$, given $\bm{\mathcal{U}}_\oplus$, $\bm{\mathcal{U}}_\ominus\in B_R(\bm{0})$ with $R =2M_1\mu \xi^6$ then
    \[
        \lVert{\bm{\mathcal{H}}}\{\bm{\mathcal{U}}_\oplus\}
        -{\bm{\mathcal{H}}}\{\bm{\mathcal{U}}_\ominus\}\rVert \le M_2\mu\xi^6
        \lVert\bm{\mathcal{U}}_\oplus-\bm{\mathcal{U}}_\ominus\rVert.
    \]
\end{lemma}
\begin{proof}
    See Appendix \ref{app:Lema3}.
\end{proof}

At this point we select $\xi_1$ s.t. $M_2\mu\xi_1^6<1/2$, so we have the following result
\begin{lemma}\label{le:fixpont}
Under the same hypotheses of Lemma \ref{le:LipschitzLema} if we reduce $\xi_1$ such that $M_2\mu\xi_1^6<1/2$, one has that
 the operator $\bm{\mathcal{H}} : B_R(\bm{0}) \rightarrow B_R(\bm{0})$ and it is a contraction and therefore there exists  a unique $\bm{\mathcal{U}}_1^e\in B_R(\bm{{0}})$ which is solution of the equation \eqref{eqH} in $\chi$.
\end{lemma}
\begin{proof}
If $\bm{\mathcal{U}}\in B_R(\bm{0})$, then:
\[
    \lVert\bm{\mathcal{H}}\{\bm{\mathcal{U}}\}\rVert
    =\lVert\bm{\mathcal{H}}\{\bm{0}\}+\bm{\mathcal{H}}\{\bm{\mathcal{U}}\}-\bm{\mathcal{H}}\{\bm{0}\}\rVert
    \leq
    \lVert\bm{\mathcal{H}}\{\bm{0}\}\rVert+\lVert\bm{\mathcal{H}}\{\bm{\mathcal{U}}\}-\bm{\mathcal{H}}\{\bm{0}\}\rVert
    \leq \frac{R}{2}+\frac{R}{2} 
    =  R,
\]
and we already know by Lemma \ref{le:LipschitzLema} that $\bm{\mathcal{H}}$ is Lipschitz with Lipschitz constant $M_2\mu\xi^6<1/2$.

By the Fixed Point Theorem there exists a unique 
$\bm{\mathcal{U}}_1^e\in B_R(\bm{{0}})$ which is solution of the equation \eqref{eqH}.
\end{proof}

Observe that once we know the existence and bounds of the function $\bm{\mathcal{U}}_1^e$, its smoothness is a consequence of being solution of a smooth differential equation.

The results of the previous lemmas give us the following properties: 
\begin{itemize}[topsep=-0.5ex]
    \item
    $\lVert\bm{\mathcal{U}}_1^e\rVert \le R=2M_1\mu \xi ^6$,
    \item
    $
    ||\bm{\mathcal{U}}_1^e-\bm{\mathcal{H}}\{\bm{{0}}\}||
    =||\bm{\mathcal{H}}\{\bm{\mathcal{U}}_1^e\}-\bm{\mathcal{H}}\{\bm{0}\}||
\le M_2\mu \xi ^6    ||\bm{\mathcal{U}}_1^e||\le 2M_1M_2\mu^2 \xi^{12}.$
\end{itemize}
Writing these inequalities in components, and using the definition of the norm \eqref{eq:norm}, we have
\begin{itemize}[topsep=-0.5ex]
    \item $\mathcal{U}_1 ^e= \mathcal{H}_1\{\bm{0}\} + \cfrac{\mu^2}{n}\mathcal{O}(\xi^{12})$,
    \item $\mathcal{V}_1 ^e= \mathcal{H}_2\{\bm{0}\} + \cfrac{\mu^2}{n}\mathcal{O}(\xi^{12})$,
    \item $\dot{\mathcal{U}}_1 ^e= \mathcal{H}_3\{\bm{0}\} + \mu^2\mathcal{O}(\xi^{12})$,
    \item $\dot{\mathcal{V}}_1 ^e= \mathcal{H}_4\{\bm{0}\} + \mu^2\mathcal{O}(\xi^{12})$,
\end{itemize}
where $\bm{\mathcal{H}} = (\mathcal{H}_1,\mathcal{H}_2,\mathcal{H}_3,\mathcal{H}_4).$

\begin{lemma} \label{le:H0}
With the same hypotheses of Lemma \ref{le:fixpont}, the value of $\bm{\mathcal{H}}\{\bm{0}\}(\hat{\mathcal{T}}_0^*)=\bm{\mathcal{H}}\{\bm{0}\}(\pi)$ is given by
    \begin{equation}
\begin{aligned}
    {\mathcal{H}}_1\{\bm{0}\}(\hat{\mathcal{T}}_0^*) &= \mathcal{U}_6^e(\hat{\mathcal{T}}_0^*)\xi^6 + \frac{\mu}{n}\mathcal{O}(\xi^8),\\
    {\mathcal{H}}_2\{\bm{0}\}(\hat{\mathcal{T}}_0^*) &= \mathcal{V}_6^e(\hat{\mathcal{T}}_0^*)\xi^6 + \frac{\mu}{n}\mathcal{O}(\xi^8) ,\\
    {\mathcal{H}}_3\{\bm{0}\}(\hat{\mathcal{T}}_0^*) &= \dot{\mathcal{U}}_6^e(\hat{\mathcal{T}}_0^*)\xi^6 + \mu\mathcal{O}(\xi^8) ,\\
    {\mathcal{H}}_4\{\bm{0}\}(\hat{\mathcal{T}}_0^*) &= \dot{\mathcal{V}}_6^e(\hat{\mathcal{T}}_0^*)\xi^6 + \mu\mathcal{O}(\xi^8).
\end{aligned}
\end{equation}
where $\bm{\mathcal{U}}_6^e(\hat{\mathcal{T}}_0^*) = (\mathcal{U}_6^e,\mathcal{V}_6^e,\dot{\mathcal{U}}_6^e,\dot{\mathcal{V}}_6^e)(\hat{\mathcal{T}}_0^*)$ are the coefficients of $\bm{\mathcal{U}}_1^e$ of order 6 in $\xi$ evaluated at $\hat{\mathcal{T}}_0^*$. They are given by:
\begin{equation}
\begin{aligned}
    \mathcal{U}_6^e(\hat{\mathcal{T}}_0^*)  &= -\frac{15(-1)^n\mu\pi\cos\theta_0(2\cos^4\theta_0 - 1)}{2n},\\
    \mathcal{V}_6^e(\hat{\mathcal{T}}_0^*) &=  -\frac{15(-1)^n\mu\pi\sin\theta_0(2\sin^4\theta_0 - 1)}{2n},\\
    \dot{\mathcal{U}}_6^e(\hat{\mathcal{T}}_0^*) & =  -\frac{4(-1)^n\mu\cos\theta_0(2\cos^4\theta_0 - 1)}{n}  ,\\
    \dot{\mathcal{V}}_6^e(\hat{\mathcal{T}}_0^*) & = -\frac{4(-1)^n\mu\sin\theta_0(2\sin^4\theta_0 - 1)}{n}.
\end{aligned}
\end{equation}

\end{lemma}
\begin{proof}
    See Appendix \ref{app:Lema5}.
\end{proof}

With this notation, we have
\begin{itemize}[topsep=-0.5ex]
    \item $\mathcal{U}^e(\hat{\mathcal{T}}_0^*) = \mathcal{U}_0^e(\hat{\mathcal{T}}_0^*) + \mathcal{U}_6^e(\hat{\mathcal{T}}_0^*)\xi^6 + \cfrac{\mu}{n}\mathcal{O}(\xi^8)$,
    \item $\mathcal{V}^e(\hat{\mathcal{T}}_0^*) = \mathcal{V}_0^e(\hat{\mathcal{T}}_0^*) + \mathcal{V}_6^e(\hat{\mathcal{T}}_0^*)\xi^6 + \cfrac{\mu}{n}\mathcal{O}(\xi^8)$,
    \item $\dot{\mathcal{U}}^e(\hat{\mathcal{T}}_0^*) = \dot{\mathcal{U}}_0^e(\hat{\mathcal{T}}_0^*) + \dot{\mathcal{U}}_6^e(\hat{\mathcal{T}}_0^*)\xi^6 + \mu\mathcal{O}(\xi^8)$,
    \item $\dot{\mathcal{V}}^e(\hat{\mathcal{T}}_0^*) = \dot{\mathcal{V}}_0^e(\hat{\mathcal{T}}_0^*) + \dot{\mathcal{V}}_6^e(\hat{\mathcal{T}}_0^*)\xi^6 + \mu\mathcal{O}(\xi^8)$.
\end{itemize}

The time needed to reach the $n$ minimum in the distance with the first primary can be obtained from the following Lemma:
\begin{lemma}
With the same hypotheses of Lemma \ref{le:fixpont}, the time $\hat{\mathcal{T}}^*$ needed for the ejection solution $\bm{\mathcal{U}}^e $ to reach the $n$ minimum in the distance with the first primary is given by
$\hat{\mathcal{T}}^*=\hat{\mathcal{T}}_0^*+\hat{\mathcal{T}}_1^*$, where $\hat{\mathcal{T}}_0^*=\pi$ and: 
    \begin{equation}
    \hat{\mathcal{T}}_1^* 
    = \frac{15\mu\pi(3\cos(4\theta_0) + 1)}{8n^2}\xi^6 + \frac{\mu}{n^2}\mathcal{O}(\xi^8).
\end{equation}
\end{lemma}

\begin{proof}
    In order to compute the $n$ minimum in the distance with the first primary we have to solve
    \[
    \begin{aligned}
    0 &=
    \left(\mathcal{U}^e\dot{\mathcal{U}}^e + \mathcal{V}^e\dot{\mathcal{V}}^e\right)(\hat{\mathcal{T}}^*)\\ 
    &=
    \left([\mathcal{U}_0^e + \mathcal{U}_1^e]
    [\dot{\mathcal{U}}_0^e + \dot{\mathcal{U}}_1^e] + 
    [\mathcal{V}_0^e + \mathcal{V}_1^e]
    [\dot{\mathcal{V}}_0^e + \dot{\mathcal{V}}_1^e]\right)(\hat{\mathcal{T}}^*)\\
    & =
    \left(\mathcal{U}_0^e\dot{\mathcal{U}}_0^e 
    + \mathcal{V}_0^e\dot{\mathcal{V}}_0^e\right)(\hat{\mathcal{T}}^*)
    + 
    \xi^6\left(\mathcal{U}_0^e\dot{\mathcal{U}}_6^e 
    +\mathcal{U}_6^e\dot{\mathcal{U}}_0^e 
    + \mathcal{V}_0^e\dot{\mathcal{V}}_6^e
    +\mathcal{V}_6^e\dot{\mathcal{V}}_0^e\right)(\hat{\mathcal{T}}^*) + \mu\mathcal{O}(\xi^8)\\
    &\hspace{5mm}
    +
    \xi^{12}\left(\mathcal{U}_6^e\dot{\mathcal{U}}_6^e + \mathcal{V}_6^e\dot{\mathcal{V}}_6^e\right)(\mathcal{T}^*) + \frac{\mu^2}{n}\mathcal{O}(\xi^{14})\\
   & =
    \left(\mathcal{U}_0^e\dot{\mathcal{U}}_0^e 
    + \mathcal{V}_0^e\dot{\mathcal{V}}_0^e\right)(\hat{\mathcal{T}}^*)
    + 
    \xi^6\left(\mathcal{U}_0^e\dot{\mathcal{U}}_6^e 
    +\mathcal{U}_6^e\dot{\mathcal{U}}_0^e 
    + \mathcal{V}_0^e\dot{\mathcal{V}}_6^e
    +\mathcal{V}_6^e\dot{\mathcal{V}}_0^e\right)(\hat{\mathcal{T}}^*) + \mu\mathcal{O}(\xi^8)\\
    & = 
   \left(\mathcal{U}_0^e\dot{\mathcal{U}}_0^e 
    + \mathcal{V}_0^e\dot{\mathcal{V}}_0^e\right)(\hat{\mathcal{T}}^*_0+\hat{\mathcal{T}}^*_1)
    +  
    \xi^6\left(\mathcal{U}_0^e\dot{\mathcal{U}}_6^e 
    + \mathcal{U}_6^e \dot{\mathcal{U}}_0^e 
    + \mathcal{V}_0^e\dot{\mathcal{V}}_6^e
    + \mathcal{V}_6^e \dot{\mathcal{V}}_0^e\right)(\hat{\mathcal{T}}^*_0+\hat{\mathcal{T}}^*_1)\\
    &\hspace{5mm}
    + \mu\mathcal{O}(\xi^{8})\\
    & = 
    \left(\mathcal{U}_0^e\dot{\mathcal{U}}_0^e + \mathcal{V}_0^e\dot{\mathcal{V}}_0^e\right)(\hat{\mathcal{T}}_0^*)
    + 
    \hat{\mathcal{T}}_1^*\left(\mathcal{U}_0^e\ddot{\mathcal{U}}_0^e + \dot{\mathcal{U}}_0^{e^2} +\mathcal{V}_0\ddot{\mathcal{V}}_0
    + \dot{\mathcal{V}}_0^{e^2}\right)(\hat{\mathcal{T}}_0^*)\\
    & \hspace{5mm}
    +
    \xi^6\left(\mathcal{U}_0^e\dot{\mathcal{U}}_6^e 
    + \mathcal{U}_6^e \dot{\mathcal{U}}_0^e 
    + \mathcal{V}_0^e\dot{\mathcal{V}}_6^e
    + \mathcal{V}_6^e \dot{\mathcal{V}}_0^e\right)(\hat{\mathcal{T}}_0^*)
    + \mu\mathcal{O}(\xi^8)\\
    & = 
    \hat{\mathcal{T}}_1^*n^2
    +
    \xi^6\left(
    \mathcal{U}_6^e\dot{\mathcal{U}}_0^e
    +\mathcal{V}_6^e\dot{\mathcal{V}}_0^e\right)(\hat{\mathcal{T}}_0^*)
    + \mu\mathcal{O}(\xi^8),
    \end{aligned}
\]
and therefore we have
\begin{equation}
    \mathcal{T}_1^* = \frac{\left(\mathcal{U}_6^e\dot{\mathcal{U}}_0^e
    +\mathcal{V}_6^e\dot{\mathcal{V}}_0^e\right)(\mathcal{T}_0^*)}{n^2}\xi^6 + \frac{\mu}{n^2}\mathcal{O}(\xi^8)
     = \frac{15\mu\pi(3\cos(4\theta_0) + 1)}{8n^2}\xi^6 + \frac{\mu}{n^2}\mathcal{O}(\xi^8).
\end{equation}

\end{proof}

Finally the angular momentum at $\hat{\mathcal{T}}^*$ is given by
\begin{lemma}
With the same hypotheses of Lemma \ref{le:fixpont},  the angular momentum of the ejection solution $\bm{\mathcal{U}}^e$ at time $\hat{\mathcal{T}}^*$ is given by:
\[
    \mathcal{M}(n,\theta_0) = -\frac{15\mu\pi\sin(4\theta_0)}{4}\xi^6 + \mu\mathcal{O}(\xi^8).
\]
\end{lemma}

\begin{proof}
    \[
    \begin{aligned}
    \mathcal{M}(n,\theta_0)
    & = \left(\mathcal{U}^e\dot{\mathcal{V}}^e-\mathcal{V}^e\dot{\mathcal{U}}^e\right)(\hat{\mathcal{T}}^*)\\
    & = \left(\mathcal{U}_0^e\dot{\mathcal{V}}_0^e 
    - \mathcal{V}_0^e\dot{\mathcal{U}}_0^e\right)(\hat{\mathcal{T}}^*)
    + 
    \xi^6\left(\mathcal{U}_0^e\dot{\mathcal{V}}_6^e 
    +\mathcal{U}_6^e\dot{\mathcal{V}}_0^e 
    - \mathcal{V}_0^e\dot{\mathcal{U}}_6^e
    - \mathcal{V}_6^e\dot{\mathcal{U}}_0^e\right)(\hat{\mathcal{T}}^*) + \mu\mathcal{O}(\xi^8)\\
    & = \left(\mathcal{U}_0^e\dot{\mathcal{V}}_0^e 
    - \mathcal{V}_0^e\dot{\mathcal{U}}_0^e\right)(\hat{\mathcal{T}}_0^*)
    + 
    \hat{\mathcal{T}}_1^*
    \left(\mathcal{U}_0^e\ddot{\mathcal{V}}_0^e - \mathcal{V}_0^e\ddot{\mathcal{U}}_0^e
    \right)
    (\hat{\mathcal{T}}_0^*)
    + \frac{\mu^2}{n}\mathcal{O}(\xi^{12})\\
    & \hspace{5mm }+ \xi^6\left(\mathcal{U}_0^e\dot{\mathcal{V}}_6^e 
    +\mathcal{U}_6^e\dot{\mathcal{V}}_0^e 
    - \mathcal{V}_0^e\dot{\mathcal{U}}_6^e
    - \mathcal{V}_6^e\dot{\mathcal{U}}_0^e\right)(\hat{\mathcal{T}}^*_0) + \mu\mathcal{O}(\xi^8)\\
    & = \left(\mathcal{U}_0^e\dot{\mathcal{V}}_0^e 
    - \mathcal{V}_0^e\dot{\mathcal{U}}_0^e\right)(\hat{\mathcal{T}}_0^*)
    + 
    \xi^6\hat{\mathcal{T}}_6^*
    \left(\mathcal{U}_0^e\ddot{\mathcal{V}}_0^e - \mathcal{V}_0^e\ddot{\mathcal{U}}_0^e
    \right)
    (\hat{\mathcal{T}}_0^*)
   \\
    & \hspace{5mm }+ \xi^6\left(\mathcal{U}_0^e\dot{\mathcal{V}}_6^e 
    +\mathcal{U}_6^e\dot{\mathcal{V}}_0^e 
    - \mathcal{V}_0^e\dot{\mathcal{U}}_6^e
    - \mathcal{V}_6^e\dot{\mathcal{U}}_0^e\right)(\hat{\mathcal{T}}^*_0) + \mu\mathcal{O}(\xi^8)\\
    & = 
    \xi^6\left( 
    \mathcal{U}_6^e\dot{\mathcal{V}}_0^e 
    - \mathcal{V}_6^e\dot{\mathcal{U}}_0^e\right)(\hat{\mathcal{T}}^*_0) + \mu\mathcal{O}(\xi^8)\\
    & = -\frac{15\mu\pi\sin(4\theta_0)}{4}\xi^6 + \mu\mathcal{O}(\xi^8) .
    \end{aligned}
\]
\end{proof}

In this way, applying the Implicit Function Theorem, we have that for $\xi\geq0$ small enough we obtain that $\mathcal{M}(n,\theta_0)$ has four and only four roots in $[0,\pi)$ given by
\[
    \theta_0 = \frac{\pi m}{4} + \mathcal{O}(\xi^2),\quad\quad m=0,1,2,3
\]
regardless of the values of the mass parameter $\mu$ and $n$. We can characterize this $n$-EC orbits in the same way as in Theorem \ref{th:maintheorem2}.

This concludes the proof of Theorem \ref{maintheoremN}.

\section{Results for the Hill problem}

As we have seen in Section \ref{sec:Hill}, Hill problem is a limit case of RTBP. In this way,
the results obtained in the previous sections can easily be extrapolated to the case of Hill  problem. In particular, if in \eqref{eq:LChill} we introduce the new variables
\begin{equation}\label{cvarnouUVHill}
    \left\{
    \begin{aligned}
    	u_h & = \sqrt{\frac{2}{K}}U_h,\\
    	v_h & = \sqrt{\frac{2}{K}}V_h,\\
    	\tau &= 2\sqrt{K} s,\\
	\end{aligned}
    \right.
\end{equation}
we obtain the system 
\begin{equation} \label{eq:sistemaHillCanvi1}
	\left\{
	\begin{aligned}
	\ddot{U} &= - U + \frac{8(U^2 + V^2)\dot{V}}{K^{3/2}} + \frac{12\left[2\left(U^4-2U^2V^2-V^4\right) + (U^2+V^2)^2\right]U}{K^3},\\[1.5ex]
	\ddot{V} &= - V - \frac{8(U^2 + V^2)\dot{U}}{K^{3/2}} + \frac{12\left[2\left(V^4-2U^2V^2-U^4\right) + (U^2+V^2)^2\right]V}{K^3},
	\end{aligned}
	\right.
\end{equation}
which is the same system of equations that we have obtained in \eqref{SistemaSerie} imposing now $\mu = 1$ and recalling that $\varepsilon = 1/\sqrt{K}$. So we already know the 
solution of system \eqref{eq:sistemaHillCanvi1} which is the one obtained for system
\eqref{SistemaSerie} with $\mu=1$.

In this way, using the extra symmetry \eqref{eq:simetriesHill2} of the Hill problem we obtain the following Corollary of Theorem \ref{th:maintheorem2}:

\vspace{1mm}
\begin{cor}
    In the Hill problem, for all $n\in\mathbb{N}$, there exists a  $\hat K(n)$ such that for  $K\ge \hat K (n)$  there exist exactly four $n$-EC orbits, which can be characterized by:
    \vspace{-3mm}
        \begin{itemize}
                \item Two $n$-EC orbits themselves symmetric with respect to the $x$ axis and one  symmetric to the other over the $y$ axis.
                \vspace{-1mm}
                \item Two $n$-EC orbits themselves symmetric with respect to the $y$ axis and one  symmetric to the other over the $x$ axis.
        \end{itemize}
        The respective families (when varying $K$) are labelled by $\alpha_n$, $\gamma_n$, $\beta_n$ and $\delta_n$.
\end{cor}

\begin{figure}[ht!]
        \centering
        \includegraphics[trim= 2mm 0 5mm 0mm, clip,width=0.325\textwidth]{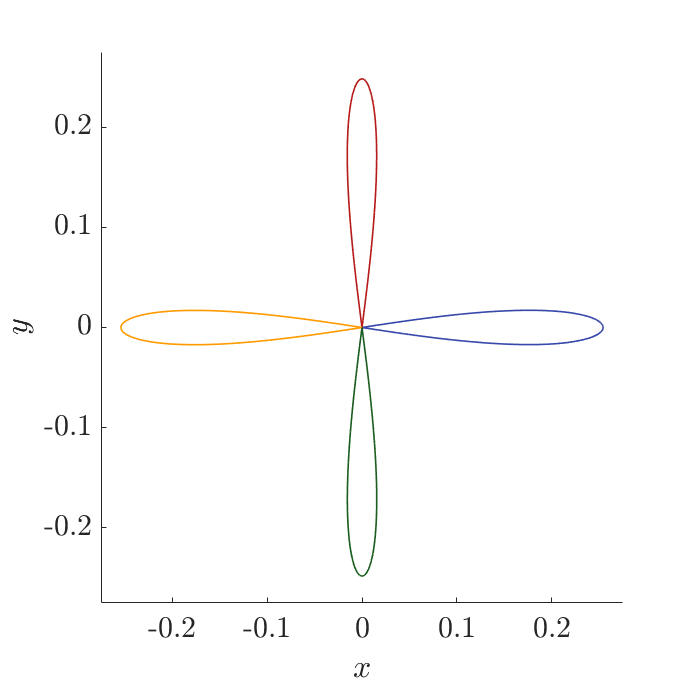}
        \includegraphics[trim= 2mm 0 5mm 0mm, clip,width=0.325\textwidth]{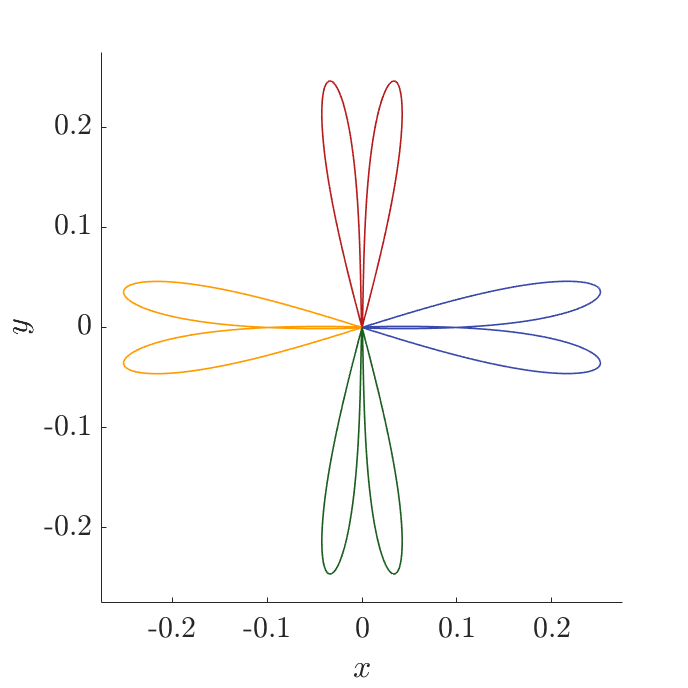}
        \includegraphics[trim= 2mm 0 5mm 0mm, clip,width=0.325\textwidth]{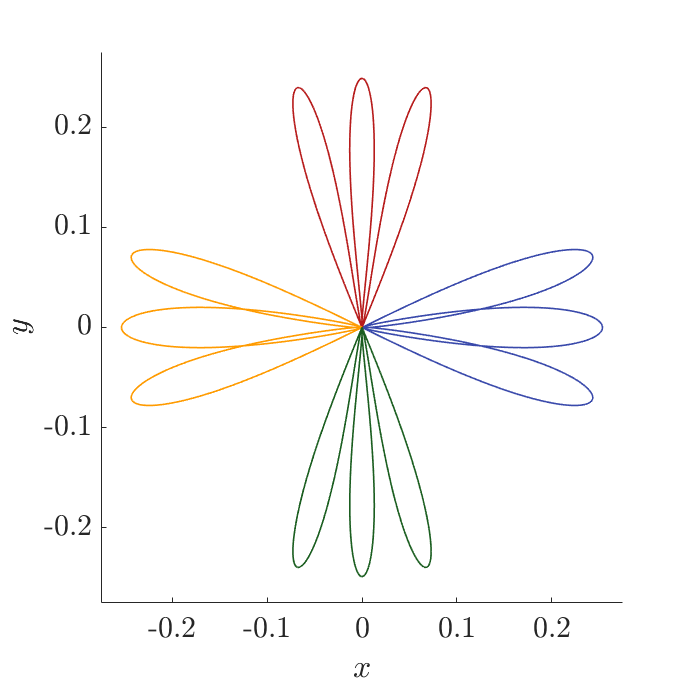}
        \caption{Trajectories of the four $n$-EC orbits $\alpha_n$ (yellow), $\beta_n$ (green), $\gamma_n$ (blue) and $\delta_n$ (red) for $n=1,\,2,\,3$ (from left to right) and $K=8$.
        }
        \label{fig:hill-nECos}
\end{figure}

It is important to note that the proof is exactly the same with the observation, as we have said before, that the families of orbits that were symmetric with respect to the $x$ axis in the RTBP ($\alpha_n$ and $\gamma_n$) are now also symmetric one of the other with respect to the $y$ axis, and the families that were symmetric one of the other in the restricted problem ($\beta_n$ and $\delta_n$) are now also symmetric themselves with respect to the $y$ axis (see Figure \ref{fig:hill-nECos}).

Furthermore, thanks to the fact that the polynomials $\bar{P}_{2k} $ and $\bar{Q}_{2k}$ disappear, it is not necessary to consider an expansion in terms of $\varepsilon = 1/\sqrt{K} $ and it can be considered directly an expansion on $\epsilon =  1/K^{3/2}$.

\begin{figure}[ht!]
    \centering
    \includegraphics[width=0.445\textwidth]{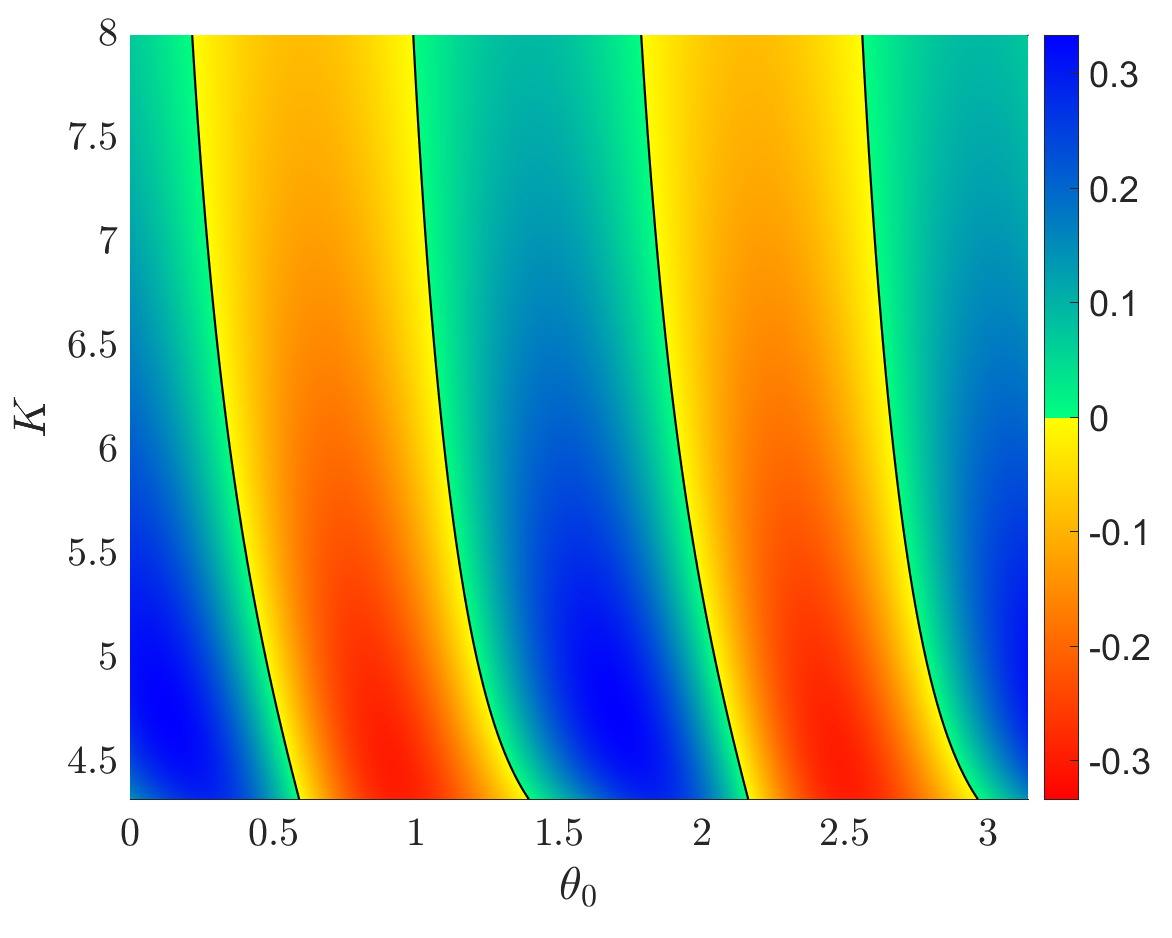}
    \includegraphics[width=0.445\textwidth]{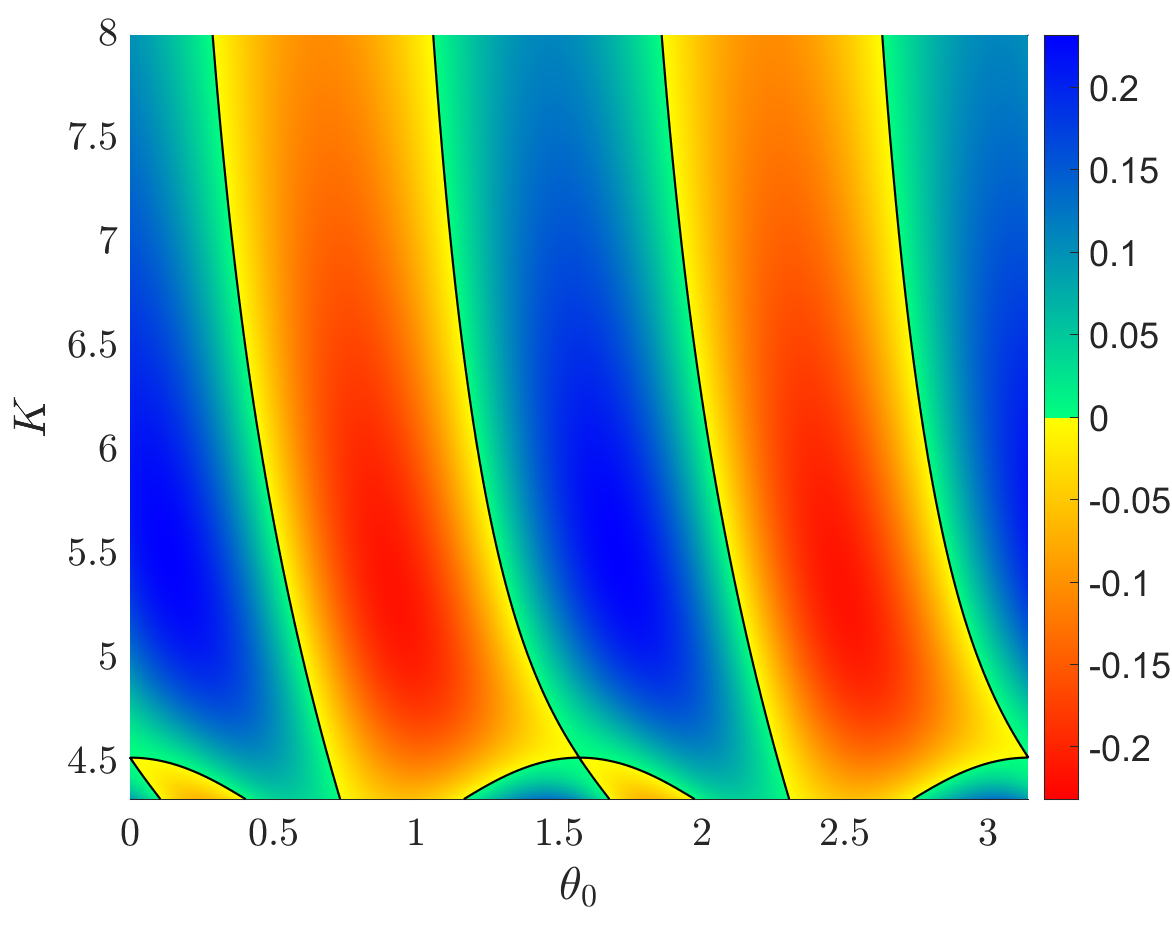}\\
    \vspace{-2mm}
    \includegraphics[width=0.445\textwidth]{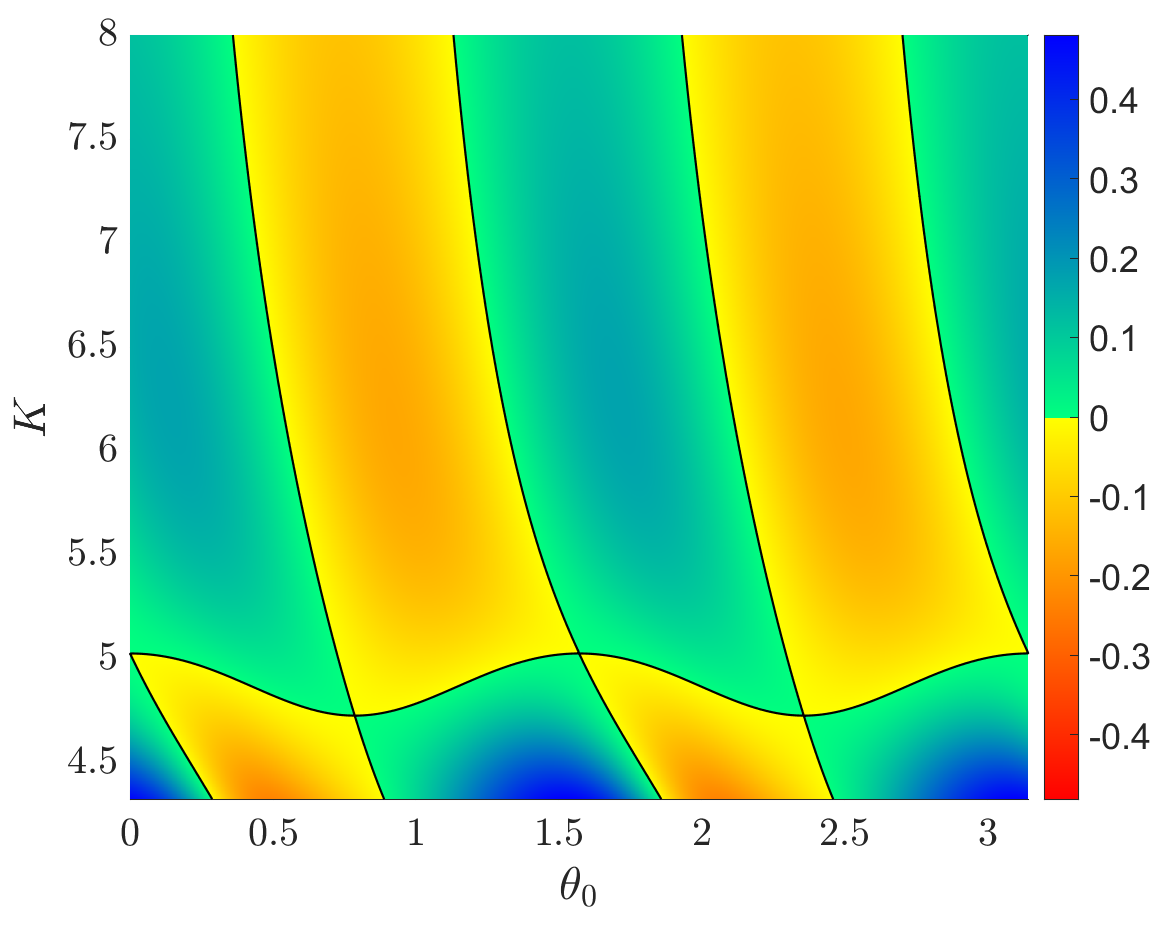}
    \includegraphics[width=0.445\textwidth]{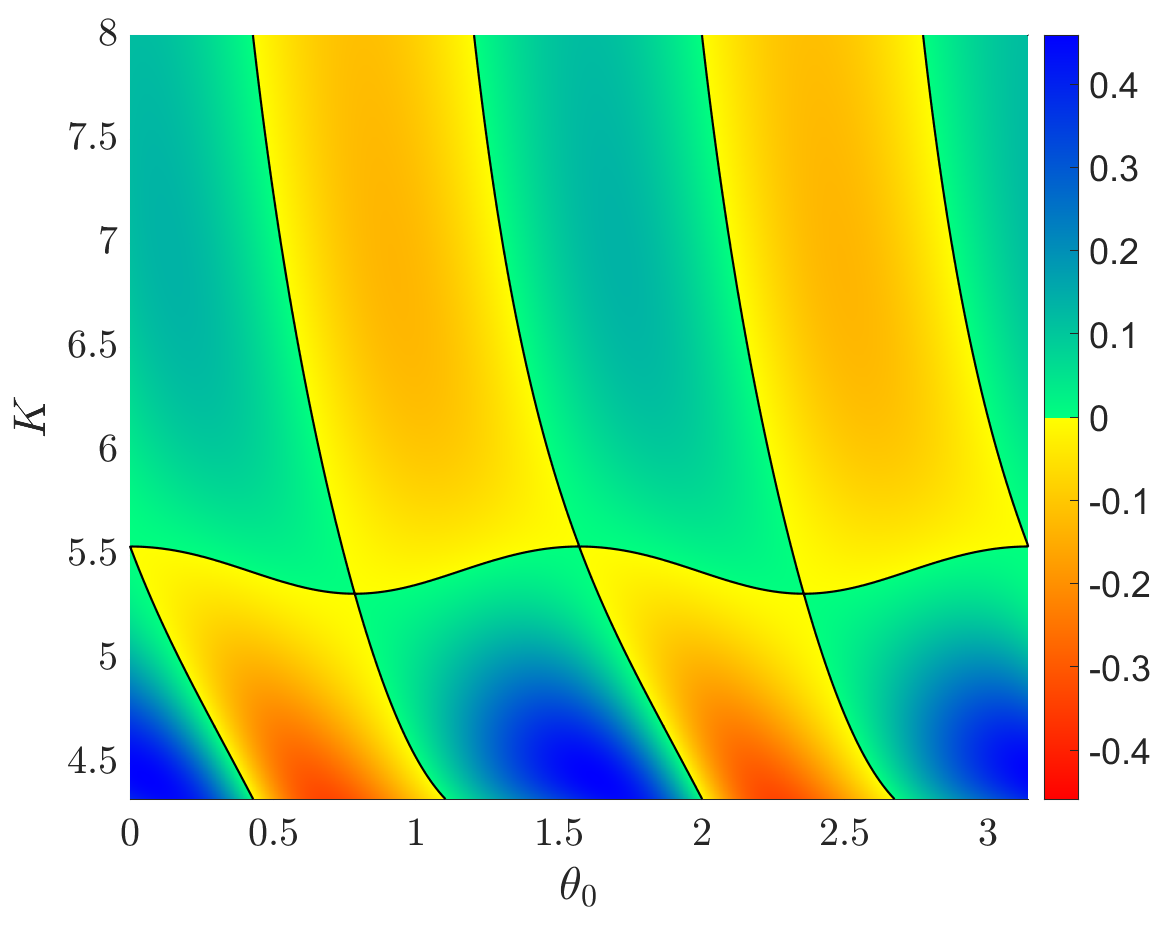}\\
    \vspace{-2mm}
    \includegraphics[width=0.445\textwidth]{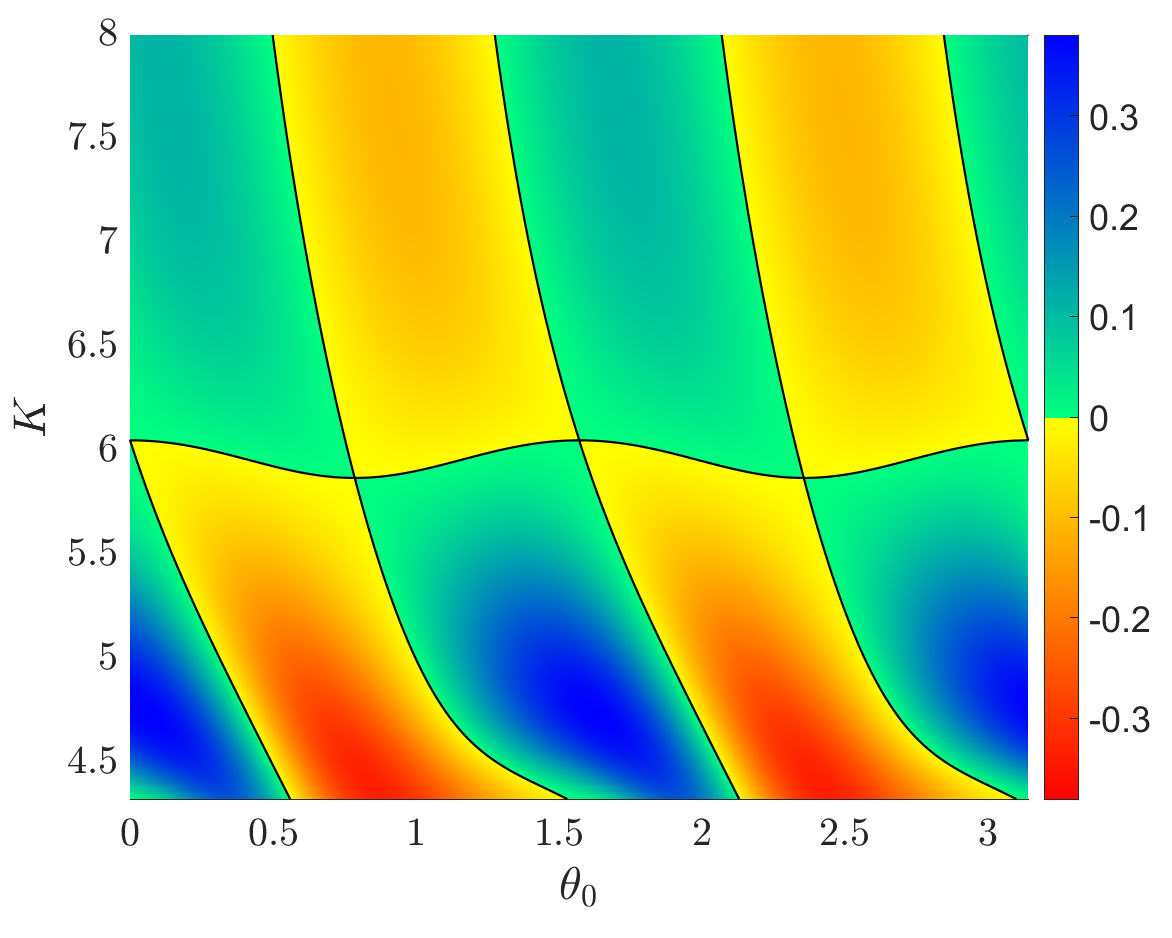}
    \includegraphics[width=0.445\textwidth]{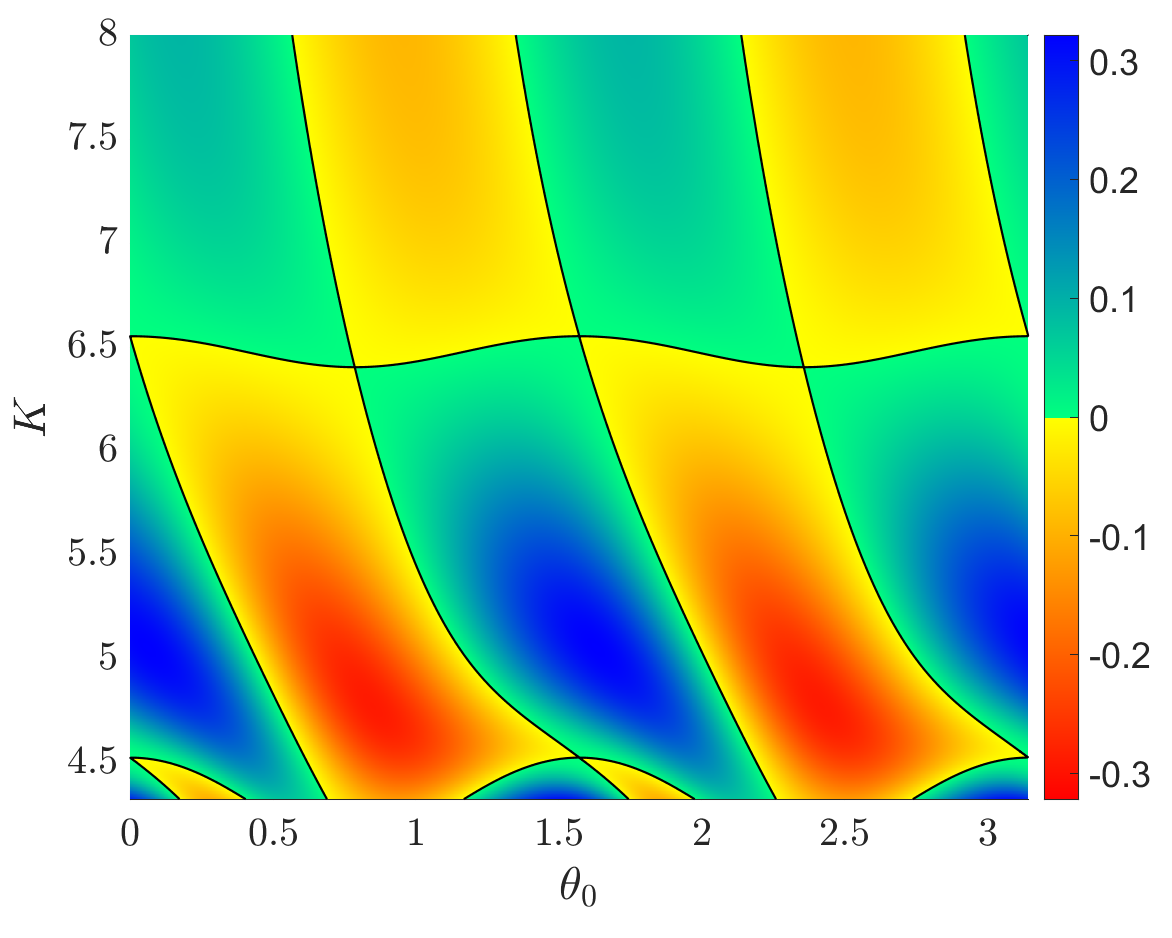}\\
    \vspace{-2mm}
    \includegraphics[width=0.445\textwidth]{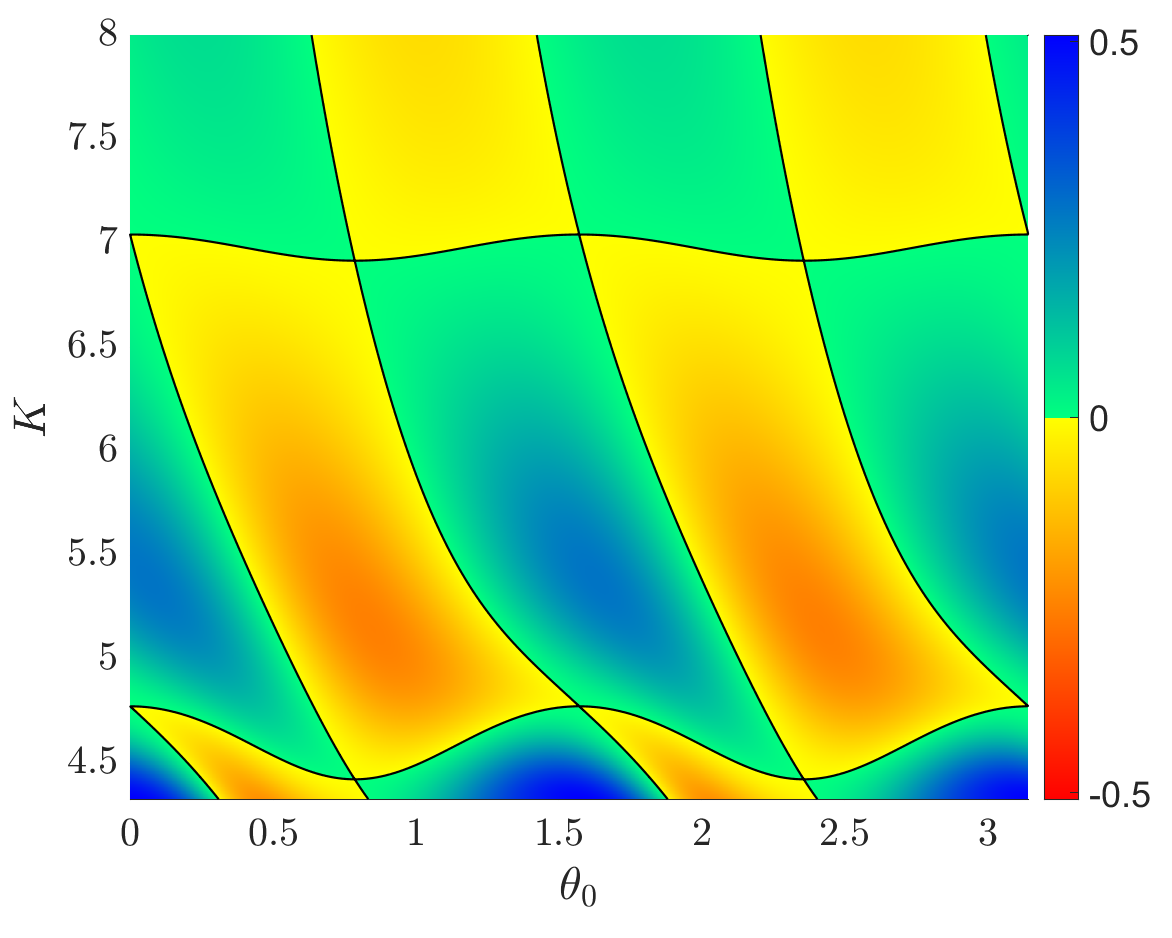}
    \includegraphics[width=0.445\textwidth]{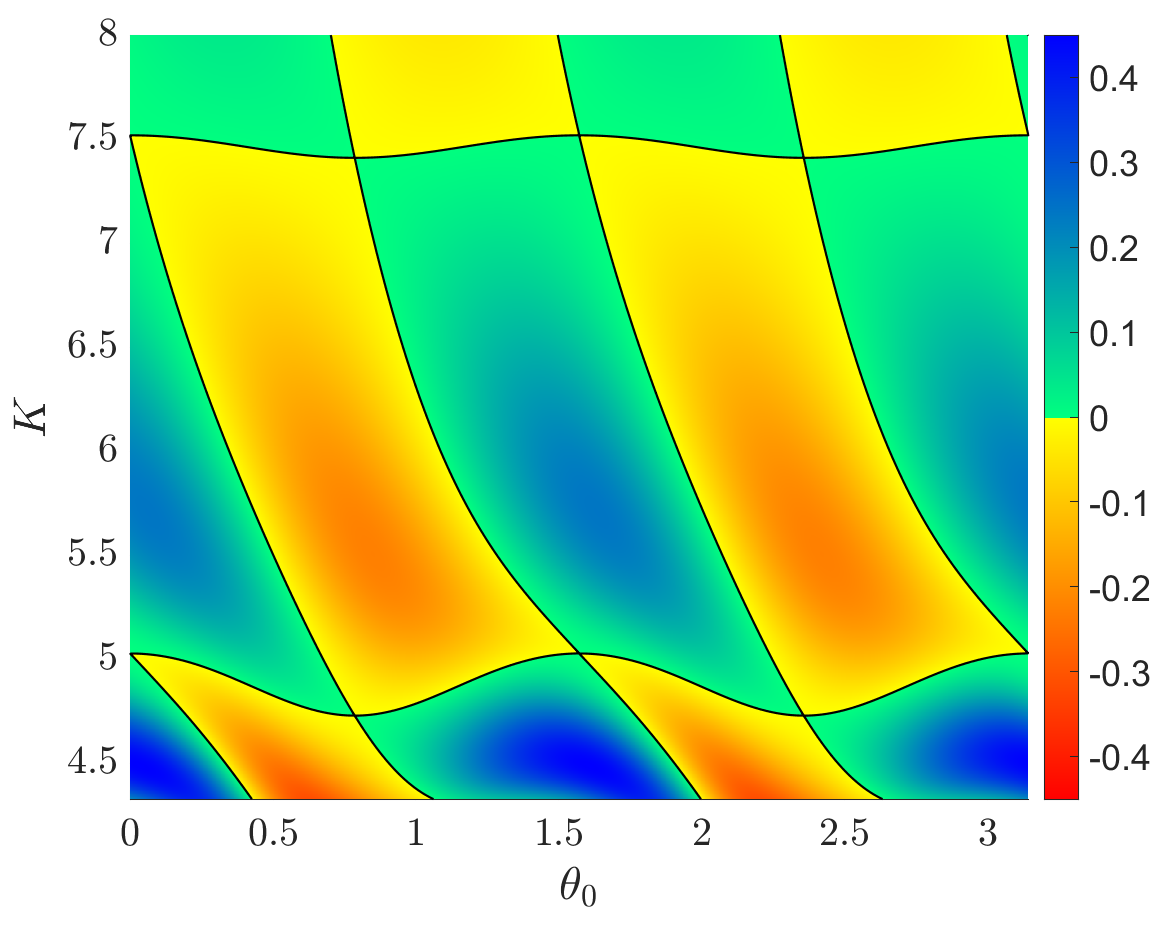}\\
    \vspace{-3mm}
    \caption{Value of the angular momentum of the ejection orbits at the $n$ intersection with $\Sigma_m$ for $K\in[K_{L_1},8]$ and $n=3,...,10$. In black the values corresponding to $n$-EC orbits.}
    \label{fig:evolucioMHill}
\end{figure}

\begin{figure}[ht!]
    \vspace{2mm}
        \centering
        \includegraphics[width=0.99\textwidth]{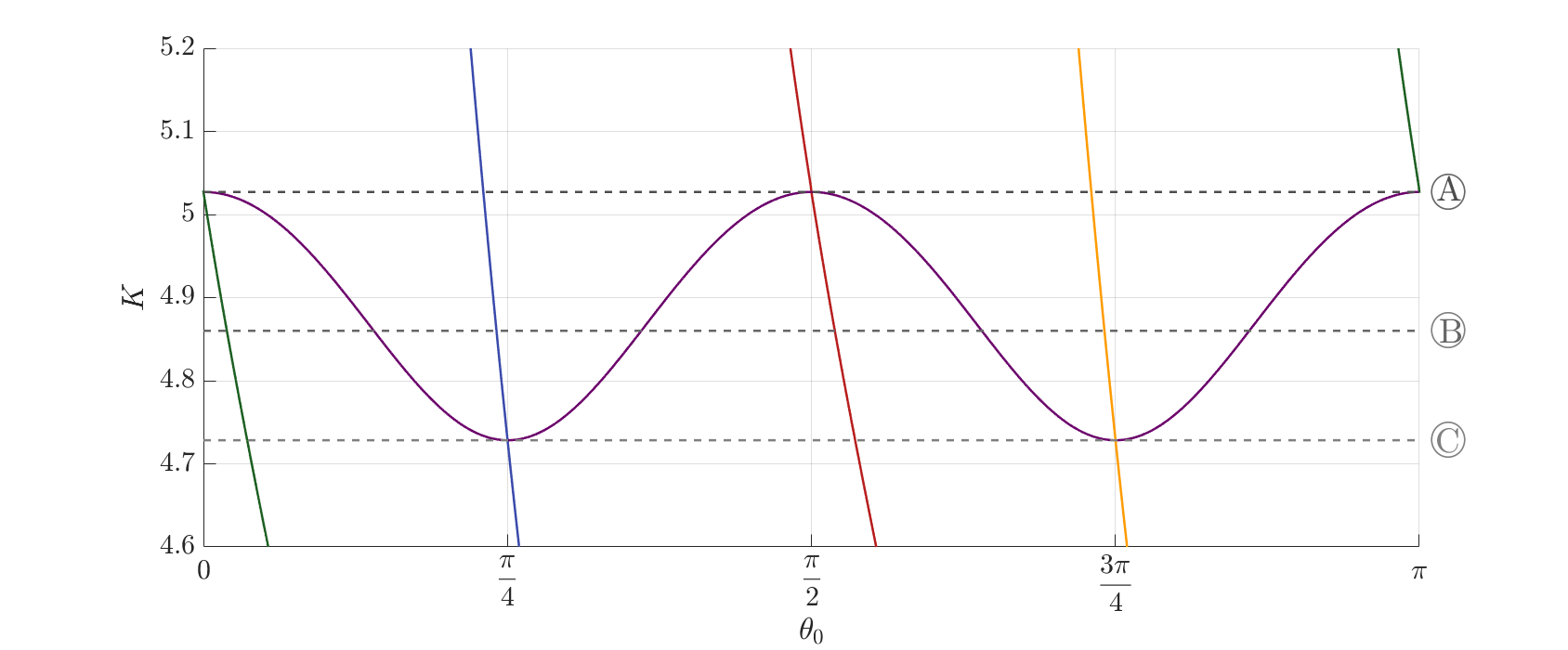}\\
        \includegraphics[trim= 2mm 0 7mm 0mm, clip,width=0.325\textwidth]{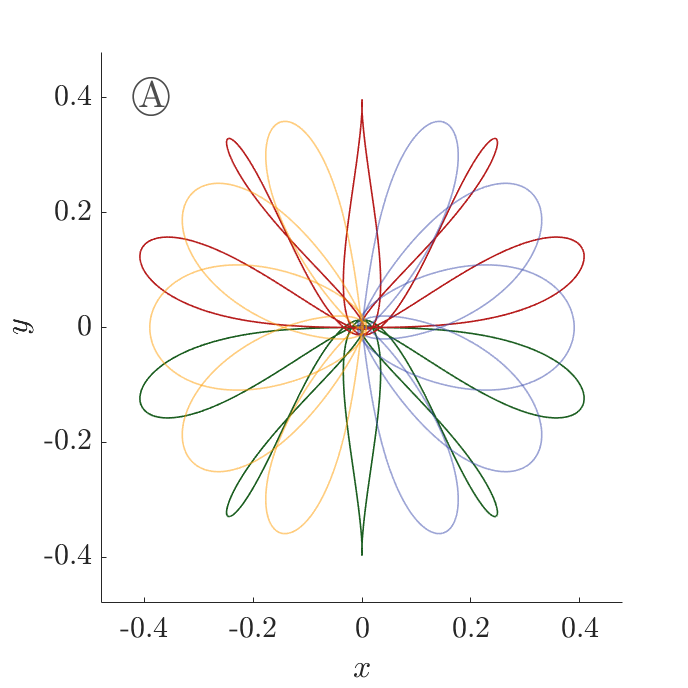}
        \includegraphics[trim= 2mm 0 7mm 0mm, clip,width=0.325\textwidth]{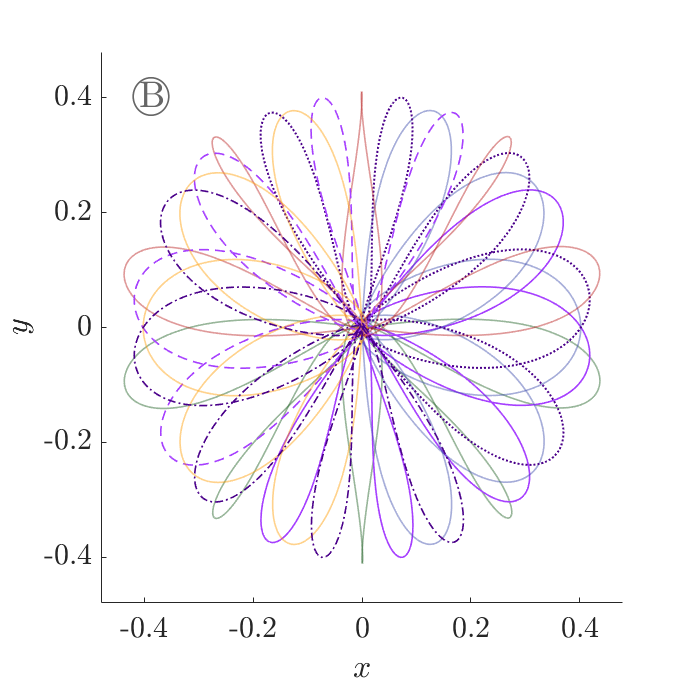}
        \includegraphics[trim= 2mm 0 7mm 0mm, clip,width=0.325\textwidth]{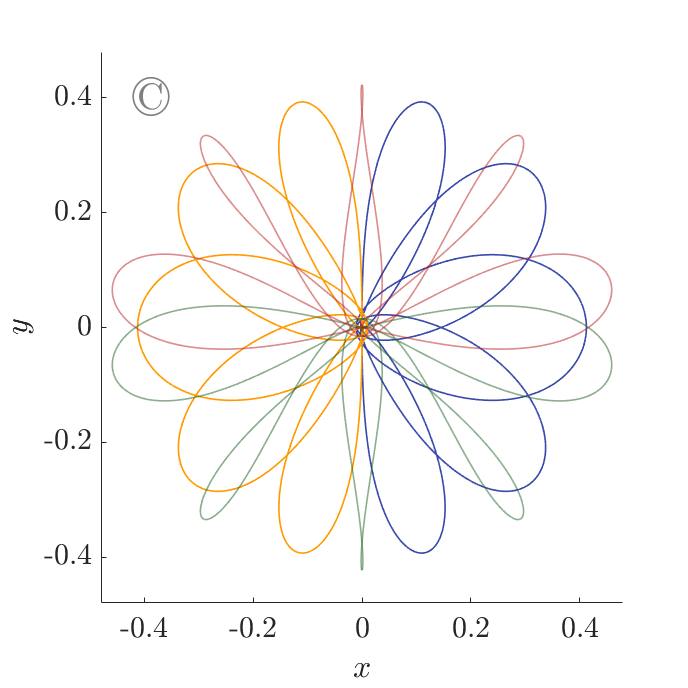}
        \caption{
        Top.
        Initial conditions for the 5-EC orbits corresponding to the families
        $\alpha_n$ (yellow), $\beta_n$ (green), $\gamma_n$ (blue), $\delta_n$ (red) and the new families of orbits (purple) as function of $K$.
        Bottom. The trajectories of the orbits (in correspondence with the previous color) that exist for the values of $K$ denoted previously.
        The values of $K$ correspond to the value of the bifurcation $K\approx5.02714993$ (left), a value where we have eight 5-EC orbits $K=4.86$ (middle) and the value of collapse $K\approx4.72835275$.
        }
        \label{fig:hill-5ECOs-bif}
\end{figure}

Similarly, if we introduce $K = Ln^{2/3}$, that is, we consider the change
\begin{equation}\label{cvarnouUVHill2}
    \left\{
    \begin{aligned}
        u_h & = \frac{\sqrt{2}}{\sqrt{L}n^{1/3}}\mathcal{U}_h,\\
        v_h & = \frac{\sqrt{2}}{\sqrt{L}n^{1/3}}\mathcal{V}_h,\\
        \hat{\mathcal{T}} &= \frac{2\sqrt{L}}{n^{2/3}} s,\\
        \end{aligned}
    \right.
\end{equation}
we obtain the same system of equations as \eqref{eq:noseriesxi} putting $\mu=1$ and considering $\xi = 1/\sqrt{L}$. In this way we can obtain the following corollary of Theorem \ref{maintheoremN}:
\vspace{2mm}

\begin{cor}\label{cor:thFort}
        There exists an $\hat{L}$ such that for $L\geq\hat{L}$ and for any value of $n\in\mathbb{N}$ and $K=Ln^{2/3}$, there exist exactly four $n$-EC orbits, which can be characterized in the same way as the previous corollary.
\end{cor}

In this way, if we do the numerical exploration to compute the $n$-EC orbits that exist for values of $K \geq K_{L}$ (see Figure \ref{fig:evolucioMHill}) we see that, as expected by the Corollary \ref{cor:thFort}, the value of $\hat{K}$ grows with $n$.

Before going into more detail on the value of $\hat{K}$ let us make a few comments about Figure \ref{fig:evolucioMHill}.
It is important to note that thanks to the extra symmetry we could only study the ejection orbits with $\theta_0 \in [0, \pi/2)$, but in order to visualize the evolution of the $n$-EC orbits we will consider the interval $\theta_0 \in [0, \pi)$ in Figure  \ref{fig:evolucioMHill}. In this figure we observe how at least the first new families of $n$-EC orbits that appear are born from two of the original families ($\alpha_n$ and $\gamma_n$, or $\beta_n$ and $\delta_n$) when the angle of ejection $\theta_0$ is $0$ and $\pi/2$ respectively (i.e. $\vartheta_0=0,\pi$) and collapse into the two other original families when the value of $\theta_0$ is $\pi/4 $ and $3\pi/4$ (i.e. $\vartheta_0=\pi/2,3\pi/2$) (see for example Figure \ref{fig:hill-5ECOs-bif}).

\begin{figure}[ht!]
        \centering
        \includegraphics[trim= 2mm 0 7mm 0mm, clip,width=0.4\textwidth]{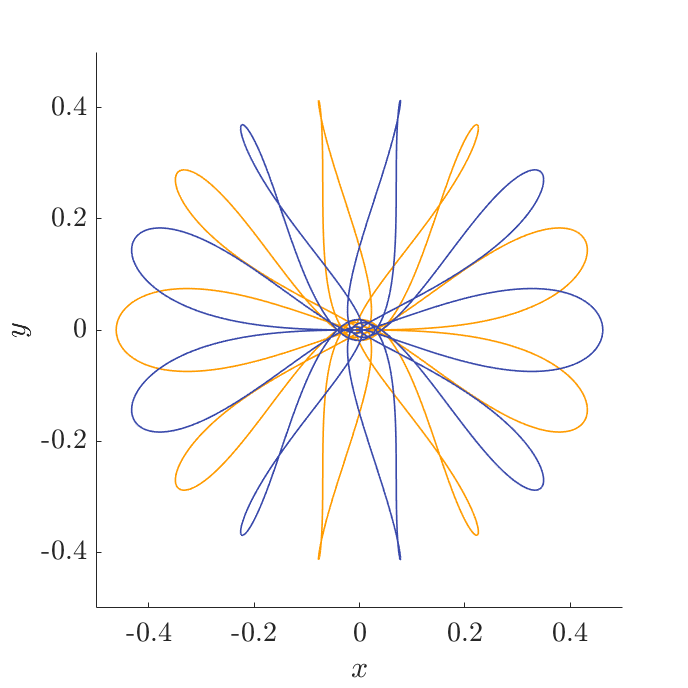}
        \includegraphics[trim= 2mm 0 7mm 0mm, clip,width=0.4\textwidth]{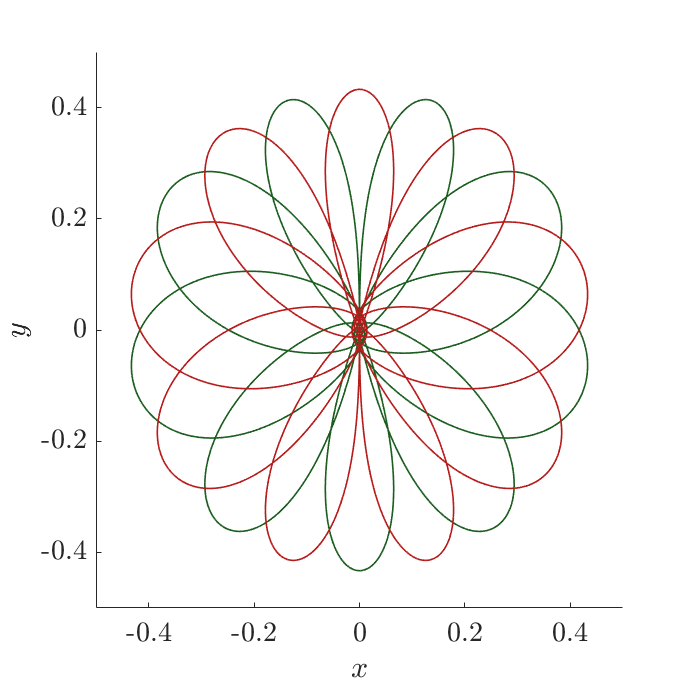}
        \caption{Trajectories of 9-EC periodic orbits associated with $\alpha_9$ (yellow) and $\gamma_9$ (blue) for $K\approx4.77318771$ (left) and $\beta_9$ (green) and $\delta_9$ (red) for $K\approx4.42215362$ (right).
        }
        \label{fig:hill-9ECOs-periodic}
\end{figure}

These respective values are very particular, since when these bifurcations take place we have that the $n$-EC orbits are periodic or are part of a periodic EC orbit. In particular we have:

\begin{itemize}
    \item If the $\theta_0$ of $\beta_n$ is $0$ or $\pi/2$ (therefore $\theta_0$ of $\delta_n$ is $\pi/2$ or $0$) then we have periodic EC orbit formed by $\beta_n$ and $\delta_n$ (see Figure \ref{fig:hill-5ECOs-bif} left). Analogously, if the $\theta_0$ of $\alpha_n$ is $\pi/4$ or $3\pi/4$ (therefore $\theta_0$ of $\gamma_n$ is $3\pi/4$ or $\pi/4$) then we have periodic EC orbit formed by $\alpha_n$ and $\gamma_n$ (see Figure \ref{fig:hill-5ECOs-bif} right).\vspace{2mm}

    \item If the $\theta_0$ of $\beta_n$ is $\pi/4$ or $3\pi/4$ (therefore $\theta_0$ of $\delta_n$ is $3\pi/4$ or $\pi/4$) then $\beta_n$ and $\delta_n$ are periodic $EC$ orbits (see Figure \ref{fig:hill-9ECOs-periodic} right). Analogously, if the $\theta_0$ of $\alpha_n$ is $0$ or $\pi/2$ (therefore $\theta_0$ of $\gamma_n$ is $\pi/2$ or $0$) then $\alpha_n$ and $\gamma_n$ are periodic EC orbits (see Figure \ref{fig:hill-9ECOs-periodic} left).
\end{itemize}

\begin{figure}[ht!]
    \centering
    \includegraphics[trim= 15mm 5mm 15mm 14mm, clip, width=0.9\textwidth]{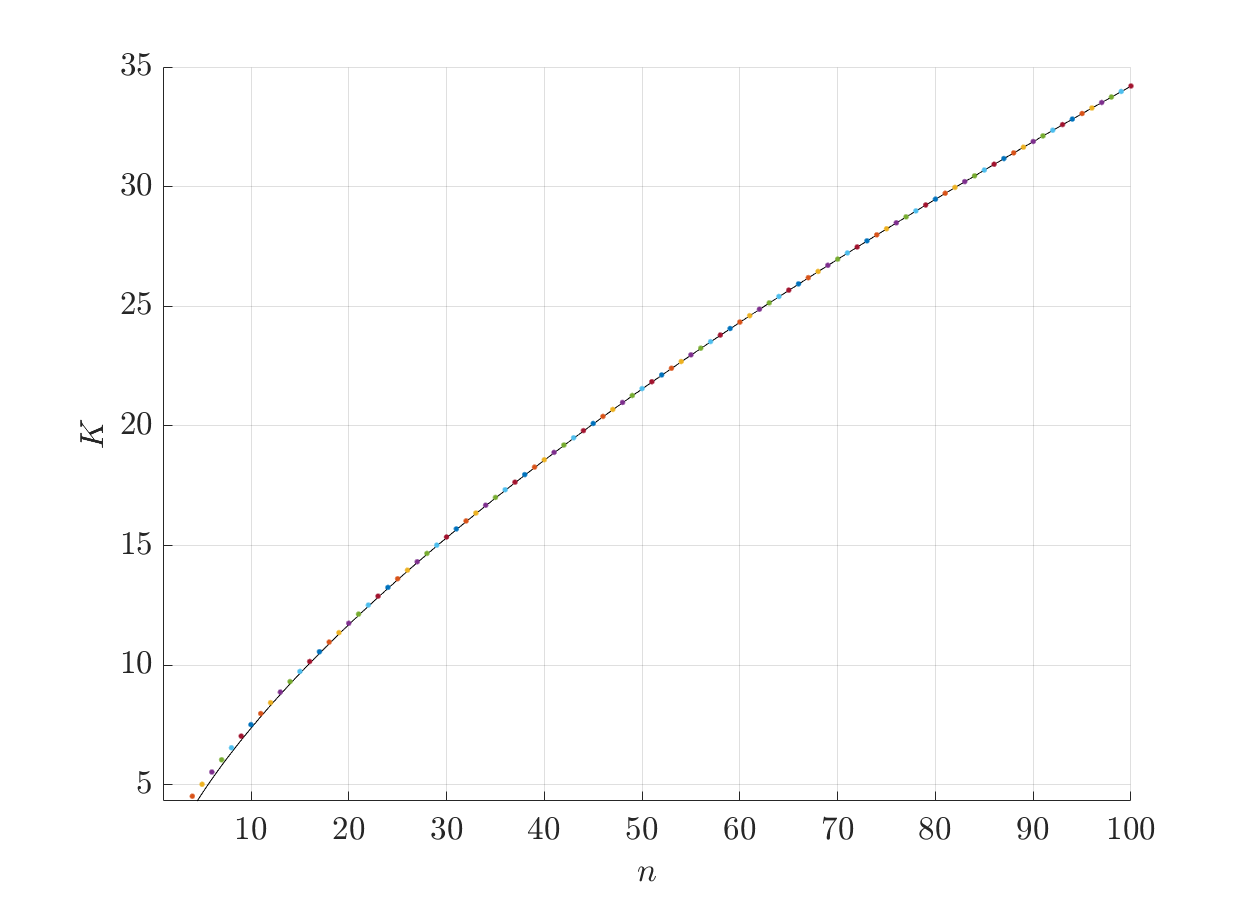}
    \caption{Dots: Values of $\hat{K}(n)$. Black line, curve $Ln^{2/3}$ with $L=2^{2/3}$.}
    \label{fig:nECOsHill}
\end{figure}

We have computed the value  $\hat{K}(n)$ for $n=1,\dots,100$. See Figure \ref{fig:nECOsHill}).
It is important to remark that the numerical value of $\hat{K}(n)$ obtained has the same shape as the analytical bound of the Corollary \ref{cor:thFort}. In particular, if we draw the curve $Ln^{2/3}$ with $L=2^{2/3}$ we can see how it practically matches the value of the numerical bound obtained for $\hat{K}$ (see Figure \ref{fig:nECOsHill}).

\begin{figure}[ht!]
    \centering
    \includegraphics[trim= 15mm 5mm 15mm 10mm, clip, width=0.9\textwidth]{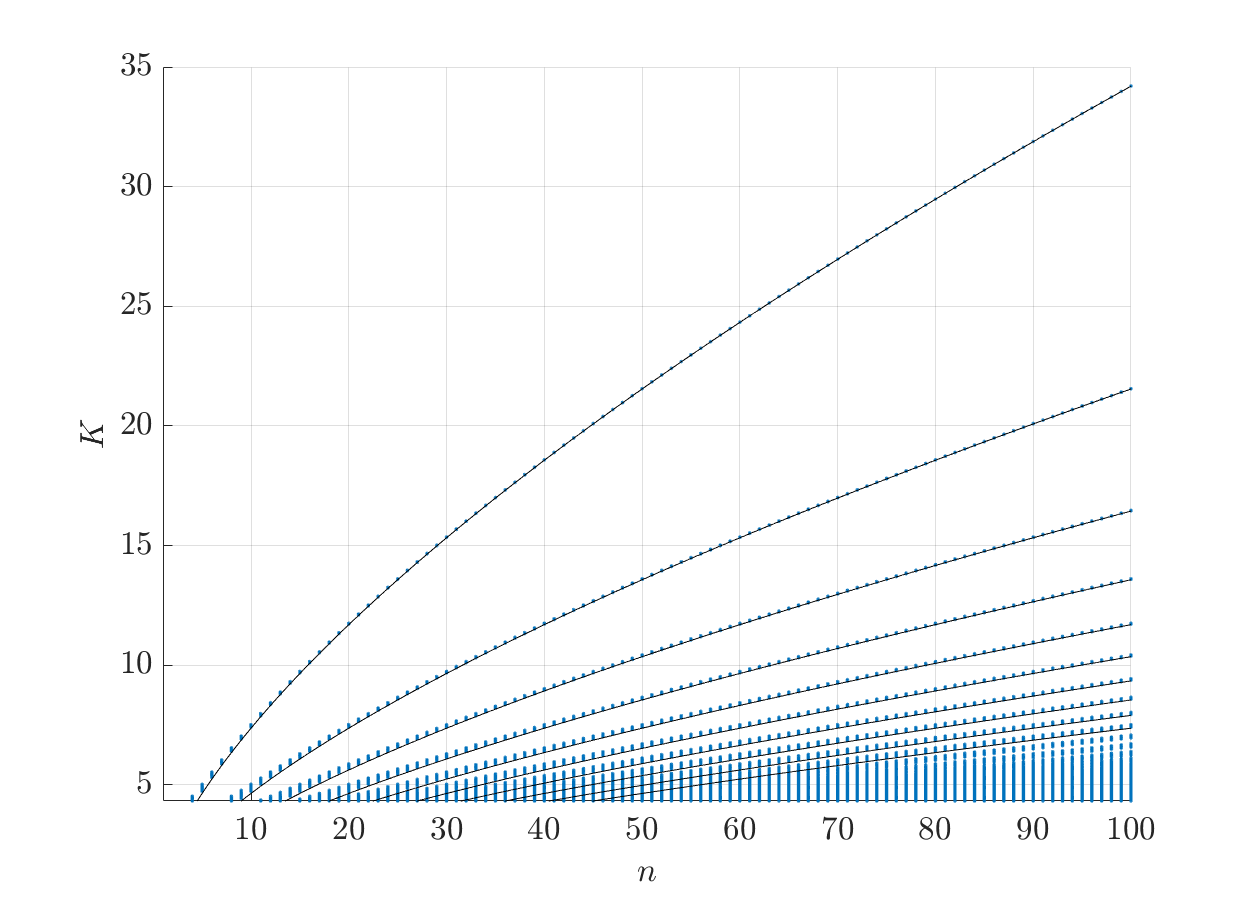}
    \caption{In color values of $K$ where exists more than 4 $n$-EC orbits for $n=1,..,100$. The black lines correspond to the curves $Ln^{2/3}$ with  $L=(2/p)^{2/3}$ and $p = 1, ..., 10$.}
    \label{fig:nECOsHillBif}
\end{figure}

To conclude, we have seen how not only does the value of $\hat{K}(n)$ follow the curve $Ln^{2/3} $ with $L = 2^{2/3}$, but also the successive bifurcations (the values of $K$ where appear new EC orbits) are closely related to the curves $Ln^{2/3}$ with $L = (2/p)^{2/3}$ being $p$ a natural number. In particular, in Figure \ref{fig:nECOsHillBif} we can see how the value of the successive bifurcations coincides with the curves $Ln^{2/3}$ with  $L=(2/p)^{2/3}$ and $p = 1, ..., 10$.

\section*{Data availability}

The datasets generated during and/or analysed during the current study are available from the corresponding author on reasonable request.

\section{Acknowledgements}
This work is supported by the Spanish State Research Agency, through the Severo Ochoa and María de Maeztu Program for Centers and Units of Excellence in R\&D (CEX2020-001084-M). 
T. M-Seara is supported by the Catalan Institution for Research and Advanced Studies via an ICREA Academia Prize 2019 and by the Spanish MINECO/FEDER grant PGC2018-098676-B-100 (AEI/FEDER/UE).
M. Oll\'e and O. Rodríguez were supported by the Spanish MINECO/FEDER grant PGC2018-100928-B-I00 (AEI/FEDER/UE). 
J. Soler was supported by MTM2016-77278-P.

\printbibliography

@article{asetal,
doi	= {10.1038/nature01622},
title	= "{Chaos-assisted capture of irregular moons}",
author	= {Astakhov, S.A.  and Burbanks, A.D. and Wiggins, S. and Farrelly, D.},
journal	= {Nature},
year	= {2003},
volume	= {423/6937},
pages	= {264--267},
}

@article{bo,
title	= "{Sets of collision periodic orbits in the Restricted problem}",
author	= {Bozis, G.},
journal	= {Periodic orbits, stability and resonances,  G.E.O. Giacaglia (eds). Holland: D. Reidel Pub. Co.},
year	= {1970},
pages	= {176--191},
}

@article{bru,
doi	= {10.1103/PhysRevA.55.3730},
title	= "{Hydrogen atom in circularly polarized microwaves: Chaotic ionization via core scattering}",
author	= {Brunello, A.F. and Uzer, T. and Farrelly, D.},
journal	= {Physical Review A},
year	= {1997},
volume	= {55/5},
pages	= {3730--3745},
}

@article{CHLL,
doi	= {10.1017/s0143385700009330},
title	= "{A note on the existence of invariant punctured tori in the planar circular restricted three-body problem}",
author	= {Chenciner, A. and Llibre, J.},
journal	= {Ergodic Theory and Dynamical Systems},
year	= {1988},
volume	= {8/8*},
pages	= {63--72},
}

@article{DPrince,
title	= {A familiy of embedded {R}unge-{K}utta formulae},
doi = {10.1016/0771-050X(80)90013-3},
author	= {Dormand, J.R. and Prince, J.P.},
publisher	= {},
journal	= {J. Computat. App. Math.},
year	= {1980},
volume	= {6/1},
pages	= {19--26},
}

@INPROCEEDINGS{Devaney,
  author = {Devaney, R. L.},
  title = "{Singularities in Classical Celestial Mechanics}",
  pages = {211--333},
  booktitle	= {Ergodic Theory and Dynamical Systems I,  Proceedings Special year, Maryland},
  year = {1981},
  doi = {10.1007/978-1-4899-6696-4_7},
}

@article{erdi,
doi	= {10.1007/s10569-004-8105-z},
title	= {Global Regularization of the Restricted Problem of Three Bodies},
author	= {Érdi, B.},
journal	= {Celestial Mechanics and Dynamical Astronomy},
year	= {2004},
volume	= {90},
pages	= {35--42},
}

@article{henon1,
title	= {Exploration num{\'e}rique du probl{\`e}me restreint {I}. {M}asses {\'e}gales},
author	= {H{\'e}non, M.},
journal	= {Ann. d'Astrophys.},
year	= {1965},
volume	= {28},
pages	= {499--511},
}

@article{henon2,
title	= {Numerical exploration of the Restricted Problem V. {H}ill’s case: {P}eriodic orbits and Their Stability},
author	= {H{\'e}non, M.},
journal	= {Astron. Astrophys.},
year	= {1969},
volume	= {1},
pages	= {223--238},
}

@article{Hurley,
doi	= {10.1046/j.1365-8711.2002.05038.x},
title	= {Evolution of binary stars and the effect of tides on binary populations},
author	= {Hurley, J.R. and Tout, C.A. and Pols, O.R.},
publisher	= {},
journal	= {Mont. Not. of the Roy. Astr. Soc.},
year	= {2002},
volume	= {329/4},
pages	= {897--928},
}

@article{JZ,
doi	= {10.1080/10586458.2005.10128904},
title	= {A software package for the numerical integration of {ODE}’s by means of high-order taylor methods},
author	= {Jorba, A. and Zou, M.},
publisher	= {},
journal	= {Exp. Maths.},
year	= {2005},
volume	= {14},
pages	= {99--117},
}

@article{Lacomba,
doi	= {10.1016/0022-0396(88)90019-8},
title	= "{Transversal ejection-collision orbits for the restricted problem and the Hill's problem with applications}",
author	= {Lacomba, E. A. and Llibre, J.},
journal	= {Journal of Differential Equations},
year	= {1988},
volume	= {74/1},
pages	= {69--85},
}

@article{LeviCivita1906,
doi	= {10.1007/BF02418577},
title	= {Sur la r{é}solution qualitative du probl{è}me restreint des trois corps},
author	= {Levi-Civita, T.},
journal	= {Acta Mathematica},
year	= {1906},
volume	= {30},
pages	= {305--327},
}

@article{Llibre,
doi	= {10.1007/bf01230662},
title	= "{On the restrited three-body problem when the mass parameter is small}",
author	= {Llibre, J.},
journal	= {Celestial Mechanics},
year	= {1982},
volume	= {28},
pages	= {83--105},
}

@article{McGehee,
doi = {10.1007/BF01390175},
title	= {Triple Collision in the Collinear Three-Body Problem},
author	= {McGehee, R.},
journal	= {Inventiones mathematicae},
year	= {1974},
volume	= {27},
pages	= {191--227},
}

@article{Mod,
doi	= {10.1017/s0074180900064706},
title	= {Mass Transfer between Binary Stars},
author	= {Modisette, J.L. and Kondo, Y.},
journal	= {Symposium - International Astronomical Union},
year	= {1980},
volume	= {88},
pages	= {123--126},
}

@article{Nagler1,
doi	= {10.1103/physreve.69.066218},
title	= {Crash test for the Copenhagen problem},
author	= {Nagler, Jan},
journal	= {Physical Review E},
year	= {2004},
volume	= {69/6},
pages	= {066218},
}

@article{Nagler2,
doi	= {10.1103/PhysRevE.71.026227},
title	= {Crash test for the restricted three-body problem},
author	= {Nagler, Jan},
journal	= {Physical Review E},
year	= {2005},
volume	= {71/2},
pages	= {026227},
}

@article{olle,
doi	= {10.1016/j.cnsns.2017.05.026},
title	= "{To and fro motion for the hydrogen atom in a circularly polarized microwave field}",
author	= {Ollé, M.},
journal	= {Communications in Nonlinear Science and Numerical Simulation},
year	= {2018},
volume	= {54},
pages	= {286--301},
}

@article{ors,
doi	= {10.1016/j.cnsns.2017.07.013},
title	= {Ejection-collision orbits in the RTBP},
author	= {Ollé, M. and Rodríguez, Ó. and Soler, J.},
journal	= {Communications in Nonlinear Science and Numerical Simulation},
year	= {2018},
volume	= {55},
pages	= {298--315},
}

@article{cadis,
doi = {10.1007/978-3-030-41321-7_3},
title	= {Regularisation in Ejection-Collision Orbits of the {RTBP}},
author	= {Oll{\'e}, M. and Rodr{í}guez, Ó. and Soler, J.},
publisher	= {},
journal	= {Recent Advances in Pure and Applied Mathematics},
year	= {2020},
pages	= {35--47},
}

@article{orspaper2,
doi	= {10.1016/j.cnsns.2020.105294},
title	= {Analytical and numerical results on families of n-ejection-collision orbits in the RTBP},
author	= {Ollé, M. and Rodríguez, Ó. and Soler, J.},
journal	= {Communications in Nonlinear Science and Numerical Simulation},
year	= {2020},
volume	= {90},
pages	= {105294},
}

@article{orspaper3,
doi	= {10.1016/j.cnsns.2020.105550},
title	= {Transit regions and ejection/collision orbits in the RTBP.},
author	= {Ollé, M. and Rodríguez, Ó. and Soler, J.},
journal	= {Communications in Nonlinear Science and Numerical Simulation},
year	= {2021},
volume	= {94},
pages	= {105550},
}

@article{PaezGuzzo,
doi	= {10.1016/j.ijnonlinmec.2020.103417},
title	= {A study of temporary captures and collisions in the Circular Restricted Three-Body Problem with normalizations of the Levi-Civita Hamiltonian},
author	= {Paez, R. and Guzzo, M.},
journal	= {International Journal of Non-Linear Mechanics},
year	= {2020},
volume	= {120},
page	= {103417},
}

@book{Pringle,
title	= {Interacting Binary Stars},
author	= {Pringle, J.E. and Wade, R.A.},
publisher	= {Cambridge Univ. Press},
year	= {1985},
}

@book{SS,
    author    = {Stiefel, E. L. and Scheifele, G.},
    title     = {Linear and Regular Celestial Mechanics},
    publisher = {Springer-Verlag},
    year      = {1971},
}

@book{Szebehely,
   title =     "{Theory of Orbits - The Restricted Problem of Three Bodies }",
   author =    {Szebehely, V.G.},
   publisher = {Academic Press},
   year =      {1967},
}

@article{Witjers,
title	= {Evolutionary processes in binary stars},
author	= {Witjers,R.A. and Davies, M.B. and Tout, C.},
publisher	= {},
journal	= {{P}roceedings of the {N}{A}{T}{O} {A}{S}{I} on evolutionary processes in binary stars},
year	= {1995},
}

\appendix
\section{Values of the solutions}
\label{app:valorsUs}

\begingroup
\allowdisplaybreaks
\begin{align*}
	U_7(\tau,\theta_0) &= 0,\\[2.7ex]
	V_7(\tau,\theta_0) &= 0,\\[2.3ex]
	U_8(\tau,\theta_0) &= \frac{\mu(1 - \mu)^{1/3}}{6}\left[
	105\tau\cos\tau - (48 + 87\cos^2\tau - 38\cos^4\tau + 8\cos^6\tau)\sin\tau
	\right]\\
	&\hspace{5mm}*(5\cos^6\theta_0 - 6\cos^2\theta_0 + 2)\cos\theta_0,\\[2.3ex]
	V_8(\tau,\theta_0) &= -\frac{\mu(1 - \mu)^{1/3}}{6}\left[
	105\tau\cos\tau - (48 + 87\cos^2\tau - 38\cos^4\tau + 8\cos^6\tau)\sin\tau
	\right]\\
	&\hspace{5mm}*(5\sin^6\theta_0 - 6\sin^2\theta_0 + 2)\sin\theta_0,\\[2.3ex]
	U_9(\tau,\theta_0) &= -\frac{1}{24}\Big[  4(\tau-\cos\tau\sin\tau)^3\sin\tau 
	- \mu\Big( 
	3\tau(23 + 144\cos^2\tau + 8\cos^4\tau)\sin\tau\\
	&\hspace{7mm}- (379-  217\cos^2\tau - 178\cos^4\tau + 16\cos^6\tau)\cos\tau
	- 480\tau(1 + 6\cos^2\tau)\sin\tau\cos^2\theta_0\\
	&\hspace{7mm} + 32( 81 - 53\cos^2\tau - 32\cos^4\tau + 4\cos^6\tau)\cos\tau\cos^2\theta_0
	- 360\tau^2\cos\tau\cos^4\theta_0\\
	&\hspace{7mm} + 240\tau(3 + 15\cos^2\tau - \cos^4\tau)\sin\tau\cos^4\theta_0\\
	&\hspace{7mm} - 8(374 - 257\cos^2\tau - 143\cos^4\tau + 26\cos^6\tau)\cos\tau\cos^4\theta_0
	\Big) \Big]\sin\theta_0,\\[2.3ex]
	V_9(\tau,\theta_0) &= \frac{1}{24}\Big[  4(\tau-\cos\tau\sin\tau)^3\sin\tau 
	- \mu\Big( 
	3\tau(23 + 144\cos^2\tau + 8\cos^4\tau)\sin\tau\\
	&\hspace{7mm}- (379-  217\cos^2\tau - 178\cos^4\tau + 16\cos^6\tau)\cos\tau
	- 480\tau(1 + 6\cos^2\tau)\sin\tau\sin^2\theta_0\\
	&\hspace{7mm} + 32( 81 - 53\cos^2\tau - 32\cos^4\tau + 4\cos^6\tau)\cos\tau\sin^2\theta_0
	- 360\tau^2\cos\tau\sin^4\theta_0\\
	&\hspace{7mm} + 240\tau(3 + 15\cos^2\tau - \cos^4\tau)\sin\tau\sin^4\theta_0\\
	&\hspace{7mm} - 8(374 - 257\cos^2\tau - 143\cos^4\tau + 26\cos^6\tau)\cos\tau\sin^4\theta_0
	\Big) \Big]\cos\theta_0,\\
	\\[2.95ex]
	U_{10}(\tau,\theta_0) &= \frac{\mu(1-\mu)^{2/3}}{8}\Big[
	315\tau\cos\tau 
	- \big(128 + 325\cos^2\tau - 210\cos^4\tau + 88\cos^6\tau \\ &\hspace{8mm}- 16\cos^8\tau\big)\sin\tau
	\Big]
	(3 - 20\cos^2\theta_0 + 30\cos^4\theta_0 - 14\cos^8\theta_0)\cos\theta_0,\\[2.1ex]
	V_{10}(\tau,\theta_0) &= \frac{\mu(1-\mu)^{2/3}}{8}\Big[
	315\tau\cos\tau 
	- \big(128 + 325\cos^2\tau - 210\cos^4\tau + 88\cos^6\tau \\ &\hspace{8mm}- 16\cos^8\tau\big)\sin\tau
	\Big]
	(3 - 20\sin^2\theta_0 + 30\sin^4\theta_0 - 14\sin^8\theta_0)\sin\theta_0.
\end{align*}
\endgroup

Then, writing the function $U\dot{U}+V\dot{V}$  as an expansion series in $\varepsilon$ and collecting terms of the same order, we can successively find  the terms $\tau_i^*$ of order $i=7,...,10$ from $(U\dot{U}+V\dot{V})(\tau^*)=0$:
\begingroup
\allowdisplaybreaks
\begin{align*}
	\tau_7^*(n,\theta_0) &= 0 ,\\[1.5ex]
	\tau_8^*(n,\theta_0) &= -\frac{35\mu(1-\mu)^{1/3}n\pi\cos(2\theta_0)(5\cos^2(2\theta_0) - 3)}{4},\\[1.5ex]
	\tau_9^*(n,\theta_0) & = \frac{15\mu n^2\pi^2\sin(4\theta_0)}{2},\\[1.5ex]
	\tau_{10}^*(n,\theta_0) &= \frac{315\mu(1-\mu)^{2/3} n\pi(13 - 10\cos(4\theta_0) - 35\cos^2(4\theta_0))}{256}.
\end{align*}
\endgroup

Now we are ready to compute the explicit expression for the angular momentum $M(n,\theta _0)=(U\dot{V}-V\dot{U})(\tau ^*)$ up to order 10 which is the following:
\[
    \begin{aligned}
	    M(n,\theta_0) &= -\frac{15\mu n\pi\sin(4\theta_0)}{4}\varepsilon^6 + \frac{105\mu(1-\mu)^{1/3}n\pi\left(\sin(2\theta_0) + 5\sin(6\theta_0)\right)}{64}\varepsilon^8 +\frac{15\mu n^2\pi^2\cos(4\theta_0)}{2}\varepsilon^9\\
	    &\hspace{6mm}- \frac{315\mu(1-\mu)^{2/3}n\pi(2\sin(4\theta_0) + 7\sin(8\theta_0))}{128}\varepsilon^{10} +\mathcal{O}(\varepsilon^{11}).
    \end{aligned}
\]
which is precisely \eqref{Maordre10}.

\section{Proof of the auxiliary lemmas}

We must prove Lemmas \ref{Hfitatlema} and \ref{le:LipschitzLema}.  
First, let us fix some notation.
Given a matrix $A=(a_{ij})_{i,j=1,...,4}$, we denote the new matrix
\[|A|=(|a_{ij}|)_{i,j=1,...,4}.\]
Analogously for vectors ${\bm v}=(v_1,...,v_4)$: 
\[
|{\bm v}|=(|v_1|,...,|v_4|),
\]
and given two vectors ${\bm v}=(v_1,...,v_4)$, ${\bm w}=(w_1,...,w_4)$, we will say that \[{\bm v}\le {\bm w} \quad \mbox{if} \quad v_i\le w_i\quad\forall i =1,..,4. \]
Similarly with matrices $A\le B$.

With this notation we have:
\[
|A {\bm v}| \le |A||{\bm v}|.
\]
During this section we will use $M$ to denote any constant which appears in the bounds and is independent of $\xi$, $\mu$ and $n \in \mathbb{N}$.

\begin{lemma} \label{le:X}
The fundamental matrix $X$ for system \eqref{eq:u1sysV2linear} can be expressed  as
\[
    X(\hat{\mathcal{T}}) =
    \left(
    \begin{matrix}
        \cos(n\hat{\mathcal{T}}) + \mathcal{O}(\xi^3)  &   \mathcal{O}(\xi^3) &  \cfrac{\sin(n\hat{\mathcal{T}}) + \mathcal{O}(\xi^3)}{n}   &         \cfrac{1}{n}\mathcal{O}(\xi^3) \\
        \mathcal{O}(\xi^3) & \cos(n\hat{\mathcal{T}}) + \mathcal{O}(\xi^3)  &         \cfrac{1}{n}\mathcal{O}(\xi^3) &    \cfrac{\sin(n\hat{\mathcal{T}}) + \mathcal{O}(\xi^3)}{n}\\[1.5ex]
        -n\sin(n\hat{\mathcal{T}}) + n \mathcal{O}(\xi^3) & n \mathcal{O}(\xi^3)  & \cos(n\hat{\mathcal{T}})+ \mathcal{O}(\xi^3) &  \mathcal{O}(\xi^3) \\[1.5ex]
        n \mathcal{O}(\xi^3) & -n\sin(n\hat{\mathcal{T}}) + n \mathcal{O}(\xi^3)&   \mathcal{O}(\xi^3) & \cos(n\hat{\mathcal{T}})+ \mathcal{O}(\xi^3)   \\
    \end{matrix}
    \right)
    ,
\]
and its inverse matrix as
\[
    X^{-1}(\hat{\mathcal{T}}) =
    \left(
    \begin{matrix}
        \cos(n\hat{\mathcal{T}}) + \mathcal{O}(\xi^3)  &   \mathcal{O}(\xi^3) &  \cfrac{-\sin(n\hat{\mathcal{T}}) + \mathcal{O}(\xi^3)}{n}   &         \cfrac{1}{n}\mathcal{O}(\xi^3) \\
        \mathcal{O}(\xi^3) & \cos(n\hat{\mathcal{T}}) + \mathcal{O}(\xi^3)  &         \cfrac{1}{n}\mathcal{O}(\xi^3) &    \cfrac{-\sin(n\hat{\mathcal{T}}) + \mathcal{O}(\xi^3)}{n}\\[1.5ex]
        n\sin(n\hat{\mathcal{T}}) + n \mathcal{O}(\xi^3) & n \mathcal{O}(\xi^3)  & \cos(n\hat{\mathcal{T}})+ \mathcal{O}(\xi^3) &  \mathcal{O}(\xi^3) \\[1.5ex]
        n \mathcal{O}(\xi^3) & n\sin(n\hat{\mathcal{T}}) + n \mathcal{O}(\xi^3)&   \mathcal{O}(\xi^3) & \cos(n\hat{\mathcal{T}})+ \mathcal{O}(\xi^3)   \\
    \end{matrix}
    \right)
    .
\]
\end{lemma}

\begin{proof}
    Consider the general solution 
    $\bm{\mathcal{U}}_0$ of system \eqref{eq:u0sys} given by \eqref{eq:sol2cside}, \eqref{eq:timet} and \eqref{eq:solucioUnewRot}, we can express the fundamental matrix of the system 
    \[
\dot{\bm{\mathcal{U}}}_{1} =
D\bm{\mathcal{F}}_0(\bm{\mathcal{U}}_{0})\bm{\mathcal{U}}_{1},
\]
    as 
    \[
        X = RA,
    \]
    where
    \begin{equation}
        R =
        \left(
        \begin{matrix}
        \cos(-t/2) & -\sin(-t/2) & 0 & 0\\
        \sin(-t/2) & \cos(-t/2) & 0 & 0\\
        0 & 0 & \cos(-t/2) & -\sin(-t/2)\\
        0 & 0 & \sin(-t/2) & \cos(-t/2)\\
        \end{matrix}
        \right),
    \end{equation}
    and $A$ is the matrix with rows
    \begin{equation}
        \begin{split}
            \bm{A}_1(\hat{\mathcal{T}}) &= \left[\frac{\partial\bar{\mathcal{U}}_0}{\partial \bm{\mathcal{U}}_0(0)}+\frac{\bar{\mathcal{V}}_0}{2}\frac{\partial t}{\partial \bm{\mathcal{U}}_0(0)}\right](\hat{\mathcal{T}}),\\[1.5ex]
            \bm{A}_2(\hat{\mathcal{T}}) &= \left[\frac{\partial\bar{\mathcal{V}}_0}{\partial \bm{\mathcal{U}}_0(0)}-\frac{\bar{\mathcal{U}}_0}{2}\frac{dt}{\partial \bm{\mathcal{U}}_0(0)}\right](\hat{\mathcal{T}}),\\[1.5ex]
            \bm{A}_3(\hat{\mathcal{T}}) &= \left[\frac{\partial\dot{\bar{\mathcal{U}}}_0}{\partial \bm{\mathcal{U}}_0(0)}+2\left(\bar{\mathcal{U}}_0^2+3\bar{\mathcal{V}}_0^2\right)\xi^3\frac{\partial\bar{\mathcal{V}}_0}{\partial \bm{\mathcal{U}}_0(0)} + 4\bar{\mathcal{U}}_0\bar{\mathcal{V}}_0\xi^3\frac{\partial\bar{\mathcal{U}}_0}{\partial \bm{\mathcal{U}}_0(0)} + \frac{\dot{\bar{\mathcal{V}}}_0-2(\bar{\mathcal{U}}_0^2+\bar{\mathcal{V}}_0^2)\bar{\mathcal{U}}_0\xi^3}{2}\frac{\partial t}{\partial \bm{\mathcal{U}}_0(0)}\right](\hat{\mathcal{T}}),\\[1.5ex]
            \bm{A}_4(\hat{\mathcal{T}}) &= \left[\frac{\partial \dot{\bar{\mathcal{V}}}_0}{\partial \bm{\mathcal{U}}_0(0)}-2\left(3\bar{\mathcal{U}}_0^2+\bar{\mathcal{V}}_0^2\right)\xi^3\frac{\partial \bar{\mathcal{U}}_0}{\partial \bm{\mathcal{U}}_0(0)} - 4\bar{\mathcal{U}}_0\bar{\mathcal{V}}_0\xi^3\frac{\partial \bar{\mathcal{V}}_0}{\partial \bm{\mathcal{U}}_0(0)} - \frac{\dot{\bar{\mathcal{U}}}_0+2(\bar{\mathcal{U}}_0^2+\bar{\mathcal{V}}_0^2)\bar{\mathcal{V}}_0\xi^3}{2}\frac{\partial t}{\partial \bm{\mathcal{U}}_0(0)}\right](\hat{\mathcal{T}}),\\[1.5ex]
        \end{split}
    \end{equation}
    where $\bar{\bm{\mathcal{U}}}_0$ is given by \eqref{eq:sol2cside}.

    Note that we are interested in solving equations \eqref{eq:u1sysV2linear}, which correspond to  the ejection orbits $\bm{\mathcal{U}}_0^e$, that have initial conditions $(0,0,n\cos\theta_0,n\sin\theta_0)$. So we must compute the fundamental matrix $X$ 
    with $\bm{\mathcal{U}}_0(\hat{\mathcal{T}})=\bm{\mathcal{U}}_0^e(\hat{\mathcal{T}})$. We denote by $A^e$ and $R^e$ the corresponding matrices.
    Recall also that the expression of $t$ is given by \eqref{eq:tempsRotacio} and the explicit elements of $A^e$ are provided in Appendix \ref{app:Aij}.

    In this way, we can express $A^e$ and $R^e$ as
    \[
        A^e =
        \left(
        \begin{matrix}
            \cos(n\hat{\mathcal{T}}) + \mathcal{O}(\xi^3)  &   \mathcal{O}(\xi^3) &  \cfrac{\sin(n\hat{\mathcal{T}}) + \mathcal{O}(\xi^3)}{n}   &         \cfrac{1}{n}\mathcal{O}(\xi^3) \\
            \mathcal{O}(\xi^3) & \cos(n\hat{\mathcal{T}}) + \mathcal{O}(\xi^3)  &         \cfrac{1}{n}\mathcal{O}(\xi^3) &    \cfrac{\sin(n\hat{\mathcal{T}}) + \mathcal{O}(\xi^3)}{n}\\[1.5ex]
            -n\sin(n\hat{\mathcal{T}}) + n \mathcal{O}(\xi^3) & n \mathcal{O}(\xi^3)  & \cos(n\hat{\mathcal{T}})+ \mathcal{O}(\xi^3) &  \mathcal{O}(\xi^3) \\[1.5ex]
            n \mathcal{O}(\xi^3) & -n\sin(n\hat{\mathcal{T}}) + n \mathcal{O}(\xi^3)&   \mathcal{O}(\xi^3) & \cos(n\hat{\mathcal{T}})+ \mathcal{O}(\xi^3)   \\
        \end{matrix}
        \right)
        ,
    \]

    \[
        R^e = Id +
        \left(
        \begin{matrix}
            \mathcal{O}(\xi^6) & \mathcal{O}(\xi^3) & 0 & 0\\
            \mathcal{O}(\xi^3) & \mathcal{O}(\xi^6) & 0 & 0\\
            0 & 0 &\mathcal{O}(\xi^6) & \mathcal{O}(\xi^3) \\
            0 & 0 &\mathcal{O}(\xi^3) & \mathcal{O}(\xi^6) \\
        \end{matrix}
        \right),
    \]
    and therefore,
    \[
        X(\hat{\mathcal{T}}) =
        \left(
        \begin{matrix}
            \cos(n\hat{\mathcal{T}}) + \mathcal{O}(\xi^3)  &   \mathcal{O}(\xi^3) &  \cfrac{\sin(n\hat{\mathcal{T}}) + \mathcal{O}(\xi^3)}{n}   &         \cfrac{1}{n}\mathcal{O}(\xi^3) \\
            \mathcal{O}(\xi^3) & \cos(n\hat{\mathcal{T}}) + \mathcal{O}(\xi^3)  &         \cfrac{1}{n}\mathcal{O}(\xi^3) &    \cfrac{\sin(n\hat{\mathcal{T}}) + \mathcal{O}(\xi^3)}{n}\\[1.5ex]
            -n\sin(n\hat{\mathcal{T}}) + n \mathcal{O}(\xi^3) & n \mathcal{O}(\xi^3)  & \cos(n\hat{\mathcal{T}})+ \mathcal{O}(\xi^3) &  \mathcal{O}(\xi^3) \\[1.5ex]
            n \mathcal{O}(\xi^3) & -n\sin(n\hat{\mathcal{T}}) + n \mathcal{O}(\xi^3)&   \mathcal{O}(\xi^3) & \cos(n\hat{\mathcal{T}})+ \mathcal{O}(\xi^3)   \\
        \end{matrix}
        \right)
        .
    \]
    
    The expression for $X^{-1}(\hat{\mathcal{T}})$ can be found in a similar way.

\end{proof}

\subsection{Proof of Lemma \ref{Hfitatlema} }\label{app:Lema2}

    From \eqref{defG} and \eqref{defH} we have
    \begin{equation}\label{H(0)}
        \bm{\mathcal{H}}\{\bm{0}\}(\hat{\mathcal{T}})
        = X(\hat{\mathcal{T}})\int_0^{\hat{\mathcal{T}}} X^{-1}(\hat{\mathcal{T}})\bm{\mathcal{G}}(\bm{0})\,d\hat{\mathcal{T}}
        = \mu X(\hat{\mathcal{T}})\int_0^{\hat{\mathcal{T}}} X^{-1}(\hat{\mathcal{T}})\bm{\mathcal{F}}_1(\bm{\mathcal{U}}_0^e(\hat{\mathcal{T}}))\,d\hat{\mathcal{T}},
    \end{equation}
    so, the first step is to bound the components of $\bm{\mathcal{F}}_1(\bm{\mathcal{U}}_0^e)$ (see \eqref{eq:F01}).
    
    Concerning the expansions involving $r_2$ in \eqref{eq:r2nou}, we have
    \begin{equation}\label{eq:F1}
    \begin{aligned}
        \left(\frac{1}{r_2}-1\right) &= - \frac{2(1-\mu)^{1/3}(\mathcal{U}^2-\mathcal{V}^2)}{n^{2/3}}\xi^2
        + \frac{4(1 - \mu)^{2/3}(\mathcal{U}^4 - 4\mathcal{U}^2\mathcal{V}^2 + \mathcal{V}^4)}{n^{4/3}}\xi^4 + \frac{1}{n^2}\mathcal{O}(\xi^6),\\
        \frac{1}{r_2^3} &= 1 - \frac{6(1-\mu)^{1/3}(\mathcal{U}^2-\mathcal{V}^2)}{n^{2/3}}\xi^2 + \frac{1}{n^{4/3}}\mathcal{O}(\xi^4),
    \end{aligned}
    \end{equation}
where the symbol $\mathcal{O}$ refers to terms bounded for bounded $\bm{\mathcal{U}}$ and  any $\mu\in(0,1)$ and $n\in\mathbb{N}$.
    Thus, we obtain
    \begin{equation}\label{eq:F1ordre}
        \bm{\mathcal{F}}_1(\mathcal{U},\mathcal{V}) = \left(
        \begin{matrix}
        0\\
        0\\
        24\left(\mathcal{U}^4-2\mathcal{U}^2\mathcal{V}^2-\mathcal{V}^4\right)\mathcal{U}
        \xi^6 + \frac{1}{n^{2/3}}\mathcal{O}(\xi^8)\\
        24\left(\mathcal{V}^4-2\mathcal{U}^2\mathcal{V}^2-\mathcal{U}^4\right)\mathcal{V}
        \xi^6 + \frac{1}{n^{2/3}}\mathcal{O}(\xi^8)
        \end{matrix}
        \right).
    \end{equation}
    Let us bound $| \bm{\mathcal{F}}_1(\mathcal{U}_0^e,\mathcal{V}_0^e) |$.
    Recall that (see \eqref{eq:solucioEcosKeplerRaro}) we have that,
    \[
        {\mathcal{U}_0^e}^2(\theta_0,\hat{\mathcal{T}})+
        {\mathcal{V}_0^e}^2(\theta_0,\hat{\mathcal{T}})=\sin ^2(n\hat{\mathcal{T}})\le 1,
    \]
Therefore, $|\mathcal{U}_0^e|\le 1$, $|\mathcal{V}_0^e|\le 1$ are bounded and consequently, for $\xi$ small enough:
    \begin{equation}\label{eq:F1fita}
        |\bm{\mathcal{F}}_1(\mathcal{U}_0^e,\mathcal{V}_0^e)| \leq M \left(
        \begin{matrix}
        0\\
        0\\
        \xi^6 + \frac{1}{n^{2/3}}\mathcal{O}(\xi^8)\\
        \xi^6 + \frac{1}{n^{2/3}}\mathcal{O}(\xi^8)
        \end{matrix}
        \right)
        =
        M \left[\xi^6 + \frac{1}{n^{2/3}}\mathcal{O}(\xi^8)\right]
        \left(
        \begin{matrix}
        0\\
        0\\
        1\\
        1
        \end{matrix}
        \right)
        \le         M \xi^6
        \left(
        \begin{matrix}
        0\\
        0\\
        1\\
        1
        \end{matrix}
        \right).
    \end{equation}
    and the constant $M$ is independent of $\mu$ and $n$.
    
    By Lemma \ref{le:X} we can bound $|X|\leq \mathcal{M}$ and $|X^{-1}|\leq \mathcal{M}$ where
    \begin{equation}\label{eq:1}
        \mathcal{M} =
        \left(
        \begin{matrix}
            1 + \mathcal{O}(\xi^3) & \mathcal{O}(\xi^3) & \cfrac{1 + \mathcal{O}(\xi^3)}{n} & \cfrac{1}{n}\mathcal{O}(\xi^3)\\[1.5ex]
            \mathcal{O}(\xi^3) & 1 + \mathcal{O}(\xi^3) & \cfrac{1}{n}\mathcal{O}(\xi^3) & \cfrac{1 + \mathcal{O}(\xi^3)}{n}\\[1.5ex]
            n + n \mathcal{O}(\xi^3) & n \mathcal{O}(\xi^3)  & 1+ \mathcal{O}(\xi^3) &  \mathcal{O}(\xi^3) \\[1.5ex]
            n \mathcal{O}(\xi^3) & n + n \mathcal{O}(\xi^3)&   \mathcal{O}(\xi^3) & 1+ \mathcal{O}(\xi^3)   \\
        \end{matrix}
        \right).
    \end{equation}
    In this way, we have
    \begin{equation}\label{eq:2}
        \begin{aligned}
            |X^{-1}(\hat{\mathcal{T}})\bm{\mathcal{F}}_1(\mathcal{U}_0^e(\hat{\mathcal{T}}),
\mathcal{V}_0^e(\hat{\mathcal{T}}))| 
            & \leq |X^{-1}(\hat{\mathcal{T}})| |\bm{\mathcal{F}}_1(\mathcal{U}_0^e(\hat{\mathcal{T}}),
\mathcal{V}_0^e(\hat{\mathcal{T}}))|\\
            & \leq \mathcal{M}|\bm{\mathcal{F}}_1(\mathcal{U}_0^e(\hat{\mathcal{T}}),
\mathcal{V}_0^e(\hat{\mathcal{T}}))|\\
            & \leq 
            M \xi^6
        \left(
        \begin{matrix}
        1/n\\
        1/n\\
        1\\
        1
        \end{matrix}
        \right).
        \end{aligned}
    \end{equation}
    And, therefore, as we have taken $T=2\pi$:
    \begin{equation}\label{eq:3}
        \int_0^T |X^{-1}(\hat{\mathcal{T}})\bm{\mathcal{F}}_1(\mathcal{U}_0^e(\hat{\mathcal{T}}),
\mathcal{V}_0^e(\hat{\mathcal{T}}))|\, d\hat{\mathcal{T}} 
    \leq
   M \xi^6 
        \left(
        \begin{matrix}
        1/n\\
        1/n\\
        1\\
        1
        \end{matrix}
        \right).
    \end{equation}
    Finally, multiplying by $\mu X$ we have
    \begin{equation}\label{eq:4}
        |\bm{\mathcal{H}}\{\bm{0}\}| \leq 
        M \mu  \xi^6
        \left(
        \begin{matrix}
        1/n\\
        1/n\\
        1\\
        1
        \end{matrix}
        \right).
    \end{equation}
    and using the definition of the norm in \eqref{eq:norm} and renaming $M_1=M$ we obtain the desired result:
    \begin{equation}\label{eq:5}
        \lVert\bm{\mathcal{H}}\{\bm{0}\}\rVert \leq 
        M \mu\xi^6 .
    \end{equation}

\subsection{Proof of Lemma \ref{le:LipschitzLema}}
\label{app:Lema3}

In order to bound 
$\bm{\mathcal{H}}(\bm{\mathcal{U}}_\oplus)-
\bm{\mathcal{H}}(\bm{\mathcal{U}}_\ominus)$ (see \eqref{defH}), first we need to bound $\bm{\mathcal{G}}(\bm{\mathcal{U}}_\oplus)-
\bm{\mathcal{G}}(\bm{\mathcal{U}}_\ominus)$, for 
$\bm{\mathcal{U}}_\oplus$, $\bm{\mathcal{U}}_\ominus\in B_R(\bm{0})$, where $R=2M_1\mu \xi^6$. In order to ease the computations,   let us introduce 
\begin{equation}
\bm{\mathcal{G}}(\bm{\mathcal{U}}_1) = \bm{\mathcal{G}}_0(\bm{\mathcal{U}}_1) + \bm{\mathcal{G}}_1(\bm{\mathcal{U}}_1),
\end{equation}
with
\begin{equation}
    \begin{aligned}
        \bm{\mathcal{G}}_0(\bm{\mathcal{U}}_1) & = \bm{\mathcal{F}}_0(\bm{\mathcal{U}}_0^e(\hat{\mathcal{T}})+\bm{\mathcal{U}}_1)
    -\bm{\mathcal{F}}_0(\bm{\mathcal{U}}_0^e(\hat{\mathcal{T}}))-D\bm{\mathcal{F}}_0(\bm{\mathcal{U}}_0^e(\hat{\mathcal{T}}))\bm{\mathcal{U}}_1,\\
    \bm{\mathcal{G}}_1(\bm{\mathcal{U}}_1) &=
    \mu \bm{\mathcal{F}}_1(\mathcal{U}_0^e(\hat{\mathcal{T}})+\mathcal{U}_1,\mathcal{V}_0^e(\hat{\mathcal{T}})+\mathcal{V}_1).
    \end{aligned}
\end{equation}
We will bound separately the term $\bm{\mathcal{G}}_0(\bm{\mathcal{U}}_\oplus)-
\bm{\mathcal{G}}_0(\bm{\mathcal{U}}_\ominus)$ in Lemma \ref{lemma_G0i} and  
 $\bm{\mathcal{G}}_1(\bm{\mathcal{U}}_\oplus)-
\bm{\mathcal{G}}_1(\bm{\mathcal{U}}_\ominus)$ in Lemma \ref{lemma_G1}.
\begin{lemma}\label{lemma_G0i} 
Take $\bm{\mathcal{U}}_\oplus, \bm{\mathcal{U}}_\ominus \in B_R(\bm{0})$.
Then for $0<\xi$ small enough we have that:
\[
\lVert\bm{\mathcal{G}}_0(\bm{\mathcal{U}}_\oplus)-\bm{\mathcal{G}}_0 (\bm{\mathcal{U}}_\ominus)\rVert \leq \frac{M \mu}{n} \xi^9\lVert\bm{\mathcal{U}}_\oplus-\bm{\mathcal{U}}_\ominus\rVert.
\]
\end{lemma}
\begin{proof}
First, we observe that $\bm{\mathcal{G}}_0(\bm{\mathcal{U}}) = (\mathcal{G}_0^1,\mathcal{G}_0^2,\mathcal{G}_0^3,\mathcal{G}_0^4)(\bm{\mathcal{U}})= (0,0,\mathcal{G}_0^3,\mathcal{G}_0^4)(\bm{\mathcal{U}})$. 
Therefore we will consider the last two components. 
We will do the computations for 
$\mathcal{G}_0^3$, the ones for $\mathcal{G}_0^4$ are analogous.
Using the Mean Value Theorem we have:
    \begin{equation}\label{eq:integral1}
    \begin{aligned}
        \mathcal{G}_0^3(\bm{\mathcal{U}}_\oplus)-\mathcal{G}_0^3 (\bm{\mathcal{U}}_\ominus) 
        & = \mathcal{F}_0^3(\bm{\mathcal{U}}_0+\bm{\mathcal{U}}_\oplus) - \mathcal{F}_0^3(\bm{\mathcal{U}}_0+\bm{\mathcal{U}}_\ominus) - D\mathcal{F}_0^3(\bm{\mathcal{U}}_0) (\bm{\mathcal{U}}_\oplus-\bm{\mathcal{U}}_\ominus)\\
        & = \int_0^1 \left[D\mathcal{F}_0^3(\bm{\mathcal{U}}_0 + s\bm{\mathcal{U}}_\oplus+(1-s)\bm{\mathcal{U}}_\ominus)(\bm{\mathcal{U}}_\oplus-\bm{\mathcal{U}}_\ominus)\right]\,ds - D\mathcal{F}_0^3(\bm{\mathcal{U}}_0) (\bm{\mathcal{U}}_\oplus-\bm{\mathcal{U}}_\ominus)\\ 
        & = \left\{\int_0^1 \left[D\mathcal{F}_0^3(\bm{\mathcal{U}}_0 + s\bm{\mathcal{U}}_\oplus+(1-s)\bm{\mathcal{U}}_\ominus) - D\mathcal{F}_0^3(\bm{\mathcal{U}}_0)\right]\,ds\right\}\,(\bm{\mathcal{U}}_\oplus-\bm{\mathcal{U}}_\ominus)\\
        & = \left\{\int_0^1 \int_0^1 \left(s\bm{\mathcal{U}}_\oplus+(1-s)\bm{\mathcal{U}}_\ominus\right)^t
        D^2\mathcal{F}_0^3\left(\bm{\mathcal{U}}_0 + z\left[s\bm{\mathcal{U}}_\oplus+(1-s)\bm{\mathcal{U}}_\ominus\right]\right) \,dz\,ds\right\}\,(\bm{\mathcal{U}}_\oplus-\bm{\mathcal{U}}_\ominus).
    \end{aligned}
    \end{equation}
%
Now we want to bound the expression appearing in the previous double integral.  
Notice that  $D^2\mathcal{F}_0^3$ (see \eqref{eq:F01}) is given by:
\[
    D^2 \mathcal{F}_0^3 =
        \begin{pmatrix}
            16\left[\dot{\mathcal{V}} + 3(5\mathcal{U}^2 + 3\mathcal{V}^2)\mathcal{U}\xi^3\right]\xi^3 &                48(3\mathcal{U}^2 + \mathcal{V}^2)\mathcal{V}\xi^6 & 0 & 16\mathcal{U}\xi^3 \\
            48(3\mathcal{U}^2 + \mathcal{V}^2)\mathcal{V}\xi^6 &
            16\left[\dot{\mathcal{V}} + 3(\mathcal{U}^2 + 3\mathcal{V}^2)\mathcal{U}\xi^3\right]\xi^3 & 0 & 16\mathcal{V}\xi^3\\
            0 & 0 & 0 & 0\\
            16\mathcal{U}\xi^3 & 16\mathcal{V}\xi^3 & 0 & 0\\
        \end{pmatrix},
\]
and thus, as (see \eqref{eq:solucioEcosKeplerRaro}):
\[
|\mathcal{U}_0^e|\le 1, |\mathcal{V}_0^e|\le 1, |\dot {\mathcal{U}}_0^e|\le n, |\dot{ \mathcal{V}}_0^e|\le n, \quad \lVert \bm{\mathcal{U}}_\otimes \rVert \le 2 M_1 \mu \xi^6,
\]
we have that:
\begin{equation}
    \left|D^2 \mathcal{F}_0^3(\bm{\mathcal{U}}_0^e+\bm{\mathcal{U}}_\otimes)\right| \leq
    M\xi^3
    \begin{pmatrix}
        n & \xi^3 & 0 & 1\\
        \xi^3 & n & 0 & 1\\
        0 & 0 & 0 & 0 \\
        1 & 1 & 0 & 0
    \end{pmatrix}.
\end{equation}
Now, as $\bm{\lVert \mathcal{U}}_\odot \rVert \le 2 M_1 \mu \xi^6$:
\begin{equation}\label{fila1}
\begin{aligned}
    |\bm{\mathcal{U}}_\odot^t D^2\mathcal{F}_0^3(\bm{\mathcal{U}}_0^e+\bm{\mathcal{U}}_\otimes)|
    &\leq|\bm{\mathcal{U}}_\odot^t|\,| D^2\mathcal{F}_0^3(\bm{\mathcal{U}}_0^e+\bm{\mathcal{U}}_\otimes)|\\
    & \leq
    2M_1\mu\xi^6
    \left(1/n,\, 1/n,\, 1,\,1\right) 
    M\xi^3
    \begin{pmatrix}
        n & \xi^3 & 0 & 1\\
        \xi^3 & n & 0 & 1\\
        0 & 0 & 0 & 0 \\
        1 & 1 & 0 & 0
    \end{pmatrix}\\
    & \leq 2M M_1 \mu \xi^9
    \left(1,\,1,\,0,\,1/n\right).
\end{aligned}
\end{equation}

Taking into account the integral expression in \eqref{eq:integral1} we obtain 
\[
|\mathcal{G}_0^3(\bm{\mathcal{U}}_\oplus)-\mathcal{G}_0^3 (\bm{\mathcal{U}}_\ominus)|
\leq 
2M M_1 \mu \xi^9
    \left(1,\,1,\,0,\,1/n\right)
    |\bm{\mathcal{U}}_\oplus-\bm{\mathcal{U}}_\ominus|   \le 
    \frac{1}{n}2M M_1 \mu \xi^9 \lVert\bm{\mathcal{U}}_\oplus-\bm{\mathcal{U}}_\ominus\rVert .
\]
We get a similar bound for the fourth components and using that the first and the second are identically zero and the definition of the norm we get the result of the lemma.
\end{proof}

The next goal is to bound $\bm{\mathcal{G}}_1(\bm{\mathcal{U}}_\oplus) - \bm{\mathcal{G}}_1(\bm{\mathcal{U}}_\ominus)$. To do so, we apply the same trick:

\begin{lemma}\label{lemma_G1}
Given $\bm{\mathcal{U}}_\oplus, \bm{\mathcal{U}}_\ominus \in B_ R(\bm 0)$. Then for $\xi >0$ small enough we have that 

\[
\lVert\bm{\mathcal{G}}_1(\bm{\mathcal{U}}_\oplus)-\bm{\mathcal{G}}_1 (\bm{\mathcal{U}}_\ominus)\rVert
\leq \frac{M \mu \xi^6}{n}
 \lVert \bm{\mathcal{U}}_\oplus-\bm{\mathcal{U}}_\ominus\rVert  .   
\]
 \end{lemma}
 
 \begin{proof}
 Using again the Main Value Theorem we obtain:
 
      \begin{equation}\label{eq:integral2}
        \bm{\mathcal{G}}_1(\bm{\mathcal{U}}_\oplus) - \bm{\mathcal{G}}_1(\bm{\mathcal{U}}_\ominus)
        =  \int_0^1D\bm{\mathcal{G}}_1(s\bm{\mathcal{U}}_\oplus + (1-s)\bm{\mathcal{U}}_\ominus)\,(\bm{\mathcal{U}}_\oplus-\bm{\mathcal{U}}_\ominus)  \,ds .
\end{equation}

So we only need to bound $| D\bm{\mathcal{G}}_1(\bm{\mathcal{U}}_\odot) |$ where $\bm{\mathcal{U}}_\odot\in B_R(\bm{0})$, $=2M_1 \mu \xi ^6$
 
Let us recall that 
\[
        D\bm{\mathcal{G}}_1(\bm{\mathcal{U}}_\odot)
    = \mu D\bm{\mathcal{F}}_1(\mathcal{U}_0^e+\mathcal{U}_\odot, \mathcal{V}_0^e+\mathcal{V}_\odot),
\]
and  $\bm{\mathcal{F}}_1$ is given in \eqref{eq:F01}. Proceeding similarly as to bound
 $\bm{\mathcal{F}}_1$ we can differentiate \eqref{eq:F1} to obtain:
\[
\begin{aligned}
    D\bm{\mathcal{F}}_1(\mathcal{U}, \mathcal{V}) 
    &=
    \begin{pmatrix}
        0 & 0 & 0 & 0\\
        0 & 0 & 0 & 0 \\
        24(5\mathcal{U}^4-6\mathcal{U}^2\mathcal{V}^2-\mathcal{V}^4)\xi^6 + \frac{1}{n^{2/3}}\mathcal{O}(\xi^8)
        & -96(\mathcal{U}^2+\mathcal{V}^2)\mathcal{U}\mathcal{V}\xi^6 + \frac{1}{n^{2/3}}\mathcal{O}(\xi^8) & 0 & 0\\
        -96(\mathcal{U}^2+\mathcal{V}^2)\mathcal{U}\mathcal{V}\xi^6 + \frac{1}{n^{2/3}}\mathcal{O}(\xi^8)
        &
        24(5\mathcal{V}^4-6\mathcal{U}^2\mathcal{V}^2-\mathcal{U}^4)\xi^6 + \frac{1}{n^{2/3}}\mathcal{O}(\xi^8)
        & 0 & 0
    \end{pmatrix}\\
    &=
    24\xi^6
    \begin{pmatrix}
        0 & 0 & 0 & 0\\
        0 & 0 & 0 & 0 \\
        5\mathcal{U}^4-6\mathcal{U}^2\mathcal{V}^2-\mathcal{V}^4 + \frac{1}{n^{2/3}}\mathcal{O}(\xi^2)
        & -4(\mathcal{U}^2+\mathcal{V}^2)\mathcal{U}\mathcal{V} + \frac{1}{n^{2/3}}\mathcal{O}(\xi^2) & 0 & 0\\
        -4(\mathcal{U}^2+\mathcal{V}^2)\mathcal{U}\mathcal{V} + \frac{1}{n^{2/3}}\mathcal{O}(\xi^2)
        &
        5\mathcal{V}^4-6\mathcal{U}^2\mathcal{V}^2-\mathcal{U}^4 + \frac{1}{n^{2/3}}\mathcal{O}(\xi^2)
        & 0 & 0
    \end{pmatrix}.
\end{aligned}
\]
So, using again that  (see \eqref{eq:solucioEcosKeplerRaro}):
\[
|\mathcal{U}_0^e|\le 1, |\mathcal{V}_0^e|\le 1, \quad \lVert \bm{\mathcal{U}}_\odot \rVert \le 2 M_1 \mu \xi^6,
\]
\[
    | D\bm{\mathcal{F}}_1(\mathcal{U}_0^e+\mathcal{U}_\odot, \mathcal{V}_0^e+\mathcal{V}_\odot) | \leq
    [120\xi^6 + \frac{1}{n^{2/3}}\mathcal{O}(\xi^8)]
    \begin{pmatrix}
        0 & 0 & 0 & 0\\
        0 & 0 & 0 & 0 \\
        1 & 1 & 0 & 0 \\
        1 & 1 & 0 & 0 
    \end{pmatrix},
\]
and therefore 
\[
    | D\bm{\mathcal{G}}_1(\bm{\mathcal{U}}_\odot) | \leq
    M  \mu\xi^6
    \begin{pmatrix}
        0 & 0 & 0 & 0\\
        0 & 0 & 0 & 0 \\
        1 & 1 & 0 & 0 \\
        1 & 1 & 0 & 0 
    \end{pmatrix}.
\]

And using the integral equation \eqref{eq:integral2} and the fact that the first two rows of the previous matrix are zero we get:
\[
|\bm{\mathcal{G}}_1(\bm{\mathcal{U}}_\oplus)-\bm{\mathcal{G}}_1 (\bm{\mathcal{U}}_\ominus)|  \le M\mu \xi ^6 
\begin{pmatrix}
0 \\
0 \\
|\bm{ \mathcal{U}}_\oplus-\bm{\mathcal{V}}_\oplus | \\
|\bm{\mathcal{U}}_\oplus-\bm{\mathcal{V}}_\oplus|
\end{pmatrix}
\le  \frac{M\mu \xi ^6}{n}\lVert \bm{\mathcal{U}}_\oplus-\bm{\mathcal{U}}_\ominus\rVert    
\begin{pmatrix}
0 \\
0 \\
1\\
1
\end{pmatrix}
.
\]
Now, using the definition of the norm we get the result.
 
 \end{proof}

From the results of  lemmas \ref{lemma_G0i} and \ref{lemma_G1} we have:
\[
\lVert\bm{\mathcal{G}}(\bm{\mathcal{U}}_\oplus)-\bm{\mathcal{G}} (\bm{\mathcal{U}}_\ominus)\rVert
\leq
 \frac{M \mu \xi^6}{n}     \lVert\bm{\mathcal{U}}_\oplus-\bm{\mathcal{U}}_\ominus\rVert.
\]

Now, proceeding as we did in the proof of Lemma \ref{Hfitatlema}, we multiply 
$\bm{\mathcal{G}}(\bm{\mathcal{U}}_\oplus)-\bm{\mathcal{G}} (\bm{\mathcal{U}}_\ominus)$ by $X^{-1}$, integrate for a finite time and multiply the resulting expression by  $X$. We use that
$|X|\leq \mathcal{M}$ and $|X^{-1}|\leq \mathcal{M}$ where $\mathcal{M}$ is given in \eqref{eq:1} and proceed to bound the expression which gives ${\bm{\mathcal{H}}}\{\bm{\mathcal{U}}_\oplus\}-{\bm{\mathcal{H}}}\{\bm{\mathcal{U}}_\ominus\}$ as we did for ${\bm{\mathcal{H}}}\{\bm 0\}$ in \eqref{eq:2}, \eqref{eq:3}, \eqref{eq:4}, \eqref{eq:5}, to obtain
\begin{equation}
\rVert{\bm{\mathcal{H}}}\{\bm{\mathcal{U}}_\oplus\}
-{\bm{\mathcal{H}}}\{\bm{\mathcal{U}}_\ominus\}\lVert \le M_2\mu\xi^6
||\bm{\mathcal{U}}_\oplus-\bm{\mathcal{U}}_\ominus||.
\end{equation}
 This finishes the proof of Lemma \ref{le:LipschitzLema}.

\subsection{Proof of Lemma \ref{le:H0}}
\label{app:Lema5}

In order to compute $\bm{\mathcal{H}}\{\bm{0}\}(\hat{\mathcal{T}}_0^*)$ let us recall its expression:
\[
    \bm{\mathcal{H}}\{\bm{0}\}(\hat{\mathcal{T}}_0^*) = 
    X(\hat{\mathcal{T}}_0^*)\int_0^{\hat{\mathcal{T}}_0^*} X^{-1}(\hat{\mathcal{T}})\bm{\mathcal{G}}(\bm{0})\, d\hat{\mathcal{T}}
     = 
    \mu X(\hat{\mathcal{T}}_0^*)\int_0^{\hat{\mathcal{T}}_0^*} X^{-1}(\hat{\mathcal{T}})\bm{\mathcal{F}}_1(\bm{\mathcal{U}}_0^e(\hat{\mathcal{T}}))\, d\hat{\mathcal{T}}.
\]

From \eqref{eq:F1ordre}, substituting \eqref{eq:solucioEcosKeplerRaro} we have
\begin{equation}\label{eq:F1U0}
    \bm{\mathcal{F}}_1(\mathcal{U}_0^e,\mathcal{V}_0^e)
    =
    \begin{pmatrix}
    0\\
    0\\
    24\sin^5(n\hat{\mathcal{T}}) \cos\theta_0(2\cos^4\theta_0 - 1)\xi^6 + \mathcal{O}(\xi^8)\\
    24\sin^5(n\hat{\mathcal{T}}) \sin\theta_0(2\sin^4\theta_0 - 1)\xi^6 + \mathcal{O}(\xi^8)
    \end{pmatrix},
\end{equation}
Multiplying \eqref{eq:F1U0} by $X^{-1}$ using the expression provided in Lemma \ref{le:X} we obtain:
\[
    X^{-1}\bm{\mathcal{F}}_1(\mathcal{U}_0^e,\mathcal{V}_0^e) =
    \left(
    \begin{matrix}
        -\cfrac{24\sin^6(n\hat{\mathcal{T}})\cos\theta_0(2\cos^4\theta_0 - 1)}{n}\xi^6 + \cfrac{1}{n}\mathcal{O}(\xi^8)\\[1.2ex]
        -\cfrac{24\sin^6(n\hat{\mathcal{T}})\sin\theta_0(2\sin^4\theta_0 -  1)}{n}\xi^6 + \cfrac{1}{n}\mathcal{O}(\xi^8)\\[1.2ex]
        24\cos(n\hat{\mathcal{T}})\sin^5(n\hat{\mathcal{T}})\cos\theta_0(2\cos^4\theta_0 - 1)\xi^6 + \mathcal{O}(\xi^8)\\[1.2ex]
        24\cos(n\hat{\mathcal{T}})\sin^5(n\hat{\mathcal{T}})\sin\theta_0(2\sin^4\theta_0 -  1)\xi^6 + \mathcal{O}(\xi^8)\\
    \end{matrix}
    \right)
    .
\]
Finally, integrating until the time $\hat{\mathcal{T}}_0^*=\pi$ and multiplying by $\mu X$ using the expression of $X$ provided in Lemma \ref{le:X} we have:
\[
    \bm{\mathcal{H}}\{\bm{0}\}(\hat{\mathcal{T}}_0^*) =
    -
    \begin{pmatrix}
        \cfrac{15(-1)^n\mu\pi\cos\theta_0(2\cos^4\theta_0 - 1)}{2n}\xi^6 + \cfrac{\mu}{n}\mathcal{O}(\xi^8)\\[1.5ex]
        \cfrac{15(-1)^n\mu\pi\sin\theta_0(2\sin^4\theta_0 - 1)}{2n}\xi^6 + \cfrac{\mu}{n}\mathcal{O}(\xi^8)\\[1.5ex]
        \cfrac{4(-1)^n\mu\cos\theta_0(2\cos^4\theta_0 - 1)}{n}\xi^6 + \mu\mathcal{O}(\xi^8)\\[1.5ex]
        \cfrac{4(-1)^n\mu\sin\theta_0(2\sin^4\theta_0 - 1)}{n}\xi^6 + \mu\mathcal{O}(\xi^8)
    \end{pmatrix}
    .
\]
This finishes the proof of Lemma \ref{le:H0}.

\section{Value of the auxiliary matrix $A^e$}\label{app:Aij}

The values of the terms $(A^e_{i,j})$ are given by
\begingroup
\allowdisplaybreaks
\begin{align*}
    A^e_{11} &= \cos(n\hat{\mathcal{T}}) 
         - \sin(2\theta_0)\left[ 
         \hat{\mathcal{T}}\cos(n\hat{\mathcal{T}})
         -\sin(n\hat{\mathcal{T}})\frac{1+\sin^2(n\hat{\mathcal{T}})}{n}
         \right]\xi^3\\
         & \hspace{5mm} + \frac{2\sin^2\theta_0}{n}\left[ \hat{\mathcal{T}}\left(1 + 2\cos^2(n\hat{\mathcal{T}})\right) -  \frac{3\cos(n\hat{\mathcal{T}})\sin(n\hat{\mathcal{T}})}{n} \right]\sin(n\hat{\mathcal{T}})\xi^6,\\[4.ex]
    A^e_{12} &= 
         2\left[
         \cos^2\theta_0\hat{\mathcal{T}}\cos(n\hat{\mathcal{T}})
         - \frac{\sin(n\hat{\mathcal{T}})
         \left(\cos^2\theta_0 - \sin^2\theta_0\sin^2(n\hat{\mathcal{T}})
         \right)
         }{n}
         \right]\xi^3\\
         & \hspace{5mm} - \frac{\sin(2\theta_0)}{n}\left[ \hat{\mathcal{T}}\left(1 + 2\cos^2(n\hat{\mathcal{T}})\right) - \frac{3\cos(n\hat{\mathcal{T}})\sin(n\hat{\mathcal{T}})}{n} \right]\sin(n\hat{\mathcal{T}})\xi^6,\\[4.ex]
    A^e_{13} &= 
         \frac{\sin(n\hat{\mathcal{T}})}{n}
         + \frac{2\sin^2\theta_0}{n}\left[\hat{\mathcal{T}} - \frac{\cos(n\hat{\mathcal{T}})\sin(n\hat{\mathcal{T}})}{n}\right]
         \sin(n\hat{\mathcal{T}}) \xi^3,\\[4.ex]
    A^e_{14} &= 
          \frac{\sin(2\theta_0)}{n}\left[\hat{\mathcal{T}} - \frac{\cos(n\hat{\mathcal{T}})\sin(n\hat{\mathcal{T}})}{n}\right]
         \sin(n\hat{\mathcal{T}}) \xi^3,\\[4.ex]
    A^e_{21} &= 
         -2\left[
         \sin^2\theta_0\hat{\mathcal{T}}\cos(n\hat{\mathcal{T}})
         + \frac{\sin(n\hat{\mathcal{T}})
         \left(\sin^2\theta_0 - \cos^2\theta_0\sin^2(n\hat{\mathcal{T}})
         \right)
         }{n}
         \right]\xi^3\\
         & \hspace{5mm} - \frac{\sin(2\theta_0)}{n}\left[ \hat{\mathcal{T}}\left(1 + 2\cos^2(n\hat{\mathcal{T}})\right) - \frac{3\cos(n\hat{\mathcal{T}})\sin(n\hat{\mathcal{T}})}{n} \right]\sin(n\hat{\mathcal{T}})\xi^6,\\[4.ex]
    A^e_{22} &= \cos(n\hat{\mathcal{T}}) 
         + \sin(2\theta_0)\left[ 
         \hat{\mathcal{T}}\cos(n\hat{\mathcal{T}})
         -\sin(n\hat{\mathcal{T}})\frac{1+\sin^2(n\hat{\mathcal{T}})}{n}
         \right]\xi^3\\
         & \hspace{5mm} + \frac{2\cos^2\theta_0}{n}\left[ \hat{\mathcal{T}}\left(1 + 2\cos^2(n\hat{\mathcal{T}})\right) -  \frac{3\cos(n\hat{\mathcal{T}})\sin(n\hat{\mathcal{T}})}{n} \right]\sin(n\hat{\mathcal{T}})\xi^6,\\[4.ex]
    A^e_{23} &= 
          -\frac{2\cos^2\theta_0}{n}\left[\hat{\mathcal{T}} - \frac{\cos(n\hat{\mathcal{T}})\sin(n\hat{\mathcal{T}})}{n}\right]
         \sin(n\hat{\mathcal{T}}) \xi^3,\\[4.ex]
    A^e_{24} &= 
         \frac{\sin(n\hat{\mathcal{T}})}{n}
         - \frac{\sin(2\theta_0)}{n}\left[\hat{\mathcal{T}} - \frac{\cos(n\hat{\mathcal{T}})\sin(n\hat{\mathcal{T}})}{n}\right]
         \sin(n\hat{\mathcal{T}}) \xi^3,\\[4.ex]
    A^e_{31} &= 
         -n\sin(n\hat{\mathcal{T}})
         +\sin(2\theta_0)\sin(n\hat{\mathcal{T}})
         \left(
         n\hat{\mathcal{T}}
         +3\cos(n\hat{\mathcal{T}})\sin(n\hat{\mathcal{T}})
         \right)\xi^3\\
         & \hspace{5mm}
         + 2\bigg[
            \sin^2\theta_0 \hat{\mathcal{T}} \left(5- 8\cos^2(n\hat{\mathcal{T}})\right)\cos(n\hat{\mathcal{T}})\\
            & \hspace{12mm}
            -\frac{\sin(n\hat{\mathcal{T}})}{n}\left(
            \cos^2\theta_0
            \left(8-13\cos^2(n\hat{\mathcal{T}})+2\cos^4(n\hat{\mathcal{T}})\right)
            + 9\cos^2(n\hat{\mathcal{T}}) 
            - 6
            \right)
         \bigg]\xi^6\\
         & \hspace{5mm}
         - \frac{2\sin(2\theta_0)}{n}\left[\hat{\mathcal{T}}\left(1+2\cos^2(n\hat{\mathcal{T}})\right)
         - \frac{3\cos(n\hat{\mathcal{T}})\sin(n\hat{\mathcal{T}})}{n}
         \right]\sin^3(n\hat{\mathcal{T}})\xi^9,\\[4.ex]
    A^e_{32} &= 
         -2\left[
         n\cos^2\theta_0 \hat{\mathcal{T}}
         -\left(1+3\sin^2\theta_0\right)\cos(n\hat{\mathcal{T}})\sin(n\hat{\mathcal{T}})
         \right]\sin(n\hat{\mathcal{T}})\xi^3\\
         & \hspace{5mm}
         + \sin(2\theta_0)\left[
         \hat{\mathcal{T}}
         \left(5-8\cos^2(n\hat{\mathcal{T}})\right)
         -
         \frac{8-13\cos^2(n\hat{\mathcal{T}})+2\cos^4(n\hat{\mathcal{T}})}{n}\sin(n\hat{\mathcal{T}})
         \right]\xi^6\\
         & \hspace{5mm}
         +
         \frac{4\cos^2\theta_0}{n}\left[
         \hat{\mathcal{T}}\left(1+2\cos^2(n\hat{\mathcal{T}})\right)
         -\frac{3\cos(n\hat{\mathcal{T}})\sin(n\hat{\mathcal{T}})}{n}
         \right]\sin^3(n\hat{\mathcal{T}})\xi^9,\\[4.ex]
    A^e_{33} &= 
         \cos(n\hat{\mathcal{T}}) 
         + \sin(2\theta_0)\left[
            \hat{\mathcal{T}}\cos(n\hat{\mathcal{T}}) + \frac{2-3\cos^2(n\hat{\mathcal{T}})}{n}\sin(n\hat{\mathcal{T}})
         \right]\xi^3
         \\ 
         &\hspace{5mm}
         -\frac{4\cos^2\theta_0}{n}\left[\hat{\mathcal{T}} - \frac{\cos(n\hat{\mathcal{T}})\sin(n\hat{\mathcal{T}})}{n}\right]\sin^3(n\hat{\mathcal{T}})\xi^6,\\[4.ex]
    A^e_{34} &= 
         2\left[
         \sin^2\theta_0\left(
         \hat{\mathcal{T}} - \frac{\cos(n\hat{\mathcal{T}})\sin(n\hat{\mathcal{T}})}{n}
         \right)\cos(n\hat{\mathcal{T}})
         +\frac{2\sin^2\theta_0+1}{n}\sin^3(n\hat{\mathcal{T}})
         \right]\xi^3\\
         &\hspace{5mm}
         - \frac{2\sin(2\theta_0)}{n}\left[
         \hat{\mathcal{T}} - \frac{\cos(n\hat{\mathcal{T}})\sin(n\hat{\mathcal{T}})}{n}
         \right]\sin^3(n\hat{\mathcal{T}})\xi^6,\\[4.ex]
    A^e_{41} &= 
         2\left[
         n\sin^2\theta_0\hat{\mathcal{T}}
         -(1+3\cos^2\theta_0)\cos(n\hat{\mathcal{T}})\sin(n)\hat{\mathcal{T}}
         \right]\sin(n\hat{\mathcal{T}})\xi^3\\
         & \hspace{5mm}
         + \sin(2\theta_0)\left[
         \hat{\mathcal{T}}
         \left(5-8\cos^2(n\hat{\mathcal{T}})\right)
         -
         \frac{8-13\cos^2(n\hat{\mathcal{T}})+2\cos^4(n\hat{\mathcal{T}})}{n}\sin(n\hat{\mathcal{T}})
         \right]\xi^6\\
         & \hspace{5mm}
         -
         \frac{4\sin^2\theta_0}{n}\left[
         \hat{\mathcal{T}}\left(1+2\cos^2(n\hat{\mathcal{T}})\right)
         -\frac{3\cos(n\hat{\mathcal{T}})\sin(n\hat{\mathcal{T}})}{n}
         \right]\sin^3(n\hat{\mathcal{T}})\xi^9,\\[4.ex]
    A^e_{42} &= 
         -n\sin(n\hat{\mathcal{T}})
         - \sin(2\theta_0)\sin(n\hat{\mathcal{T}})
         \left(
         n\hat{\mathcal{T}}
         +3\cos(n\hat{\mathcal{T}})\sin(n\hat{\mathcal{T}})
         \right)\xi^3\\
         & \hspace{5mm}
         + 2\bigg[
            -\cos^2\theta_0 \hat{\mathcal{T}} \left(5- 8\cos^2(n\hat{\mathcal{T}})\right)\cos(n\hat{\mathcal{T}})\\
            & \hspace{12mm}
            -\frac{\sin(n\hat{\mathcal{T}})}{n}\left(
            \sin^2\theta_0
            \left(8-13\cos^2(n\hat{\mathcal{T}})+2\cos^4(n\hat{\mathcal{T}})\right)
            + 9\cos^2(n\hat{\mathcal{T}}) 
            - 6
            \right)
         \bigg]\xi^6\\
         & \hspace{5mm}
         + \frac{2\sin(2\theta_0)}{n}\left[\hat{\mathcal{T}}\left(1+2\cos^2(n\hat{\mathcal{T}})\right)
         - \frac{3\cos(n\hat{\mathcal{T}})\sin(n\hat{\mathcal{T}})}{n}
         \right]\sin^3(n\hat{\mathcal{T}})\xi^9,\\[4.ex]
    A^e_{43} &= 
         -2\left[
         \cos^2\theta_0\left(
         \hat{\mathcal{T}} - \frac{\cos(n\hat{\mathcal{T}})\sin(n\hat{\mathcal{T}})}{n}
         \right)\cos(n\hat{\mathcal{T}})
         +\frac{2\cos^2\theta_0+1}{n}\sin^3(n\hat{\mathcal{T}})
         \right]\xi^3\\
         &\hspace{5mm}
         - \frac{2\sin(2\theta_0)}{n}\left[
         \hat{\mathcal{T}} - \frac{\cos(n\hat{\mathcal{T}})\sin(n\hat{\mathcal{T}})}{n}
         \right]\sin^3(n\hat{\mathcal{T}})\xi^6,\\[4.ex]
    A^e_{44} &= 
         \cos(n\hat{\mathcal{T}}) 
         - \sin(2\theta_0)\left[
            \hat{\mathcal{T}}\cos(n\hat{\mathcal{T}}) + \frac{2-3\cos^2(n\hat{\mathcal{T}})}{n}\sin(n\hat{\mathcal{T}})
         \right]\xi^3
         \\ 
         &\hspace{5mm}
         -\frac{4\sin^2\theta_0}{n}\left[\hat{\mathcal{T}} - \frac{\cos(n\hat{\mathcal{T}})\sin(n\hat{\mathcal{T}})}{n}\right]\sin^3(n\hat{\mathcal{T}})\xi^6.
\end{align*}
\endgroup

\end{document}